\documentclass[a4, reqno, 11pt]{amsart}
\usepackage[top=30truemm,bottom=30truemm,left=20truemm,right=20truemm]{geometry}
\usepackage{amsmath,amssymb,amsthm,enumerate}
\usepackage{mathrsfs} 
\usepackage[all]{xy}
\usepackage{eucal} 
\usepackage{stmaryrd} 
\SetSymbolFont{stmry}{bold}{U}{stmry}{m}{n} 
\usepackage{mathtools}
\usepackage{leftindex} 
\allowdisplaybreaks[4] 
\usepackage{enumitem}
\usepackage{graphicx} 
\usepackage{float}
\usepackage{array}
\usepackage{tikz}
\usepackage{tikz-cd}
\tikzcdset{scale cd/.style={every label/.append style={scale=#1}, 
cells={nodes={scale=#1}}}} 
\usepackage{quiver}
\usepackage{bm}
\usepackage[noadjust]{cite} 
\usepackage{indentfirst} 
\usepackage[
setpagesize=false,bookmarks=true,bookmarksnumbered=true,
hidelinks,
colorlinks=false,
]{hyperref} 

\makeatletter
\@namedef{subjclassname@2020}{%
    \textup{2020} Mathematics Subject Classification}
\makeatother

\theoremstyle{plain}
\newtheorem{thm}{Theorem}[section]
\AtBeginEnvironment{thm}{\setlist[enumerate]{label={\upshape{}(\arabic*)}}} 

\newtheorem{mainthm}{Theorem}

\AtBeginEnvironment{mainthm}{\setlist[enumerate]{label={\upshape{}(\arabic*)}}}

\newtheorem{lem}[thm]{Lemma}
\AtBeginEnvironment{lem}{\setlist[enumerate]{label={\upshape(\arabic*)}}}

\newtheorem{prop}[thm]{Proposition}
\AtBeginEnvironment{prop}{\setlist[enumerate]{label={\upshape(\arabic*)}}}

\newtheorem{cor}[thm]{Corollary}
\AtBeginEnvironment{cor}{\setlist[enumerate]{label={\upshape(\arabic*)}}}

\newtheorem*{prob}{Problem 22}

\theoremstyle{definition}

\newtheorem{ex}[thm]{Example}

\newtheorem{defn}[thm]{Definition}

\newtheorem{rem}[thm]{Remark}

\newtheorem{nota}[thm]{Notation}

\numberwithin{equation}{section}

\newcommand{\papervar}[8]{\bibitem{#1}#2, \textit{#3}, #4\ \textbf{#5} (#6), no.~#7, #8.}
\newcommand{\paper}[9]{\papervar{#1}{#2}{#3}{#4}{#5}{#6}{#7}{#8--#9}}
\newcommand{\paperv}[8]{\bibitem{#1}#2, \textit{#3}, #4\ \textbf{#5} (#6), #7--#8.} 
\newcommand{\papervv}[7]{\bibitem{#1}#2, \textit{#3}, #4\ \textbf{#5} (#6), #7.}
\newcommand{\book}[6]{\bibitem{#1}#2, #3, #4, #5, #6.}
\newcommand{\bookv}[7]{\bibitem{#1}#2, #3, #4 edition, #5, #6.}

\newcommand{\ABSO}{augmented braided simplicial object}
\newcommand{\ABsSO}{augmented braided semisimplicial object}
\newcommand{\ABbsbSO}{augmented braided (semi)simplicial object}
\newcommand{\ABcosSO}{augmented braided cosemisimplicial object}
\newcommand{\Drinfeld}{Drinfel'd}
\newcommand{\maclane}{Mac Lane}

\newcommand{\sm}[4]{\begin{smallmatrix} #1 & #2 \\ #3 & #4 \end{smallmatrix}}
\newcommand{\gbf}[1]{{\mathbf{g}}^{#1}}
\DeclareMathOperator{\Ker}{Ker}
\DeclareMathOperator{\rank}{rank}

\newcommand{\ktag}[2]{\href{https://kerodon.net/tag/#2}{#1 #2}}

\NewDocumentCommand\term{o m}{\textit{#2}}

\newcommand{\eps}{\varepsilon}
\newcommand{\myspace}{\hspace{0.05em}}
\newcommand{\minusmyspace}{\hspace{-0.05em}}
\newcommand{\longxrightarrow}[1]{\overset{#1}{\longrightarrow}}
\newcommand{\family}[2]{\{#1\}_{#2}}
\newcommand{\famnat}[1]{\family{#1}{n \ge 0}}
\newcommand{\wt}[1]{\widetilde{#1}}
\newcommand{\cpxdot}{\bullet}
\newcommand{\incl}{\iota}

\NewDocumentCommand\func{o m m o o}{\IfNoValueF{#1}{#1 \colon} {#2} \to {#3}\IfNoValueF{#4}{; #4 \mapsto }\IfNoValueF{#5}{#5}}
\NewDocumentCommand\funcgene{o m m o o}{\func[#1]{\gene{#2}}{\gene{#3}}[#4][#5]}
\renewcommand{\subset}{\subseteq}

\newcommand{\cat}[1]{\mathcal{#1}}
\NewDocumentCommand\op{o m}{#2^{\IfNoValueF{#1}{\hspace{#1em}} \mathrm{op}}}
\renewcommand{\hom}[3]{\mathrm{Hom}_{{\myspace} #1} (#2, #3)}
\NewDocumentCommand\fct{o m m o o}{\func[#1]{\cat{#2}}{\cat{#3}}[#4][#5]}
\newcommand{\id}[1]{\mathrm{id}_{#1}}
\newcommand{\iso}{\cong}
\newcommand{\Aut}[2]{\mathrm{Aut}_{{\myspace} #1} (#2)}
\newcommand{\Fun}[2]{[\cat{#1}, \cat{#2}]}
\NewDocumentCommand\opvar{o m}{#2_{\IfNoValueF{#1}{\hspace{#1em}} \mathrm{op}}}
\newcommand{\opcirc}{\mathbin{\opvar{\circ}}}
\newcommand{\opC}{\op[0.07]{\cat{C}}}
\NewDocumentCommand\colim{o m}{\varinjlim_{{\myspace} #1} {\hspace{-0.1em} #2}}
\NewDocumentCommand\dcolim{o m}{\displaystyle\varinjlim_{#1} {\hspace{-0.1em} #2}}

\newcommand{\ot}{\varotimes}
\newcommand{\ordot}{\mathord{\ot}}
\newcommand{\asscvar}{a}
\newcommand{\assc}[3]{a_{#1, #2, #3}}
\newcommand{\asscv}[3]{a'_{#1, #2, #3}}
\newcommand{\unit}{I}
\newcommand{\lucvar}{l}
\newcommand{\luc}[1]{l_{#1}}
\newcommand{\lucv}[1]{l'_{#1}}
\newcommand{\rucvar}{r}
\newcommand{\ruc}[1]{r_{#1}}
\newcommand{\rucv}[1]{r'_{#1}}
\newcommand{\asscvinv}[3]{a_{#1, #2, #3}^{\prime {\myspace} -1}}
\newcommand{\rucvinv}[1]{r_{#1}^{\prime {\myspace} -1}}
\newcommand{\moncat}[1]{(\cat{#1}, \ordot, \unit)}
\newcommand{\monfctuni}[1]{\varphi_0}
\newcommand{\monfctvuni}[1]{\monfctuni{#1}'}
\newcommand{\monfctmulvar}[1]{\varphi_2}
\newcommand{\monfctmul}[3]{(\monfctmulvar{#1})_{#2, #3}}
\newcommand{\monfctvmulvar}[1]{\monfctmulvar{#1}'}
\newcommand{\monfctvmul}[3]{(\monfctvmulvar{#1})_{#2, #3}}
\newcommand{\monfct}[1]{(#1, \monfctmulvar{#1}, \monfctuni{#1})}
\newcommand{\monfctvar}[3]{\fct[\monfct{#1}]{#2}{#3}}
\newcommand{\monfctv}[1]{(#1, \monfctvmulvar{#1}, \monfctvuni{#1})}
\newcommand{\br}{c}
\newcommand{\braid}[2]{{\br}_{#1, #2}}
\newcommand{\brmoncat}[1]{(\cat{#1}, \br)}
\newcommand{\braidv}[2]{{\br'}_{#1, #2}}
\newcommand{\brmoncatv}[1]{(\cat{#1}, \br')}
\newcommand{\brmonfct}[3]{\func[\monfct{#1}]{\brmoncat{#2}}{\brmoncatv{#3}}}

\newcommand{\bicat}[1]{\mathsf{#1}}
\newcommand{\deloop}[1]{\bicat{B} \hspace{0.1em} #1}
\newcommand{\opIvar}{\deloop{\hspace{-0.1em}(\opI)}}
\newcommand{\biCAT}{\bicat{CAT}}

\NewDocumentCommand\brgrp{o}{B_{\IfValueTF{#1}{#1}{n + 1}}}
\newcommand{\brperm}[2]{g_{#1 \adjustsp}^{#2 \vadjustsp}}
\newcommand{\Symmgrp}{S} 
\NewDocumentCommand\symmgrp{o}{{\Symmgrp}_{\IfValueTF{#1}{#1}{n + 1}}}
\newcommand{\simpcat}{\bm{\Delta}}
\newcommand{\csg}[1]{\simpcat \mathbf{#1}}
\newcommand{\augcsg}[1]{(\csg{#1})_+}
\newcommand{\augbr}{\augcsg{B}}
\newcommand{\symmcat}{\csg{\Symmgrp}}
\newcommand{\augsymm}{\augcsg{\Symmgrp}}
\newcommand{\augbrsemi}{(\simpcat_{\hspace{0.03em} {\textnormal{inj}}} \mathbf{B})_+}
\newcommand{\augsymmsemi}{(\simpcat_{\hspace{0.03em} {\textnormal{inj}}} \mathbf{\Symmgrp})_+}

\newcommand{\coalg}[1]{(#1, \Delta, \eps)}
\newcommand{\coalgv}[1]{(#1, \Delta', \eps')}

\newcommand{\End}[2]{{\mathrm{End}}_{{\myspace} #1} (#2)}
\NewDocumentCommand\linhom{m m m}{{\underline{\mathrm{Hom}}}_{{\myspace} \cat{#1}} (#2, #3)}
\NewDocumentCommand\K{o m o}{\mathrm{K}^{\Ifcpxdot{#1}}(#2, \IfValueTF{#3}{#3}{M})}
\NewDocumentCommand\KC{o}{\K[#1]{C, \br}}
\NewDocumentCommand\KD{o}{\K[#1]{D, \br'}[F(M)]}
\NewDocumentCommand\Kcpx{o}{K^{\Ifcpxdot{#1} + 1}}
\NewDocumentCommand\Lcpx{o}{L^{\Ifcpxdot{#1} + 1}}
\NewDocumentCommand\fcpx{o}{f^{\hspace{0.03em} \Ifcpxdot{#1} + 1}}
\newcommand{\fn}{f^{\hspace{0.03em} n}}
\newcommand{\cosf}{\cosobj[0.03]{f}}
\newcommand{\Ifcpxdot}[1]{\IfValueTF{#1}{#1}{\cpxdot}}
\NewDocumentCommand\KBvar{o m}{\mathrm{KB}^{\Ifcpxdot{#1}}(#2)}
\NewDocumentCommand\KB{o m o}{\KBvar[#1]{#2, \IfValueTF{#3}{#3}{M}}}
\NewDocumentCommand\KBA{o}{\KB[#1]{A, R}}
\NewDocumentCommand\KBC{o}{\KB[#1]{C, \br}}
\NewDocumentCommand\KBD{o}{\KB[#1]{D, \br'}[F(M)]}
\NewDocumentCommand\KSvar{o m}{\mathrm{K{\Symmgrp}}^{\Ifcpxdot{#1}}(#2)}
\NewDocumentCommand\KS{o m o}{\KSvar[#1]{#2, \IfValueTF{#3}{#3}{M}}}
\NewDocumentCommand\KSC{o}{\KS[#1]{C, \br}}
\NewDocumentCommand\HB{o m o}{\mathrm{HB}^{\Ifcpxdot{#1}}(#2, \IfValueTF{#3}{#3}{M})}
\NewDocumentCommand\HBC{o}{\HB[#1]{C, \br}}
\NewDocumentCommand\HBD{o}{\HB[#1]{D, \br'}[F(M)]}
\NewDocumentCommand\HS{o m o}{\mathrm{H{\Symmgrp}}^{\Ifcpxdot{#1}}(#2, \IfValueTF{#3}{#3}{M})}
\NewDocumentCommand\HSC{o}{\HS[#1]{C, \br}}

\NewDocumentCommand\sgn{o}{\IfValueTF{#1}{\mathrm{sgn}(#1)}{\mathrm{sgn}}}
\newcommand{\Xcal}{\mathcal{X}}
\newcommand{\Ycal}{\mathcal{Y}}
\newcommand{\sobjcalinj}[1]{\sobj{\mathcal{#1}'}}
\newcommand{\opI}{\op[0.05]{\cat{I}}}
\newcommand{\Seta}{S_\eta}
\newcommand{\Teta}[2]{T_\eta(#1, #2)}
\newcommand{\Ieta}{\cat{I}'_\eta}
\newcommand{\braidtrivar}{{\br}^R}
\newcommand{\braidtri}[2]{{\br}^R_{#1, #2}}
\newcommand{\braidtrivarv}{{\br}^{R'}}
\newcommand{\braidtriv}[2]{{\br}^{R'}_{#1, #2}}

\NewDocumentCommand\ff{o}{U_{\IfNoValueF{#1}{#1}}}
\newcommand{\bialg}[1]{\coalg{#1}}
\newcommand{\qtribialg}[1]{(#1, \Delta, \eps, R)}
\newcommand{\qtribialgv}[1]{(#1, \Delta', \eps', R')}
\newcommand{\qtb}[1]{(#1, R)}
\newcommand{\qtbv}[1]{(#1, R')}

\newcommand{\Mod}[1]{{\leftindex_{#1} {\mathcal{M}}}}
\newcommand{\AMod}{\Mod{A}}
\newcommand{\BMod}{\Mod{B}}
\newcommand{\vMod}[1]{\hspace{-0.2em} \Mod{#1}} 
\newcommand{\modvar}[1]{\Mod{#1}^{{\myspace} \textnormal{fin}}}
\newcommand{\Amod}{\modvar{A}}
\newcommand{\Bmod}{\modvar{B}}
\newcommand{\sw}[2]{#1^{(#2)}} 
\newcommand{\swvar}[1]{R_{#1}} 
\newcommand{\swvarL}[1]{(R_\lambda)_{#1}}

\newcommand{\gene}[1]{\langle #1 \rangle}
\newcommand{\simpdot}{\bullet}
\NewDocumentCommand\Ifsimpdot{o}{\IfValueTF{#1}{#1}{\simpdot}}
\NewDocumentCommand\sobj{m}{#1_{\simpdot}}
\NewDocumentCommand\cosobj{o m}{#2^{\IfNoValueF{#1}{\hspace{#1em}} \simpdot}}
\NewDocumentCommand\adjustsp{o}{\vphantom{i j + 1 #1}}
\NewDocumentCommand\vadjustsp{o}{\vphantom{n #1}}
\newcommand{\coface}[2]{\delta_{#1 \adjustsp[m]}^{{\myspace} #2 \vadjustsp[m]}}
\newcommand{\cofacevar}{\coface{0}{0}}
\NewDocumentCommand\codiffvar{m o}{\coface{#1}{\vphantom{\prime} \Ifcpxdot{#2}}}
\NewDocumentCommand\codiff{o}{\codiffvar{}[#1]}
\newcommand{\face}[2]{d^{\hspace{0.1em} #1 \adjustsp}_{#2 \vadjustsp}}
\newcommand{\facei}[2]{d^{\hspace{0.07em} #1 \adjustsp}_{#2 \vadjustsp}}
\newcommand{\codege}[2]{\sigma_{#1 \adjustsp[m]}^{#2 \vadjustsp[m]}}
\newcommand{\codegevar}{\codege{0}{0}}
\newcommand{\dege}[2]{s^{\hspace{0.1em} #1 \adjustsp}_{#2 \vadjustsp}}
\newcommand{\coperm}[2]{\tau_{#1 \adjustsp}^{#2 \vadjustsp}}
\newcommand{\copermvar}{\coperm{1}{1}}
\newcommand{\perm}[2]{t^{{\myspace} #1 \adjustsp}_{#2 \vadjustsp}}
\newcommand{\permj}[2]{t^{\hspace{0.1em} #1 \adjustsp}_{#2 \vadjustsp}}

\newcommand{\etavar}{\eta^X_{[X', Y]}}
\NewDocumentCommand\evvar{m o}{\epsilon^{#1}_{\IfNoValueF{#2}{#2}}}
\NewDocumentCommand\ev{o m}{\mathrm{ev}_{\IfNoValueF{#1}{\hspace{#1em}} #2}}
\newcommand{\evunit}{\ev{\unit}}

\newcommand{\WT}{\wt{\mathrm{T}}}
\NewDocumentCommand\wT{o}{{\WT}_{\Ifsimpdot[#1]}}
\NewDocumentCommand\wTC{o}{\wT[#1]\brobj{C}}

\newcommand{\WTCAL}{\wt{\mathcal{T}}}
\NewDocumentCommand\wTcal{o}{{\WTCAL}_{\Ifsimpdot[#1]}}
\NewDocumentCommand\wTcalC{o}{\wTcal[#1]\brobj{C}}
\NewDocumentCommand\wTcalD{o}{\wTcal[#1]\brobjv{D}}

\newcommand{\sobjcal}[1]{\sobj{\mathcal{#1}}}
\NewDocumentCommand{\vareta}{o}{\eta_{\IfValueTF{#1}{#1 \adjustsp}{\simpdot}}^{\vphantom{\prime \theta}}}
\NewDocumentCommand{\varetaVV}{o}{\eta_{\IfValueTF{#1}{#1 \adjustsp}{\simpdot}}^{\prime \vphantom{\theta}}}
\newcommand{\varetazero}{\vareta[{\myspace} 0]}
\newcommand{\varetaVVzero}{\varetaVV[{\myspace} 0]}

\newcommand{\brobj}[1]{(#1, \br)}
\newcommand{\brobjv}[1]{(#1, \br')}
\newcommand{\brcoalg}[1]{(#1, \Delta, \eps, \br)}
\newcommand{\brcoalgv}[1]{(#1, \Delta', \eps', \br')}

\NewDocumentCommand\otfold{o m}{\IfValueTF{#1}{#1}{C}^{\ot {\myspace} #2 \vphantom{i j m n}}}
\NewDocumentCommand\otfoldv{o m}{\IfValueTF{#1}{#1}{C}^{\ot \hspace{0.03em} (#2) \vphantom{i j m n - + 1 2}}}

\title[Braided cohomology of quasi-triangular bialgebras and braided Morita invariance]{Braided cohomology of quasi-triangular bialgebras\\ and braided Morita invariance}
\author{Shota Inoue}
\address{Department of Mathematics, 
Graduate School of Science, 
Tokyo University of Science, 
1-3 Kagurazaka, Shinjuku-ku, Tokyo, 162-8601, JAPAN}
\email{1125701@ed.tus.ac.jp}

\author{Ayako Itaba}
\address{Katsushika Division, 
Institute of Arts and Sciences, 
Tokyo University of Science, 
6-3-1 Niijuku, Katsushika-ku, Tokyo, 125-8585, JAPAN}
\email{itaba@rs.tus.ac.jp}
\date{\today}
\begin{document}
\begin{abstract}
    In this paper, we introduce the braided cochain complex and the braided cohomology of braided coalgebras in linear monoidal categories, and compare those of braided coalgebras living in different linear monoidal categories using relative morphisms.
    The symmetric cohomology was introduced for groups by Staic, and was generalized to cocommutative Hopf algebras by Shiba, Sanada, and the second author.
    This notion involves degreewise actions of the symmetric groups on a cochain complex, which come from the symmetry on the monoidal category of group representations.
    We generalize this framework by dealing with arbitrary linear monoidal categories, and by replacing symmetries with braidings defined merely on objects.
    We first describe relative morphisms in this general setting, and apply this result to prove that the braided cochain complex of quasi-triangular bialgebras is a braided Morita invariant under some hypotheses. 
\end{abstract}

\maketitle

\section{Introduction}

Staic \cite{staic-A} defined the notion of the symmetric cohomology of groups, motivated by lattice topological field theory.
Analogous to the situations for the classical cohomology of groups, Staic \cite{staic-B} and Pirashvili \cite{pirashvili-B} interpreted the second and third symmetric cohomology of groups in terms of group extensions and of crossed extensions, respectively.
Pirashvili \cite{pirashvili-A} related the symmetric cohomology of groups with the exterior cohomology of groups, introduced by Zarelua \cite{zarelua}, by means of a spectral sequence.
We refer the reader to Bardakov--Neshchadim--Singh \cite{bardakov-neshchadim-singh} and Pirashvili--Pirashvili \cite{pirashvili-pirashvili} for related work.
One common way to obtain a reasonable generalization of concepts in the theory of groups is to formulate them in the language of cocommutative Hopf algebras, by dealing with group algebras.
Shiba--Sanada--Itaba \cite{shiba-sanada-itaba} generalized the symmetric cohomology from this point of view.
They constructed from cocommutative Hopf algebras two isomorphic cochain complexes, called the non-homogeneous and homogeneous cochain complexes, whose cohomology is (isomorphic to) the symmetric cohomology.
In the same work, they also defined the symmetric Hochschild cohomology of cocommutative Hopf algebras, and established an Eckmann-Shapiro type isomorphism (see Pevtsova--Witherspoon \cite[Lemma 12]{pevtsova-witherspoon} for a classical statement).

The category of modules over a bialgebra has the structure of a linear monoidal category.
Parallel to the classical Morita theory for algebras, bialgebras are often studied and are classified up to monoidal Morita equivalence; that is, up to linear monoidal equivalence of their module categories, where the role of progenerators is instead played by bialgebra twists.
Ng--Schauenburg \cite{ng-schauenburg} proved that the (higher) Frobenius-Schur indicators of complex semisimple quasi-Hopf algebras are stable under bialgebra twists.
Meanwhile, braidings on the monoidal category of modules over a bialgebra correspond to R-matrices, and a bialgebra equipped with an R-matrix is called a quasi-triangular bialgebra.
Quasi-triangular bialgebras can be viewed as generalizations of cocommutative bialgebras, and they are studied up to braided Morita equivalence; that is, up to linear braided monoidal equivalence of their module categories.
Shimizu \cite{shimizu} proved the braided Morita invariance of some representations of the braid groups (of Coxeter type A) associated to finite-dimensional quasi-triangular Hopf algebras.
This was extended by M\"uller--Walton \cite{muller-walton} to the setting of quasi-triangular coideal subalgebras of finite-dimensional quasi-triangular Hopf algebras, using representations of the braid groups of Coxeter types BC or D instead.

The construction of the (non-)homogeneous cochain complex, mentioned above, involves degreewise actions of the symmetric groups, which are induced from the symmetry on the monoidal category of modules over cocommutative Hopf algebras.
It is natural to extend this notion to the setting of quasi-triangular bialgebras, by replacing symmetric groups with braid groups, and the symmetry with arbitrary braidings.
We take a step further, and define the braided cochain complex and the braided cohomology, which take as input braided coalgebras (introduced by Ardizzoni--Menini \cite{ardizzoni-menini}) in linear monoidal categories.
These notions generalize the homogeneous cochain complex and the symmetric cohomology defined in previous work.
Working in this generality allows us to compute the braided cohomology of important examples of quasi-triangular bialgebras such as the Sweedler Hopf algebra and the {\Drinfeld} doubles of finite-dimensional Hopf algebras.
Meanwhile, this framework enables us to study braided coalgebras in the filtered or (differential) graded setting, or those in the Yetter-{\Drinfeld} category, for example.

Our constructions of the braided cochain complex and of the braided cohomology involve (monoidal) functors from the augmented braided category $\augbr$ and its variant $\augbrsemi$.
The augmented braided category is an example of (the augmented version of) crossed simplicial groups, introduced independently by Krasauskas \cite{krasauskas} and Fiedorowicz--Loday \cite{fiedorowicz-loday}.
We refer to \cite[Section 6]{fiedorowicz-loday} for a unified treatment of the homology theories associated to crossed simplicial groups, and to Angelini-Knoll--Merling--P\'eroux \cite{angelini-knoll-merling-peroux} for their $\infty$-categorical refinement called the (positive) topological homology.
Our work reveals a connection between the (augmented) braided category and the braided cohomology of braided coalgebras, which gives an answer to Problem 22 by Singh in Bardakov--Gongopadhyay--Singh--Vesnin--Wu \cite{bardakov-gongopadhyay-singh-vesnin-wu}, as follows:

\begin{prob}[{\cite{bardakov-gongopadhyay-singh-vesnin-wu}}]
    Fiedorowicz and Loday \emph{\cite{fiedorowicz-loday}} introduced crossed simplicial groups and a cohomology theory of these objects, which is similar to symmetric cohomology \emph{\cite{staic-A}}.
    It would be interesting to explore some connection between the two.
\end{prob}

Throughout this paper, we fix a commutative ring $\Bbbk$.
Let $\cat{C}$ be a $\Bbbk$-linear monoidal category, $\brobj{C}$ a braided coalgebra in $\cat{C}$ (see Definition~\ref{defn:brcoalg}), and $M$ an object in $\cat{C}$.
We associate to $\brobj{C}$ a strict monoidal functor $\func[\wTcalC]{\op{\augbr}}{\Fun{C}{C}}$, where $\Fun{C}{C}$ is the (large) strict monoidal category of endofunctors of $\cat{C}$.
This functor plays a crucial role in the construction of the braided cochain complex $\KBC$ of $\brobj{C}$, and also in the study of relative morphisms (see Definition~\ref{defn:rel-mor}).
Our first main result describes some relative morphisms between strict monoidal functors of the form $\wTcalC$, in terms of the structural morphisms of braided coalgebras $\brobj{C}$.

\begin{mainthm}[Theorem~\ref{thm:mor-4}]
    \label{mainthm:A}
    Let $\cat{C}$ and $\cat{D}$ be monoidal categories, $\brcoalg{C}$ and $\brcoalgv{D}$ braided coalgebras in $\cat{C}$ and $\cat{D}$, respectively, and $\monfctvar{F}{C}{D}$ a monoidal functor.
    Then, the assignment
    \begin{equation*}
		\vareta \mapsto \rucv{F(C)} \circ (F(C) \ot \monfctuni{F}^{-1}) \circ \monfctmul{F}{C}{\unit}^{-1} \circ (\vareta[1])_{\unit} \circ (D \ot \monfctuni{F}) \circ \rucvinv{D}
	\end{equation*}
	induces a bijective correspondence between
    \begin{itemize}
        \item collections $\vareta = \famnat{\func[\vareta[n]]{\wTcalD[n] \circ F}{F \circ \wTcalC[n]}}$ of natural transformations such that the pair $(F, \vareta)$ is a relative morphism from $\wTcalC$ to $\wTcalD$, and $\vareta[1] = \monfctmul{F}{C}{-} \circ (\alpha' \ot F(-))$ for some morphism $\func[\alpha']{D}{F(C)}$ in $\cat{D}$, and
        \item morphisms $\func[\alpha]{D}{F(C)}$ in $\cat{D}$ which satisfy the following equalities\emph{:}
        \begin{align*}
            \monfctuni{F} \circ \eps' = F(\eps) \circ \alpha; \qquad 
            \monfctmul{F}{C}{C} \circ (\alpha \ot \alpha) \circ \Delta' &= F(\Delta) \circ \alpha; \\
            \monfctmul{F}{C}{C} \circ (\alpha \ot \alpha) \circ \br' &= F(\br) \circ \monfctmul{F}{C}{C} \circ (\alpha \ot \alpha).
        \end{align*}
    \end{itemize}
\end{mainthm}

We use this result to illustrate the utility of the braided cochain complex and the braided cohomology as invariants for studying quasi-triangular $\Bbbk$-bialgebras.
Since every quasi-triangular $\Bbbk$-bialgebra $\qtb{A}$ gives rise to a braided coalgebra $(A, \braidtri{A}{A})$ in the $\Bbbk$-linear monoidal category $\AMod$ of $A$-modules, where $\braidtrivar$ denotes the corresponding braiding on $\AMod$, we can consider the braided cochain complex of $(A, \braidtri{A}{A})$ with coefficients in an $A$-module $M$, which is denoted by $\KBA$.
We let $\ff[A]$ denote the forgetful functor from $\AMod$ to the category of $\Bbbk$-modules.
Our next main result asserts that, under mild hypotheses, the braided cochain complexes of quasi-triangular $\Bbbk$-bialgebras are determined by their braided Morita equivalence classes up to isomorphism.

\begin{mainthm}[Theorem~\ref{thm:Morita}]
    \label{mainthm:B}
    Let $\qtb{A}$ and $\qtbv{B}$ be quasi-triangular $\Bbbk$-bialgebras, and suppose that $\func[F]{(\AMod, \braidtrivar)}{(\BMod, \braidtrivarv)}$ is a $\Bbbk$-linear braided monoidal equivalence such that $\ff[B] \circ F \iso \ff[A]$ as $\Bbbk$-linear functors.
    Then, there exists an isomorphism $\KBA \iso \KB{B, R'}[F(M)]$ for all $A$-modules $M$.
    In particular, the braided cohomology of quasi-triangular $\Bbbk$-bialgebras is invariant under braided Morita equivalences which commute with the forgetful functors up to $\Bbbk$-linear natural isomorphism.
\end{mainthm}

We note that the condition on the $\Bbbk$-linear braided monoidal equivalence is 
automatically satisfied in the finite-dimensional case, which we record as Corollary~\ref{cor:Morita}.
Finally, we give a complete classification of the braided Morita equivalence classes of quasi-triangular $\Bbbk$-bialgebras of the form $(E(2), R_X)$, where $E(2)$ is the $8$-dimensional unimodular ribbon Hopf algebra defined in Radford \cite{radford-B}, whose R-matrices are parameterized by $2 \times 2$ matrices $X$ with values in $\Bbbk$ (see Proposition~\ref{prop:E2-2} for a precise statement).

This paper is organized as follows.
In Section~\ref{sec:2}, we review some concepts from the theory of (braided) monoidal categories and of $\Bbbk$-linear categories, and introduce the augmented braided category and its variant.
In Section~\ref{sec:3}, we construct strict monoidal functors from braided coalgebras in monoidal categories (see Proposition~\ref{prop:nonstr-augbrobj}), and in Subsection~\ref{subsec:3.3}, we introduce the braided cochain complex and the braided cohomology of braided coalgebras in $\Bbbk$-linear monoidal categories.
In Section~\ref{sec:4}, we define and study the notion of relative morphisms (see Definition~\ref{defn:rel-mor}), whose explicit description is given in Theorem~\ref{thm:mor-4}.
In Section~\ref{sec:5}, we apply this to prove results on the braided Morita invariance of the braided cochain complex of quasi-triangular $\Bbbk$-bialgebras in Theorem~\ref{thm:Morita} and Corollary~\ref{cor:Morita}, and in Subsection~\ref{subsec:5.2}, we compute the dimension of the first braided cochain complex of the $8$-dimensional Hopf algebra $E(2)$.

We refer to locally small $1$-categories 
as categories, and covariant functors as functors.
Given a category $\cat{C}$, we let $\opC$ denote the opposite category of $\cat{C}$, which is small if $\cat{C}$ is.
For categories $\cat{C}$ and $\cat{D}$, we let $\Fun{C}{D}$ denote the (large) category of functors from $\cat{C}$ to $\cat{D}$.

\section{Preliminaries}
\label{sec:2}

In this section, we recall some basic definitions and preliminary results from the theory of (braided) monoidal categories; we refer to Kassel \cite[Chapters XI and XII]{kassel} for details.
Then, we introduce two (small) strict monoidal categories, the augmented braided category as well as its variant (see Definitions~\ref{defn:augbr} and \ref{defn:augbrsemi}), which play a significant role in this paper.
We also collect elementary facts on evaluations and on $\Bbbk$-linear categories, which will be used in an argument after Definition~\ref{defn:lin-mon-eqv}.

\subsection{(Braided) monoidal categories and (braided) monoidal functors}

We recall that a \term{monoidal category} is a triple $\moncat{C}$ consisting of a category $\cat{C}$, a functor $\func[\ordot]{\cat{C} \times \cat{C}}{\cat{C}}$, an object $\unit$ in $\cat{C}$, equipped with fixed collections $\asscvar = \family{\func[\assc{X}{Y}{Z}]{(X \ot Y) \ot Z}{X \ot (Y \ot Z)}}{X, Y, Z \in \cat{C}}$, $\lucvar = \family{\func[\luc{X}]{\unit \ot X}{X}}{X \in \cat{C}}$, and $\rucvar = \family{\func[\ruc{X}]{X \ot \unit}{X}}{X \in \cat{C}}$ of isomorphisms in $\cat{C}$, subject to the naturality of $\asscvar$, $\lucvar$, and $\rucvar$ and certain coherence conditions between them (see \cite[Definition XI.2.1]{kassel} for a precise definition).
In this case, we often refer to the pair $(\ordot, \unit)$ as a \term{monoidal structure} on $\cat{C}$.
A monoidal category $\moncat{C}$ is said to be \term{strict} if the equality $(X \ot Y) \ot Z = X \ot (Y \ot Z)$ holds for all $X, Y, Z \in \cat{C}$, the equality $\unit \ot X = X = X \ot \unit$ holds for all $X \in \cat{C}$, and the morphisms $\assc{X}{Y}{Z}$, $\luc{X}$, and $\ruc{X}$ are the identities.

We will often abbreviate a monoidal category $\moncat{C}$ to $\cat{C}$.
We will frequently distinguish structural collections of isomorphisms by writing them either as $\asscvar$, $\lucvar$, $\rucvar$, or as $\asscvar'$, $\lucvar'$, $\rucvar'$.

\begin{defn}[{cf.~\cite[Definition XI.4.1 (a)]{kassel}}]
    Let $\cat{C}$ and $\cat{D}$ be monoidal categories.
    A \term{monoidal functor} from $\cat{C}$ to $\cat{D}$ is a triple $\monfct{F}$ consisting of
    \begin{itemize}
        \item a functor $\fct[F]{C}{D}$,
        \item an isomorphism $\func[\monfctuni{F}]{\unit}{F(\unit)}$ in $\cat{D}$, and
        \item an isomorphism $\func[\monfctmul{F}{X}{Y}]{F(X) \ot F(Y)}{F(X \ot Y)}$ in $\cat{D}$ for all $X, Y \in \cat{C}$,
    \end{itemize}
    such that the following conditions are satisfied:
    \begin{enumerate}
        \item \label{enum:monfct-1} For all pairs of morphisms $\func[f]{X}{X'}$ and $\func[g]{Y}{Y'}$ in $\cat{C}$, we have the equality
        \begin{equation}
            \label{eq:monfct-0}
            \monfctmul{F}{X'}{Y'} \circ (F(f) \ot F(g)) = F(f \ot g) \circ \monfctmul{F}{X}{Y}.
        \end{equation}
        \item \label{enum:monfct-2} For all $X, Y, Z \in \cat{C}$, we have the equality
        \begin{equation}
            \label{eq:monfct-1}
            F(\assc{X}{Y}{Z}) \circ \monfctmul{F}{X \ot Y}{Z} \circ (\monfctmul{F}{X}{Y} \ot F(Z)) 
            = \monfctmul{F}{X}{Y \ot Z} \circ (F(X) \ot \monfctmul{F}{Y}{Z}) \circ \asscv{F(X)}{F(Y)}{F(Z)}.
        \end{equation}
        \item \label{enum:monfct-3} For all $X \in \cat{C}$, we have the equalities
        \begin{align}
            \label{eq:monfct-2-l} \lucv{F(X)} &= F(\luc{X}) \circ \monfctmul{F}{\unit}{X} \circ (\monfctuni{F} \ot F(X)); \\
            \label{eq:monfct-2-r} \rucv{F(X)} &= F(\ruc{X}) \circ \monfctmul{F}{X}{\unit} \circ (F(X) \ot \monfctuni{F}).
        \end{align}
    \end{enumerate}
    In this case, we often refer to the pair $(\monfctmulvar{F}, \monfctuni{F})$ as a \term{monoidal structure} on $F$.
    When $\cat{C}$ and $\cat{D}$ are strict, a monoidal functor $\monfct{F}$ from $\cat{C}$ to $\cat{D}$ is said to be \textit{strict} if the equality $F(X) \ot F(Y) = F(X \ot Y)$ holds for all $X, Y \in \cat{C}$, the equality $\unit = F(\unit)$ holds, and the morphisms $\monfctmul{F}{X}{Y}$ and $\monfctuni{F}$ are the identities.
    Given such a functor $\fct[F]{C}{D}$, condition~\ref{enum:monfct-1} holds if and only if the equality $F(f) \ot F(g) = F(f \ot g)$ holds for all pairs of morphisms $\func[f]{X}{X'}$ and $\func[g]{Y}{Y'}$ in $\cat{C}$, whereas conditions~\ref{enum:monfct-2} and \ref{enum:monfct-3} become automatic (see also Johnson--Yau \cite[Proposition 4.1.8]{johnson-yau}).
\end{defn}

We will often abbreviate a monoidal functor $\monfct{F}$ to $F$.
We note that some authors do not require $\monfctmul{F}{X}{Y}$ and $\monfctuni{F}$ of a monoidal functor to be isomorphisms, and they refer to a monoidal functor in our sense as a \textit{strong} monoidal functor (see \cite[page 12]{johnson-yau}).

\begin{defn}[{cf.~\cite[Definition XI.4.1 (b)]{kassel}}]
    \label{defn:mon-nat-trans}
    Let $\fct[\monfct{F}, \monfctv{G}]{C}{D}$ be monoidal functors.
    A \term{monoidal natural transformation} from $\monfct{F}$ to $\monfctv{G}$ is a natural transformation $\func[\eta]{F}{G}$ such that the equalities $\eta_{\unit} \circ \monfctuni{F} = \monfctvuni{G}$ and $\eta_{X \ot Y} \circ \monfctmul{F}{X}{Y} = \monfctvmul{G}{X}{Y} \circ (\eta_X \ot \eta_Y)$ hold for all $X, Y \in \cat{C}$.
\end{defn}

\begin{lem}
    \label{lem:monfct-easy}
    Let $\monfctvar{F}{C}{D}$ be a monoidal functor.
    Then, the assignment $V \mapsto V \ot F(-)$ gives a faithful functor $\func{\cat{D}}{\Fun{C}{D}}$.
\end{lem}

\begin{proof}
    This functor is defined by postcomposing $F$ with the functor $\func{\cat{D}}{\Fun{D}{D}}[V][V \ot (-)]$, and the faithfulness can be verified using the equalities
    \begin{align*}
        \rucv{W} \circ (W \ot \monfctuni{F}^{-1}) \circ (g \ot F(-))_{\unit} \circ (V \ot \monfctuni{F}) \circ \rucvinv{V} 
        &= \rucv{W} \circ (W \ot \monfctuni{F}^{-1}) \circ (g \ot F(\unit)) \circ (V \ot \monfctuni{F}) \circ \rucvinv{V} \\
        &= \rucv{W} \circ (g \ot \unit) \circ \rucvinv{V} 
        = g
    \end{align*}
    for all morphisms $\func[g]{V}{W}$ in $\cat{D}$, where the third equality follows from the naturality of $\rucvar'$.
\end{proof}

\begin{defn}[{cf.~\cite[Definition XIII.1.1]{kassel}}]
    A \term{braided monoidal category} is a pair $\brmoncat{C}$ consisting of
    \begin{itemize}
        \item a monoidal category $\cat{C}$, and
        \item an isomorphism $\func[\braid{X}{Y}]{X \ot Y}{Y \ot X}$ in $\cat{C}$ for all $X, Y \in \cat{C}$,
    \end{itemize}
    such that the following conditions are satisfied:
    \begin{enumerate}
        \item For all pairs of morphisms $\func[f]{X}{X'}$, $\func[g]{Y}{Y'}$ in $\cat{C}$, we have the equality
        \begin{equation}
            \label{eq:brmoncat-1}
            \braid{X'}{Y'} \circ (f \ot g) = (g \ot f) \circ \braid{X}{Y}.
        \end{equation}
        \item For all $X, Y, Z \in \cat{C}$, we have the equalities
        \begin{align}
            (Y \ot \braid{X}{Z}) \circ \assc{Y}{X}{Z} \circ (\braid{X}{Y} \ot Z) &= \assc{Y}{Z}{X} \circ \braid{X}{Y \ot Z} \circ \assc{X}{Y}{Z}; \label{eq:brmoncat-2} \\
            (\braid{X}{Z} \ot Y) \circ \assc{X}{Z}{Y}^{-1} \circ (X \ot \braid{Y}{Z}) &= \assc{Z}{X}{Y}^{-1} \circ \braid{X \ot Y}{Z} \circ \assc{X}{Y}{Z}^{-1}. \label{eq:brmoncat-3}
        \end{align}
    \end{enumerate}
    In this case, we refer to the collection $\br = \family{\braid{X}{Y}}{X, Y \in \cat{C}}$ as a \term{braiding} on $\cat{C}$.
\end{defn}

\begin{defn}[{cf.~\cite[Definition XIII.3.6]{kassel}}]
    Let $\brmoncat{C}$ and $\brmoncatv{D}$ be braided monoidal categories.
    A \term{braided monoidal functor} from $\brmoncat{C}$ to $\brmoncatv{D}$ is a monoidal functor $\monfctvar{F}{C}{D}$ such that the following equality holds for all $X, Y \in \cat{C}$:
    \begin{equation}
        \label{eq:brmonfct}
        \monfctmul{F}{Y}{X} \circ \braidv{F(X)}{F(Y)} = F(\braid{X}{Y}) \circ \monfctmul{F}{X}{Y}.
    \end{equation}
\end{defn}

\subsection{The augmented braided category}

Recall that for $n \ge 0$, the \textit{braid group $\brgrp[n]$ with $n$ strands} was introduced by Artin \cite{artin} as the group of ambient isotopy classes of braids with $n$ strands.
This group is isomorphic to the fundamental group of the configuration space of unordered distinct $n$ points in $\mathbb{R}^2$.
For our purpose, it suffices to recall only the following presentation of the braid groups; for all $n \ge 0$, the group $\brgrp$ is generated by $\{\brperm{i}{n} \mid 1 \le i \le n\}$, subject to the relations $\brperm{i}{n} \brperm{j}{n} = \brperm{j}{n} \brperm{i}{n}$ for all $1 \le i, j \le n$ with $\lvert i - j \rvert \ge 2$ and $\brperm{i}{n} \brperm{i + 1}{n} \brperm{i}{n} = \brperm{i + 1}{n} \brperm{i}{n} \brperm{i + 1}{n}$ for all $1 \le i < n$ (see also \cite[Corollary X.6.6]{kassel}).

\begin{defn}[{cf.~\cite[Definition 4.2.1]{angelini-knoll-chan-gerhardt-merling-peroux}}]
    \label{defn:augbr}
	The \term{augmented braided category} $\augbr$ is the category, with the objects $\gene{n}$ for $n \ge 0$, and morphisms generated by $\funcgene[\coface{j}{n}]{n}{n + 1}$ for $0 \le j \le n$, $\funcgene[\codege{j}{n}]{n + 2}{n + 1}$ for $0 \le j \le n$, and $\funcgene[\coperm{i}{n}, (\coperm{i}{n})^{-1}]{n + 1}{n + 1}$ for $1 \le i \le n$, subject to the following conditions:
	\begin{enumerate}
		\item \label{enum:cosimp-1} For all $0 \le i < j \le n$, we have the equality $\coface{j}{n} \circ \coface{i}{n - 1} = \coface{i}{n} \circ \coface{j - 1}{n - 1}$.
		\item \label{enum:cosimp-2} For all $0 \le i \le j \le n$, we have the equality $\codege{j}{n} \circ \codege{i}{n + 1} = \codege{i}{n} \circ \codege{j + 1}{n + 1}$.
		\item \label{enum:cosimp-3} For all $0 \le i \le n + 1$ and $0 \le j \le n$, we have the following equality:
		\begin{equation*}
            \codege{j}{n} \circ \coface{i}{n + 1} 
			= \begin{cases*}
				\coface{i}{n} \circ \codege{j - 1}{n - 1} & if $i < j$, \\
				\id{\gene{n + 1}} & if $i = j$ or $i = j + 1$, \\
				\coface{i - 1}{n} \circ \codege{j}{n - 1} & if $i > j + 1$.
			\end{cases*}
		\end{equation*}
		\item \label{enum:coaugbr-0} For all $1 \le i \le n$, the morphism $\coperm{i}{n}$ is an isomorphism in $\augbr$, whose inverse is given by $(\coperm{i}{n})^{-1}$.
		\item \label{enum:coaugbr-1} For all $1 \le i < n$, we have the equality $\coperm{i}{n} \circ \coperm{i + 1}{n} \circ \coperm{i}{n} = \coperm{i + 1}{n} \circ \coperm{i}{n} \circ \coperm{i + 1}{n}$.
		\item \label{enum:coaugbr-2} For all $1 \le i, j \le n$ with $\lvert i - j \rvert \ge 2$, we have the equality $\coperm{i}{n} \circ \coperm{j}{n} = \coperm{j}{n} \circ \coperm{i}{n}$.
		\item \label{enum:coaugbr-3} For all $1 \le i \le n$ and $0 \le j \le n$, we have the equality
		\[\coperm{i}{n} \circ \coface{j}{n} 
		= \begin{cases*}
			\coface{j}{n} \circ \coperm{i}{n - 1} & if $i < j$, \\
			\coface{j - 1}{n} & if $i = j$, \\
			\coface{j + 1}{n} & if $i = j + 1$, \\
			\coface{j}{n} \circ \coperm{i - 1}{n - 1} & if $i > j + 1$.
		\end{cases*}\]
		\item \label{enum:coaugbr-4} For all $1 \le i \le n$ and $0 \le j \le n$, we have the equality
		\[\coperm{i}{n} \circ \codege{j}{n} 
		= \begin{cases*}
			\codege{j}{n} \circ \coperm{i}{n + 1} & if $ i < j$, \\
			\codege{j - 1}{n} \circ \coperm{j + 1}{n + 1} \circ \coperm{j}{n + 1} & if $i = j$, \\
			\codege{j + 1}{n} \circ \coperm{j + 1}{n + 1} \circ \coperm{j + 2}{n + 1} & if $i = j + 1$, \\
			\codege{j}{n} \circ \coperm{i + 1}{n + 1} & if $i > j + 1$.
		\end{cases*}\]
	\end{enumerate}
    We refer to morphisms in $\augbr$ of the form $\coface{j}{n}$, $\codege{j}{n}$, $\coperm{i}{n}$, or $(\coperm{i}{n})^{-1}$, where $0 \le j \le n$ and $1 \le i \le n$, as \textit{generating morphisms of $\augbr$}.
    The category $\augbr$ admits a strict monoidal structure $(\ordot, \gene{0})$, where the functor $\ot$ is uniquely determined by the following conditions:
	\begin{enumerate}[resume]
		\item \label{enum:coaugbr-mon-0} For all $m, n \ge 0$, we have the equality $\gene{m} \ot \gene{n} = \gene{m + n}$. 
		\item \label{enum:coaugbr-mon-1} For all $m \ge 0$ and $0 \le j \le n$, we have the equalities $\gene{m} \ot \coface{j}{n} = \coface{m + j}{m + n}$ and $\coface{j}{n} \ot \gene{m} = \coface{j}{m + n}$.
		\item \label{enum:coaugbr-mon-2} For all $m \ge 0$ and $0 \le j \le n$, we have the equalities $\gene{m} \ot \codege{j}{n} = \codege{m + j}{m + n}$ and $\codege{j}{n} \ot \gene{m} = \codege{j}{m + n}$.
		\item \label{enum:coaugbr-mon-3} For all $m \ge 0$ and $1 \le i \le n$, we have the equalities $\gene{m} \ot \coperm{i}{n} = \coperm{m + i}{m + n}$ and $\coperm{i}{n} \ot \gene{m} = \coperm{i}{m + n}$.
    \end{enumerate}
\end{defn}

\begin{defn}
    \label{defn:augbrsemi}
    We let $\augbrsemi$ denote the category, with the same objects as $\augbr$, and morphisms generated by $\funcgene[\coface{j}{n}]{n}{n + 1}$ for $0 \le j \le n$ and $\funcgene[\coperm{i}{n}, (\coperm{i}{n})^{-1}]{n + 1}{n + 1}$ for $1 \le i \le n$, subject to conditions~\ref{enum:cosimp-1} and \ref{enum:coaugbr-0}--\ref{enum:coaugbr-3} of Definition~\ref{defn:augbr}.
    The category $\augbrsemi$ admits a strict monoidal structure $(\ordot, \gene{0})$, where the functor $\ot$ is uniquely determined by conditions~\ref{enum:coaugbr-mon-0}, \ref{enum:coaugbr-mon-1}, and \ref{enum:coaugbr-mon-3} of Definition~\ref{defn:augbr}.
\end{defn}

\begin{rem}
    Using the equalities in \ref{enum:coaugbr-mon-0}--\ref{enum:coaugbr-mon-3}, we can deduce the equality $\gene{n} = \otfold[\gene{1}]{n}$ for all $n \ge 0$, as well as the equalities $\coface{j}{n} = \coface{j}{j} \ot \gene{n - j} = \gene{j} \ot \cofacevar \ot \gene{n - j}$ and $\codege{j}{n} = \codege{j}{j} \ot \gene{n - j} = \gene{j} \ot \codegevar \ot \gene{n - j}$ for all $0 \le j \le n$, and the equality $\coperm{i}{n} = \coperm{i}{i} \ot \gene{n - i} = \gene{i - 1} \ot \copermvar \ot \gene{n - i}$ for all $1 \le i \le n$, and those except the third one hold not only in $\augbr$ but also in the subcategory $\augbrsemi$.
\end{rem}

We will make use of the above equalities implicitly in what follows.

The category $\augbr$ (and its subcategory $\augbrsemi$) are related to the notion of \textit{crossed simplicial groups}, introduced independently by Krasauskas \cite{krasauskas} and Fiedorowicz--Loday \cite{fiedorowicz-loday}.

\begin{rem}
    \label{rem:augsymm}
    \begin{enumerate}
        \item \label{enum:brcat} The full subcategory $\csg{B}$ of $\augbr$ consisting of objects $\gene{n}$ for all $n > 0$ is the same as the crossed simplicial group arising from the collection $\famnat{\brgrp}$ of the braid groups.
        \item The category $\augbrsemi$ is isomorphic to the \textit{category $\mathfrak{B}$ of braided injections}, which was introduced by Schlichtkrull--Solberg \cite[Definition 2.1]{schlichtkrull-solberg}, whose presentation appeared in \cite[Lemma 3.4]{schlichtkrull-solberg}.
        \item \label{enum:symmcat} Let $\augsymm$ denote the category constructed from $\augbr$ by imposing an extra condition that the equality $(\coperm{i}{n})^2 = \id{\gene{n + 1}}$ hold for all $1 \le i \le n$, and let $\symmcat$ denote the full subcategory of $\augsymm$ defined similarly to \ref{enum:brcat}.
        Then, $\symmcat$ is also the same as the crossed simplicial group arising from the collection $\famnat{\symmgrp}$ of the symmetric groups.
        The category $\augsymm$ inherits a strict monoidal structure from that on $\augbr$ via the evident functor $\func{\augbr}{\augsymm}$, which is described in a more general context in Angelini-Knoll--Merling--P\'eroux \cite[Lemma 2.25]{angelini-knoll-merling-peroux}.
        Let $\augsymmsemi$ denote the (monoidal) subcategory of $\augsymm$, defined similarly to Definition~\ref{defn:augbrsemi}.
        \item The nerve of $\augbr$ was identified with the $\mathbb{E}_2$-monoidal envelope of the associative operad in Angelini-Knoll--Chan--Gerhardt--Merling--P\'eroux \cite[Theorem 4.2.5]{angelini-knoll-chan-gerhardt-merling-peroux}, following an analogous result for $\augsymm$ given in \cite[Proposition 3.10]{angelini-knoll-merling-peroux}.
    \end{enumerate}
\end{rem}

We collect a few properties relating the categories $\augbr$ and $\augbrsemi$ and the braid groups.
Given a category $\cat{C}$ and $X \in \cat{C}$, we write $\Aut{\cat{C}}{X}$ for the group of automorphisms of $X$ in $\cat{C}$.

\begin{rem}
    \label{rem:augbr}
    \begin{enumerate}
        \item \label{enum:rem-augbr-1} We have an evident strict monoidal functor $\func[\incl]{\augbrsemi}{\augbr}$.
        In fact, $\incl$ is a faithful functor, and we can thus regard $\augbrsemi$ as a monoidal subcategory of $\augbr$, though we will not need this observation in this paper.
        \item \label{enum:rem-augbr-2} For all $n \ge 0$, we have a unique group homomorphism $\func[\rho_{n + 1}]{\brgrp}{\Aut{\augbrsemi}{\gene{n + 1}}}$ such that $\rho_{n + 1}(\brperm{i}{n}) = \coperm{i}{n}$ for all $1 \le i \le n$.
        In fact, $\rho_{n + 1}$ are isomorphisms of groups, though we will not need this observation in this paper.
        \item Similarly to \ref{enum:rem-augbr-1} and \ref{enum:rem-augbr-2}, we have a strict monoidal functor $\func{\augsymmsemi}{\augsymm}$ as well as group homomorphisms $\func{\symmgrp}{\Aut{\augsymmsemi}{\gene{n + 1}}}$ for all $n \ge 0$.
    \end{enumerate}
\end{rem}

\begin{rem}
    \label{rem:generate-augbr}
    We claim that every morphism in $\augbr$ can be written as a composition of generating morphisms.
    To prove this, let $\cat{I}'$ be a subcategory of $\augbr$ which contains all the generating morphisms.
    It suffices to prove that $\cat{I}' = \augbr$.
    Since conditions~\ref{enum:cosimp-1}--\ref{enum:coaugbr-4} of Definition~\ref{defn:augbr} still hold in the subcategory $\cat{I}'$, we have a (unique) functor $\func{\augbr}{\cat{I}'}$ which sends every object and generating morphism to itself.
    Since the composition $\augbr \to \cat{I'} \hookrightarrow \augbr$ is the identity on objects and on generating morphisms, this composition is equal to the identity functor $\id{\augbr}$.
    Thus, the inclusion functor $\cat{I'} \hookrightarrow \augbr$ is a retraction, and is in particular full.
    This finishes the proof.
\end{rem}

\subsection{Evaluations}

Let $\cat{C}$ and $\cat{D}$ be categories.
The assignments $(F, X) \mapsto F(X)$ and $(\eta, X) \mapsto \eta_X$ give a functor $\func[\mathrm{ev}]{\Fun{C}{D} \times \cat{C}}{\cat{D}}$, and given a specified object $X$ in $\cat{C}$, precomposing $\mathrm{ev}$ with the evident functor $\Fun{C}{D} \to \Fun{C}{D} \times \{X\} \hookrightarrow \Fun{C}{D} \times \cat{C}$ gives a functor $\func{\Fun{C}{D}}{\cat{D}}$ which witnesses the evaluation at $X$.
This procedure can also be formulated in the setting of closed monoidal categories.
Recall that a monoidal category $\cat{C}$ is said to be (\textit{left}) \textit{closed} if the functor $(-) \ot X$ admits a right adjoint (written $[X, -]$) for all $X \in \cat{C}$ (see B\"ohm \cite[Definition 3.23]{bohm}).
Note that a reader can safely skip the proof of Lemma~\ref{lem:ev}, as it becomes an elementary fact when $\cat{C}$ is the (huge) Cartesian closed category of categories.

\begin{nota}
    Let $\cat{C}$ be a closed monoidal category, and $X$ and $Y$ objects in $\cat{C}$.
    \begin{enumerate}
        \item We let $\func[\eta^X]{\id{\cat{C}}}{[X, (-) \ot X]}$ and $\func[\evvar{X}]{[X, -] \ot X}{\id{\cat{C}}}$ denote the unit and the counit of the adjunction $(-) \ot X \dashv [X, -]$, respectively.
        \item Given a morphism $\func[x]{\unit}{X}$ in $\cat{C}$, we let $\ev{x}$ denote the composite
        \[[X, Y] \xrightarrow{\ruc{[X, Y]}^{-1}} [X, Y] \ot \unit \xrightarrow{[X, Y] \ot x} [X, Y] \ot X \xrightarrow{\evvar{X}[Y]} Y.\]
        \item Given a morphism $\func[f]{X}{X'}$, we let $[f, Y]$ denote the composite
        \[[X', Y] \xrightarrow{\etavar} [X, [X', Y] \ot X] \xrightarrow{[X, [X', Y] \ot f]} [X, [X', Y] \ot X'] \xrightarrow{[X, \evvar{X'}[Y]]} [X, Y].\]
    \end{enumerate}
\end{nota}

\begin{lem}
    \label{lem:ev}
    Let $\cat{C}$ be a closed monoidal category, $X$ and $Y$ objects in $\cat{C}$, and $\func[x]{\unit}{X}$ a morphism in $\cat{C}$.
    \begin{enumerate}
        \item \label{enum:ev-1} For all morphisms $\func[g]{Y}{Y'}$ in $\cat{C}$, we have $g \circ \ev{x} = \ev{x} \circ [X, g]$.
        \item \label{enum:ev-2} For all morphisms $\func[f]{X}{X'}$ in $\cat{C}$, we have $\ev{f \circ x} = \ev{x} \circ [f, Y]$.
    \end{enumerate}
\end{lem}

Note that we use the symbol $\star$ for the horizontal composition of natural transformations.

\begin{proof}
    \begin{enumerate}
        \item This is rewritten as $g \circ \evvar{X}[Y] \circ ([X, Y] \ot x) \circ \ruc{[X, Y]}^{-1} = \evvar{X}[Y'] \circ ([X, Y'] \ot x) \circ \ruc{[X, Y']}^{-1} \circ [X, g]$, which follows from the naturality of $\evvar{X} \circ ([X, -] \ot x) \circ \ruc{[X, -]}^{-1}$.
        \item This equality is rewritten as
        \begin{align*}
            & \evvar{X'}[Y] \circ ([X', Y] \ot f) \circ ([X', Y] \ot x) \circ \ruc{[X', Y]}^{-1} \\
            & \qquad \qquad = \evvar{X}[Y] \circ ([X, Y] \ot x) \circ \ruc{[X, Y]}^{-1} \circ [X, \evvar{X'}[Y]] \circ [X, [X', Y] \ot f] \circ \etavar.
        \end{align*}
        Since
        \begin{align*}
            & ([X, Y] \ot x) \circ \ruc{[X, Y]}^{-1} \circ [X, \evvar{X'}[Y]] \circ [X, [X', Y] \ot f] \circ \etavar \\
            & \qquad \qquad  = ([X, \evvar{X'}[Y]] \ot X) \circ ([X, [X', Y] \ot f] \ot X) \circ (\etavar \ot X) \circ ([X', Y] \ot x) \circ \ruc{[X', Y]}^{-1}
        \end{align*}
        by the naturality of $((-) \ot x) \circ \rucvar^{-1}$, it suffices to prove that
        \[\evvar{X'}[Y] \circ ([X', Y] \ot f) = \evvar{X}[Y] \circ ([X, \evvar{X'}[Y]] \ot X) \circ ([X, [X', Y] \ot f] \ot X) \circ (\etavar \ot X).\]
        Since $\evvar{X}[Y] \circ ([X, \evvar{X'}[Y]] \ot X) \circ ([X, [X', Y] \ot f] \ot X) = \evvar{X'}[Y] \circ ([X', Y] \ot f) \circ \evvar{X}[[X', Y] \ot X]$ by the naturality of $\evvar{X}$, it remains to prove that $\evvar{X}[[X', Y] \ot X] \circ (\etavar \ot X) = \id{[X', Y] \ot X}$, which follows from (one of) the triangle identities $(\evvar{X} \star ((-) \ot X)) \circ (((-) \ot X) \star \eta^X) = \id{(-) \ot X}$ (see \cite[page 85, (9)]{maclane}).
    \end{enumerate}
    \vspace{-1\baselineskip}
\end{proof}

\subsection{Linear categories}

Recall that a \term{$\Bbbk$-linear category} is a category which is enriched over the category of $\Bbbk$-modules.
In other words, a $\Bbbk$-linear category is a category $\cat{C}$ where the set $\hom{\cat{C}}{X}{Y}$ of morphisms from $X$ to $Y$ is endowed with a fixed structure of a $\Bbbk$-module (written $\linhom{C}{X}{Y}$) for all $X, Y \in \cat{C}$, such that the composition law $\func{\linhom{\cat{C}}{Y}{Z} \times \linhom{\cat{C}}{X}{Y}}{\linhom{\cat{C}}{X}{Z}}$ is $\Bbbk$-bilinear for all $X, Y, Z \in \cat{C}$.
We omit the definitions of \term{$\Bbbk$-linear functors}, \term{$\Bbbk$-linear natural transformations}, and \term{$\Bbbk$-linear equivalences}, which are just special cases of enriched functors, enriched transformations, and equivalences of enriched categories.
In what follows, we will not distinguish a $\Bbbk$-linear category and its \term{underlying category}.

The next result is standard in enriched category theory, and we omit the proof.

\begin{lem}
    \label{lem:k-lin}
    Let $\fct[F]{C}{D}$ be a $\Bbbk$-linear functor, and $M$ an object in $\cat{C}$.
    Then, the collection of $\Bbbk$-linear maps $\func{\linhom{C}{X}{M}}{\linhom{D}{F(X)}{F(M)}}$ indexed by $X \in \cat{C}$ induces a natural transformation $\func{\linhom{C}{-}{M}}{\linhom{D}{F(-)}{F(M)}}$, which is a natural isomorphism if $F$ is fully faithful in the sense that the $\Bbbk$-linear maps $\func{\linhom{C}{X}{Y}}{\linhom{D}{F(X)}{F(Y)}}$ are isomorphisms for all $X, Y \in \cat{C}$.
\end{lem}

\section{Braided cohomology of braided coalgebras}
\label{sec:3}

In this section, we begin by recalling the notion of braided coalgebras in monoidal categories in the sense of Ardizzoni--Menini \cite{ardizzoni-menini}.
To every braided coalgebra in a monoidal category $\cat{C}$, we can associate a strict monoidal functor from $\op{\augbr}$ (called a strict monoidal {\ABSO} in our terminology), which takes values in $\cat{C}$ when $\cat{C}$ is strict (see Proposition~\ref{prop:brcoalg-to-ABSO}), and in $\Fun{C}{C}$ in the general case (see Proposition~\ref{prop:nonstr-augbrobj}).
We will then use this to define the notion of the braided cochain complex and the braided cohomology of a braided coalgebra in a $\Bbbk$-linear monoidal category in Subsection~\ref{subsec:3.3}.
We also give a functoriality result on the braided cochain complex (see Lemma~\ref{lem:KB}) which will be useful later, and we provide an explicit description of an {\ABSO} associated to a braided coalgebra, defined in Definition~\ref{defn:br-cpx}~\ref{br-cpx:ABSO}, up to isomorphism (see Remark~\ref{rem:wTC}).

\subsection{Braided coalgebras}

We review here the notion of braided coalgebras and related concepts, and provide an example of a braided coalgebra which comes from a coalgebra.

\begin{defn}[{\cite[Definition 2.1 (1) and (3)]{ardizzoni-menini}}]
    \label{defn:brcoalg}
    Let $\cat{C}$ be a monoidal category.
    \begin{enumerate}
        \item A triple $\coalg{C}$ is called a \term{coalgebra in $\cat{C}$} if $C$ is an object in $\cat{C}$, $\func[\Delta]{C}{C \ot C}$ and $\func[\eps]{C}{\unit}$ are morphisms in $\cat{C}$, and the following equalities are satisfied:
        \begin{align}
            \assc{C}{C}{C} \circ (\Delta \ot C) \circ \Delta &= (C \ot \Delta) \circ \Delta; \label{eq:coalg-1} \\
            (\eps \ot C) \circ \Delta &= \luc{C}^{-1}; \label{eq:coalg-2} \\
            (C \ot \eps) \circ \Delta &= \ruc{C}^{-1}. \nonumber 
        \end{align}
        \item A pair $\brobj{C}$ is called a \term{braided object in $\cat{C}$} if $C$ is an object in $\cat{C}$, $\func[\br]{C \ot C}{C \ot C}$ is an isomorphism in $\cat{C}$, and the following equality is satisfied:
        \begin{equation}
            \begin{split}
                \label{eq:braid}
                & \assc{C}{C}{C} \circ (\br \ot C) \circ \assc{C}{C}{C}^{-1} \circ (C \ot \br) \circ \assc{C}{C}{C} \circ (\br \ot C) \\
                & \qquad \qquad \qquad \qquad = (C \ot \br) \circ \assc{C}{C}{C} \circ (\br \ot C) \circ \assc{C}{C}{C}^{-1} \circ (C \ot \br) \circ \assc{C}{C}{C}.
            \end{split}
        \end{equation}
        \item A quadruple $\brcoalg{C}$ is called a \term{braided coalgebra in $\cat{C}$} if $\coalg{C}$ is a coalgebra in $\cat{C}$, $\brobj{C}$ is a braided object in $\cat{C}$, and the following equalities are satisfied:
        \begin{align}
            \luc{C} \circ (\eps \ot C) \circ \br &= \ruc{C} \circ (C \ot \eps); \label{eq:brcoalg-1} \\
            \ruc{C} \circ (C \ot \eps) \circ \br &= \luc{C} \circ (\eps \ot C); \nonumber \\
            \assc{C}{C}{C} \circ (\Delta \ot C) \circ \br &= (C \ot \br) \circ \assc{C}{C}{C} \circ (\br \ot C) \circ \assc{C}{C}{C}^{-1} \circ (C \ot \Delta); \label{eq:brcoalg-2} \\
            \assc{C}{C}{C}^{-1} \circ (C \ot \Delta) \circ \br &= (\br \ot C) \circ \assc{C}{C}{C}^{-1} \circ (C \ot \br) \circ \assc{C}{C}{C} \circ (\Delta \ot C). \nonumber 
        \end{align}
    \end{enumerate}
\end{defn}

\begin{ex}
    \label{ex:brcoalg-in-opaugbr}
    $(\gene{1}, \codegevar, \cofacevar, \copermvar)$ is a braided coalgebra in $\op{\augbr}$.
    Indeed, it follows from conditions~\ref{enum:cosimp-1}--\ref{enum:cosimp-3} of Definition~\ref{defn:augbr} that $(\gene{1}, \codegevar, \cofacevar)$ is a coalgebra in $\op{\augbr}$ (see \cite[page 175]{maclane} for a dual statement).
    We write $\opcirc$ for the composition law in $\op{\augbr}$.
    It follows from condition~\ref{enum:coaugbr-0} that $\copermvar$ is an isomorphism in $\op{\augbr}$, and from conditions~\ref{enum:coaugbr-1} and \ref{enum:coaugbr-mon-3} that the equalities
    \begin{align*}
        (\copermvar \ot \gene{1}) \opcirc (\gene{1} \ot \copermvar) \opcirc (\copermvar \ot \gene{1}) 
        &= \coperm{1}{2} \opcirc \coperm{2}{2} \opcirc \coperm{1}{2} 
        = \coperm{2}{2} \opcirc \coperm{1}{2} \opcirc \coperm{2}{2} \\
        &= (\gene{1} \ot \copermvar) \opcirc (\copermvar \ot \gene{1}) \opcirc (\gene{1} \ot \copermvar)
    \end{align*}
    hold, and hence $(\gene{1}, \copermvar)$ is a braided object in $\op{\augbr}$.
    On the other hand, (\ref{eq:brcoalg-1}) is verified as
    \[(\cofacevar \ot \gene{1}) \opcirc \copermvar = \coface{0}{1} \opcirc \copermvar = \coface{1}{1} = \gene{1} \ot \cofacevar,\]
    where the first and third equalities follow from condition~\ref{enum:coaugbr-mon-1}, and the second equality from condition~\ref{enum:coaugbr-3}, whereas (\ref{eq:brcoalg-2}) is verified as
    \[(\codegevar \ot \gene{1}) \opcirc \copermvar 
    = \codege{0}{1} \opcirc \copermvar 
    = \coperm{2}{2} \opcirc \coperm{1}{2} \opcirc \codege{1}{1} 
    = (\gene{1} \ot \copermvar) \opcirc (\copermvar \ot \gene{1}) \opcirc (\gene{1} \ot \codegevar),\]
    where the first and third equalities follow from conditions~\ref{enum:coaugbr-mon-2} and \ref{enum:coaugbr-mon-3}, and the second equality from condition~\ref{enum:coaugbr-4}.
    The remaining equalities can be verified similarly.
\end{ex}

We now explain how coalgebras in braided monoidal categories can be promoted to braided coalgebras in a natural manner.

\begin{ex}
    \label{ex:coalg-br}
    Let $\brmoncat{C}$ be a braided monoidal category.
    \begin{enumerate}
        \item \label{enum:ex-1} Let $C$ be an object in $\cat{C}$.
        Then, $(C, \braid{C}{C})$ is a braided object in $\cat{C}$, since $\braid{C}{C}$ is an isomorphism in $\cat{C}$ and since (\ref{eq:braid}) follows from \cite[Theorem XIII.1.3]{kassel}.
        \item \label{enum:ex-2} Let $\coalg{C}$ be a coalgebra in $\cat{C}$.
        Then, $(C, \Delta, \eps, \braid{C}{C})$ is a braided coalgebra in $\cat{C}$.
        Indeed, it follows from \ref{enum:ex-1} that $(C, \braid{C}{C})$ is a braided object in $\cat{C}$, and we have
        \begin{align}
            \label{eq:ex-1} \luc{C} \circ (\eps \ot C) \circ \braid{C}{C} 
            &= \luc{C} \circ \braid{C}{\unit} \circ (C \ot \eps) 
            = \ruc{C} \circ (C \ot \eps); \\
            \label{eq:ex-2} \ruc{C} \circ (C \ot \eps) \circ \braid{C}{C} 
            &= \ruc{C} \circ \braid{\unit}{C} \circ (\eps \ot C) 
            = \luc{C} \circ (\eps \ot C); \\
            \begin{split}
                \label{eq:ex-3} \assc{C}{C}{C} \circ (\Delta \ot C) \circ \braid{C}{C} 
                &= \assc{C}{C}{C} \circ \braid{C}{C \ot C} \circ (C \ot \Delta) \\
                &= (C \ot \braid{C}{C}) \circ \assc{C}{C}{C} \circ (\braid{C}{C} \ot C) \circ \assc{C}{C}{C}^{-1} \circ (C \ot \Delta);
            \end{split} \\
            \begin{split}
                \label{eq:ex-4} \assc{C}{C}{C}^{-1} \circ (C \ot \Delta) \circ \braid{C}{C} 
                &= \assc{C}{C}{C}^{-1} \circ \braid{C \ot C}{C} \circ (\Delta \ot C) \\
                &= (\braid{C}{C} \ot C) \circ \assc{C}{C}{C}^{-1} \circ (C \ot \braid{C}{C}) \circ \assc{C}{C}{C} \circ (\Delta \ot C),
            \end{split}
        \end{align}
        where the first equalities follow from (\ref{eq:brmoncat-1}), the second equalities of (\ref{eq:ex-1}) and (\ref{eq:ex-2}) follow from \cite[Proposition XIII.1.2]{kassel}, and those of (\ref{eq:ex-3}) and (\ref{eq:ex-4}) follow from (\ref{eq:brmoncat-2}) and (\ref{eq:brmoncat-3}), respectively.
    \end{enumerate}
\end{ex}

We will often abbreviate a braided coalgebra $\brcoalg{C}$ to $\brobj{C}$.

In Subsection~\ref{subsec:4.3}, we will deal with more concrete examples of braided coalgebras, which arise from quasi-triangular bialgebras (see Definition~\ref{defn:qtb} and Remark~\ref{rem:qtb-to-brcoalg}).

\subsection{Monoidal {\ABSO}s}

In this subsection, we define the notion of (monoidal) {\ABbsbSO}s.
Then, we associate, to every braided coalgebra $\brobj{C}$ in a monoidal category $\cat{C}$, a strict monoidal {\ABSO} $\wTcalC$ in $\Fun{C}{C}$ (see Proposition~\ref{prop:nonstr-augbrobj}).

\begin{defn}
    Let $\cat{C}$ be a category.
    \begin{enumerate}
        \item An \term{{\ABSO} in $\cat{C}$} is a functor $\func{\op{\augbr}}{\cat{C}}$.
        \item An \term{{\ABsSO} in $\cat{C}$} is a functor $\func{\op{\augbrsemi}}{\cat{C}}$.
    \end{enumerate}
\end{defn}

\begin{defn}
    \label{defn:MABSO}
    Let $\cat{C}$ be a monoidal category.
    A \term{monoidal {\ABSO} in $\cat{C}$} is a monoidal functor $\func{\op{\augbr}}{\cat{C}}$.
    When $\cat{C}$ is strict as a monoidal category, a monoidal {\ABSO} in $\cat{C}$ is said to be \term{strict} if it is strict as a monoidal functor.
\end{defn}

A (\textit{strict}) \textit{monoidal {\ABsSO}} is defined similarly.

\begin{rem}
    By precomposing with the strict monoidal functor $\func[\incl]{\op{\augbrsemi}}{\op{\augbr}}$ of Remark~\ref{rem:augbr} \ref{enum:rem-augbr-1}, any (monoidal) {\ABSO} gives rise to a (monoidal) {\ABsSO}, and this transference preserves strictness.
\end{rem}

\begin{nota}
    Let $\cat{C}$ be a (monoidal) category.
    A (monoidal) {\ABbsbSO} in $\cat{C}$ is typically denoted by $\sobj{X}$ or $X$.
    In these cases, we write $X_n$ for its value on $\gene{n}$ for all $n \ge 0$.
\end{nota}

\begin{nota}
    \label{nota:ABSO}
    Let $\cat{C}$ be a category, and $\sobj{X}$ an {\ABSO} in $\cat{C}$.
    \begin{enumerate}
        \item \label{enum:face} For all $0 \le j \le n$, we write $\face{j}{n} \coloneqq X(\coface{j}{n})$, which is a morphism from $X_{n + 1}$ to $X_n$ in $\cat{C}$.
        \item For all $0 \le j \le n$, we write $\dege{j}{n} \coloneqq X(\codege{j}{n})$, which is a morphism from $X_{n + 1}$ to $X_{n + 2}$ in $\cat{C}$.
        \item \label{enum:perm} For all $1 \le i \le n$, we write $\perm{i}{n} \coloneqq X(\coperm{i}{n})$, which is a morphism from $X_{n + 1}$ to $X_{n + 1}$ in $\cat{C}$.
    \end{enumerate}
\end{nota}

We use the symbols of Notation~\ref{nota:ABSO}~\ref{enum:face} and \ref{enum:perm} for {\ABsSO}s as well.

Let $\cat{C}$ be a monoidal category, and $\sobj{X}$ a monoidal {\ABSO} in $\cat{C}$.
Since $(\gene{1}, \codegevar, \cofacevar, \copermvar)$ is a braided coalgebra in $\op{\augbr}$ by Example~\ref{ex:brcoalg-in-opaugbr}, one can construct a braided coalgebra in $\cat{C}$ using Proposition~\ref{prop:brcoalg-monfct} below.
The following proposition gives a converse assignment in the case when the monoidal category $\cat{C}$ is strict.

\begin{prop}
    \label{prop:brcoalg-to-ABSO}
	Let $\cat{C}$ be a strict monoidal category, and $\brcoalg{C}$ a braided coalgebra in $\cat{C}$.
    Then, there exists a unique strict monoidal {\ABSO} $\sobj{X}$ in $\cat{C}$ which satisfies the equalities $X_1 = C$, $X(\cofacevar) = \eps$, $X(\codegevar) = \Delta$, and $X(\copermvar) = \br$.
\end{prop}

\begin{proof}
    We first prove the uniqueness.
    Let $\sobj{X}$ be any strict monoidal {\ABSO} in $\cat{C}$.
    By Definition~\ref{defn:augbr}, $\sobj{X}$ can be recovered from the collection $\famnat{X_n}$ of objects in $\cat{C}$ and the collections
    \[\family{\func[\face{j}{n}]{X_{n + 1}}{X_n}}{0 \le j \le n}; 
    \quad \family{\func[\dege{j}{n}]{X_{n + 1}}{X_{n + 2}}}{0 \le j \le n}; 
    \quad \family{\func[\perm{i}{n}]{X_{n + 1}}{X_{n + 1}}}{1 \le i \le n}\]
    of morphisms in $\cat{C}$.
    Since $X_n = X(\otfold[\gene{1}]{n}) = \otfold[(X_1)]{n}$ for all $n \ge 0$,
    \begin{align*}
        \face{j}{n} 
        &= X(\gene{j} \ot \cofacevar \ot \gene{n - j}) = \otfold[(X_1)]{j} \ot X(\cofacevar) \ot \otfoldv[(X_1)]{n - j}; \\
        \dege{j}{n} 
        &= X(\gene{j} \ot \codegevar \ot \gene{n - j}) = \otfold[(X_1)]{j} \ot X(\codegevar) \ot \otfoldv[(X_1)]{n - j}
    \end{align*}
    for all $0 \le j \le n$, and $\perm{i}{n} = X(\gene{i - 1} \ot \copermvar \ot \gene{n - i}) = \otfoldv[(X_1)]{i - 1} \ot X(\copermvar) \ot \otfoldv[(X_1)]{n - i}$ for all $1 \le i \le n$, it follows that $\sobj{X}$ is uniquely determined from the quadruple $(X_1, X(\cofacevar), X(\codegevar), X(\copermvar))$.
    This proves the uniqueness.
    Next, we prove the existence.
    We define collections
    \[\family{\func[\face{j}{n}]{\otfoldv{n + 1}}{\otfold{n}}}{0 \le j \le n}; 
    \quad \family{\func[\dege{j}{n}]{\otfoldv{n + 1}}{\otfoldv{n + 2}}}{0 \le j \le n}; 
    \quad \family{\func[\perm{i}{n}]{\otfoldv{n + 1}}{\otfoldv{n + 1}}}{1 \le i \le n}\]
    of morphisms in $\cat{C}$ as follows:
    \begin{equation}
        \label{eq:wTC}
        \face{j}{n} \coloneqq \otfold{j} \ot \eps \ot \otfoldv{n - j}; 
        \quad \dege{j}{n} \coloneqq \otfold{j} \ot \Delta \ot \otfoldv{n - j}; 
        \quad \perm{i}{n} \coloneqq \otfoldv{i - 1} \ot \br \ot \otfoldv{n - i}.
    \end{equation}
    To define an {\ABSO} in $\cat{C}$ from these collections of morphisms and from the collection $\famnat{\otfold{n}}$ of objects, it suffices to prove that these collections satisfy the (dual) conditions of \ref{enum:cosimp-1} through \ref{enum:coaugbr-4} of Definition~\ref{defn:augbr}.
    The condition~\ref{enum:cosimp-1} is verified as follows:
    \begin{align*}
        \facei{i}{n - 1} \circ \face{j}{n} 
        &= (\otfold{i} \ot \eps \ot \otfoldv{n - i - 1}) \circ (\otfold{j} \ot \eps \ot \otfoldv{n - j}) \\
        &= \otfold{i} \ot ((\eps \ot \otfoldv{j - i - 1} \ot \unit) \circ (C \ot \otfoldv{j - i - 1} \ot \eps)) \ot \otfoldv{n - j} \\
        &= \otfold{i} \ot ((\unit \ot \otfoldv{j - i - 1} \ot \eps) \circ (\eps \ot \otfoldv{j - i - 1} \ot C)) \ot \otfoldv{n - j} \\
        &= (\otfoldv{j - 1} \ot \eps \ot \otfoldv{n - j}) \circ (\otfold{i} \ot \eps \ot \otfoldv{n - i}) 
        = \face{j - 1}{n - 1} \circ \facei{i}{n}.
    \end{align*}
    We note that condition~\ref{enum:cosimp-2} in the case $i < j$, conditions~\ref{enum:cosimp-3}, \ref{enum:coaugbr-3}, and \ref{enum:coaugbr-4} in the case $i < j$ or $i > j + 1$, and condition~\ref{enum:coaugbr-2} follow directly from (\ref{eq:wTC}) as well.
    The condition~\ref{enum:cosimp-2} in the case $i = j$ is verified as
    \begin{align*}
        \dege{j}{n + 1} \circ \dege{j}{n} 
        &= (\otfold{j} \ot \Delta \ot \otfoldv{n - j + 1}) \circ (\otfold{j} \ot \Delta \ot \otfoldv{n - j}) \\
        &= \otfold{j} \ot ((\Delta \ot C) \circ \Delta) \ot \otfoldv{n - j} \\
        &= \otfold{j} \ot ((C \ot \Delta) \circ \Delta) \ot \otfoldv{n - j} \\
        &= (\otfoldv{j + 1} \ot \Delta \ot \otfoldv{n - j}) \circ (\otfold{j} \ot \Delta \ot \otfoldv{n - j}) 
        = \dege{j + 1}{n + 1} \circ \dege{j}{n},
    \end{align*}
    where the third equality follows from (\ref{eq:coalg-1}).
    The condition~\ref{enum:cosimp-3} in the case $i = j$ is verified as
    \begin{align*}
        \face{j}{n + 1} \circ \dege{j}{n} 
        &= (\otfold{j} \ot \eps \ot \otfoldv{n - j + 1}) \circ (\otfold{j} \ot \Delta \ot \otfoldv{n - j}) \\
        &= \otfold{j} \ot ((\eps \ot C) \circ \Delta) \ot \otfoldv{n - j} 
        = \otfold{j} \ot \id{C} \ot \otfoldv{n - j} 
        = \id{\otfoldv{n + 1}},
    \end{align*}
    where the third equality follows from (\ref{eq:coalg-2}).
    The condition~\ref{enum:coaugbr-0} follows since $\br$ is an isomorphism in $\cat{C}$.
    The condition~\ref{enum:coaugbr-1} is verified as
    \begin{align*}
        \perm{i}{n} \circ \perm{i + 1}{n} \circ \perm{i}{n} 
        &= (\otfoldv{i - 1} \ot \br \ot \otfoldv{n - i}) \circ (\otfold{i} \ot \br \ot \otfoldv{n - i - 1}) \circ (\otfoldv{i - 1} \ot \br \ot \otfoldv{n - i}) \\
        &= \otfoldv{i - 1} \ot ((\br \ot C) \circ (C \ot \br) \circ (\br \ot C)) \ot \otfoldv{n - i - 1} \\
        &= \otfoldv{i - 1} \ot ((C \ot \br) \circ (\br \ot C) \circ (C \ot \br)) \ot \otfoldv{n - i - 1} \\
        &= (\otfold{i} \ot \br \ot \otfoldv{n - i - 1}) \circ (\otfoldv{i - 1} \ot \br \ot \otfoldv{n - i}) \circ (\otfold{i} \ot \br \ot \otfoldv{n - i - 1}) \\
        &= \perm{i + 1}{n} \circ \perm{i}{n} \circ \perm{i + 1}{n},
    \end{align*}
    where the third equality follows from (\ref{eq:braid}).
    The condition~\ref{enum:coaugbr-3} in the case $i = j + 1$ is verified as
    \begin{align*}
        \face{j}{n} \circ \permj{j + 1}{n} 
        &= (\otfold{j} \ot \eps \ot \otfoldv{n - j}) \circ (\otfold{j} \ot \br \ot \otfoldv{n - j - 1}) \\
        &= \otfold{j} \ot ((\eps \ot C) \circ \br) \ot \otfoldv{n - j - 1} \\
        &= \otfold{j} \ot (C \ot \eps) \ot \otfoldv{n - j - 1} 
        = \otfoldv{j + 1} \ot \eps \ot \otfoldv{n - j - 1} 
        = \face{j + 1}{n},
    \end{align*}
    where the third equality follows from (\ref{eq:brcoalg-1}).
    The condition~\ref{enum:coaugbr-4} in the case $i = j + 1$ is verified as
    \begin{align*}
        \dege{j}{n} \circ \perm{j + 1}{n} 
        &= (\otfold{j} \ot \Delta \ot \otfoldv{n - j}) \circ (\otfold{j} \ot \br \ot \otfoldv{n - j - 1}) \\
        &= \otfold{j} \ot ((\Delta \ot C) \circ \br) \ot \otfoldv{n - j - 1} \\
        &= \otfold{j} \ot ((C \ot \br) \circ (\br \ot C) \circ (C \ot \Delta)) \ot \otfoldv{n - j - 1} \\
        &= (\otfoldv{j + 1} \ot \br \ot \otfoldv{n - j - 1}) \circ (\otfold{j} \ot \br \ot \otfoldv{n - j}) \circ (\otfoldv{j + 1} \ot \Delta \ot \otfoldv{n - j - 1}) \\
        &= \permj{j + 2}{n + 1} \circ \permj{j + 1}{n + 1} \circ \dege{j + 1}{n},
    \end{align*}
    where the third equality follows from (\ref{eq:brcoalg-2}).
    The remaining conditions follow from similar arguments, using the unlabeled equalities of Definition~\ref{defn:brcoalg}.
    We thus obtain an {\ABSO} $\sobj{X}$ in $\cat{C}$ which satisfies the displayed equalities.
    It remains to show that the functor $\sobj{X}$ is strict monoidal.
    We have $X_m \ot X_n = \otfold{m} \ot \otfold{n} = \otfoldv{m + n} = X_{m + n}$ for all $m, n \ge 0$ and $X_0 = \otfold{0} = \unit$ by definition.
    Note that the set $\{(\phi, \phi') \mid X(\phi) \ot X(\phi') = X(\phi \ot \phi')\}$ contains all the pairs of the identity morphisms, and is closed under pointwise composition.
    By Remark~\ref{rem:generate-augbr}, it is enough to verify that $X_m \ot X(\phi) = X(\gene{m} \ot \phi)$ and $X(\phi) \ot X_m = X(\phi \ot \gene{m})$ hold for all $m \ge 0$ and all the generating morphisms $\phi$ of $\augbr$.
    We only prove the former equality; the latter can be proven similarly.
    For $0 \le j \le n$, we have
    \begin{align*}
        X_m \ot X(\coface{j}{n}) 
        &= \otfold{m} \ot \otfold{j} \ot \eps \ot \otfoldv{n - j} 
        = \otfoldv{m + j} \ot \eps \ot \otfoldv{n - j} 
        = X(\coface{m + j}{m + n}) 
        = X(\gene{m} \ot \coface{j}{n}); \\
        X_m \ot X(\codege{j}{n}) 
        &= \otfold{m} \ot \otfold{j} \ot \Delta \ot \otfoldv{n - j} 
        = \otfoldv{m + j} \ot \Delta \ot \otfoldv{n - j} 
        = X(\codege{m + j}{m + n}) 
        = X(\gene{m} \ot \codege{j}{n}),
    \end{align*}
    where we used conditions~\ref{enum:coaugbr-mon-1} and \ref{enum:coaugbr-mon-2} of Definition~\ref{defn:augbr} in the fourth equality, and for $1 \le i \le n$, we have
    \begin{align*}
        X_m \ot X(\coperm{i}{n}) 
        &= \otfold{m} \ot \otfoldv{i - 1} \ot \br \ot \otfoldv{n - i} 
        = \otfoldv{m + i - 1} \ot \br \ot \otfoldv{n - i} 
        = X(\coperm{m + i}{m + n}) 
        = X(\gene{m} \ot \coperm{i}{n}),
    \end{align*}
    where we used condition~\ref{enum:coaugbr-mon-3} of Definition~\ref{defn:augbr} in the fourth equality.
    The case $\phi = (\coperm{i}{n})^{-1}$ can be verified similarly, and it can also be deduced from the case $\phi = \coperm{i}{n}$, since the set mentioned above is closed under taking pointwise inverses.
    This completes the proof.
\end{proof}

\begin{rem}
    Let $\cat{C}$ be a strict monoidal category.
    The assignment $\brcoalg{C} \mapsto \sobj{X}$ of Proposition~\ref{prop:brcoalg-to-ABSO} induces a bijective correspondence between braided coalgebras in $\cat{C}$ and strict monoidal {\ABSO}s in $\cat{C}$.
    We note that $\augbr$ is equipped with a unique braiding $b$ such that $b_{\gene{1}, \gene{1}} = \copermvar$ (see \cite[Definition 4.2.1]{angelini-knoll-chan-gerhardt-merling-peroux}).
    Given a braiding $\br$ on $\cat{C}$, we obtain as a restriction a bijective correspondence between braided coalgebras of the form $(C, \Delta, \eps, \braid{C}{C})$ in $\cat{C}$ and strict braided monoidal functors from $(\op{\augbr}, b)$ to $\brmoncat{C}$, and the former objects can be identified with coalgebras in $\cat{C}$.
    It follows that strict braided monoidal functors from $(\op{\augbr}, b)$ to $\brmoncat{C}$ correspond bijectively to coalgebras in $\cat{C}$.
    This fact appears, for example, in \cite[Observation 4.2.2]{angelini-knoll-chan-gerhardt-merling-peroux}.
\end{rem}

Let $\func[\imath]{\cat{C}}{\wt{\cat{C}}}$ be a monoidal functor, and suppose that $\wt{\cat{C}}$ is strict.
Given a braided coalgebra $\brobj{C}$ in $\cat{C}$, one might attempt to construct a monoidal {\ABSO} in $\wt{\cat{C}}$, either by
\begin{enumerate}[label=(\alph*)]
	\item passing $\brobj{C}$ to a braided coalgebra in $\wt{\cat{C}}$ via $\imath$, and then using Proposition~\ref{prop:brcoalg-to-ABSO}, or
    \item formulating Proposition~\ref{prop:brcoalg-to-ABSO} in the non-strict setting to pass $\brobj{C}$ to a (possibly non-strict) monoidal {\ABSO} in $\cat{C}$, and then postcomposing it with $\imath$.
\end{enumerate}
As it would be tedious to take the second approach, we prove that a monoidal functor transfers braided coalgebras living in different monoidal categories.

\begin{prop}[{\cite[Proposition 2.5]{ardizzoni-menini}}]
    \label{prop:brcoalg-monfct}
    Let $\monfctvar{F}{C}{D}$ be a monoidal functor, and $\brcoalg{C}$ a braided coalgebra in $\cat{C}$.
	Then, $\brcoalgv{F(C)}$ is a braided coalgebra in $\cat{D}$, where we set
    \[\Delta' \coloneqq \monfctmul{F}{C}{C}^{-1} \circ F(\Delta); 
	\quad \eps' \coloneqq \monfctuni{F}^{-1} \circ F(\eps); 
	\quad \br' \coloneqq \monfctmul{F}{C}{C}^{-1} \circ F(\br) \circ \monfctmul{F}{C}{C}.\]
\end{prop}

\begin{proof}
	Since $\coalg{C}$ is a coalgebra in $\cat{C}$ and since $(F, \monfctmulvar{F}^{-1}, \monfctuni{F}^{-1})$ is an opmonoidal functor $\fct{C}{D}$, it follows from 
    \cite[Proposition 2.16]{bulacu-caenepeel-panaite-van-oystaeyen} that $\coalgv{F(C)}$ is a coalgebra in $\cat{D}$.
	On the other hand, since $\brobj{C}$ is a braided object in $\cat{C}$ and since $\monfct{F}$ is a monoidal functor, it follows from \cite[Lemma XIII.3.2]{kassel} that $\brobjv{F(C)}$ is a braided object in $\cat{D}$.
	We prove that the quadruple $\brcoalgv{F(C)}$ satisfies (\ref{eq:brcoalg-1}) and (\ref{eq:brcoalg-2}); the other equalities can be proven dually.
	We omit the symbol $\circ$.
    The equality (\ref{eq:brcoalg-1}) is rewritten as
	\[\lucv{F(C)} (\monfctuni{F}^{-1} \ot F(C)) (F(\eps) \ot F(C)) \monfctmul{F}{C}{C}^{-1} F(\br) \monfctmul{F}{C}{C} = \rucv{F(C)} (F(C) \ot \monfctuni{F}^{-1}) (F(C) \ot F(\eps)).\]
	Since $F(C) \ot F(\eps) = \monfctmul{F}{C}{\unit}^{-1} F(C \ot \eps) \monfctmul{F}{C}{C}$ by (\ref{eq:monfct-0}), it suffices to prove that
    \[\lucv{F(C)} (\monfctuni{F}^{-1} \ot F(C)) (F(\eps) \ot F(C)) \monfctmul{F}{C}{C}^{-1} F(\br) = \rucv{F(C)} (F(C) \ot \monfctuni{F}^{-1}) \monfctmul{F}{C}{\unit}^{-1} F(C \ot \eps),\]
	which using $(F(\eps) \ot F(C)) \monfctmul{F}{C}{C}^{-1} = \monfctmul{F}{\unit}{C}^{-1} F(\eps \ot C)$ (which again follows from (\ref{eq:monfct-0})) is rewritten as
	\[\lucv{F(C)} (\monfctuni{F}^{-1} \ot F(C)) \monfctmul{F}{\unit}{C}^{-1} F(\eps \ot C) F(\br) = \rucv{F(C)} (F(C) \ot \monfctuni{F}^{-1}) \monfctmul{F}{C}{\unit}^{-1} F(C \ot \eps).\]
	Since $\lucv{F(C)} = F(\luc{C}) \monfctmul{F}{\unit}{C} (\monfctuni{F} \ot F(C))$ by (\ref{eq:monfct-2-l}) and since $\rucv{F(C)} = F(\ruc{C}) \monfctmul{F}{C}{\unit} (F(C) \ot \monfctuni{C})$ by (\ref{eq:monfct-2-r}), this equality is rewritten as $F(\luc{C}) F(\eps \ot C) F(\br) = F(\ruc{C}) F(C \ot \eps)$, which follows from (\ref{eq:brcoalg-1}) for the braided coalgebra $\brcoalg{C}$.
    On the other hand, the equality (\ref{eq:brcoalg-2}) is rewritten as
    \begin{align*}
        & \asscv{F(C)}{F(C)}{F(C)} (\monfctmul{F}{C}{C}^{-1} \ot F(C)) (F(\Delta) \ot F(C)) \monfctmul{F}{C}{C}^{-1} F(\br) \monfctmul{F}{C}{C} \\
        & \qquad = (F(C) \ot \monfctmul{F}{C}{C}^{-1}) (F(C) \ot F(\br)) (F(C) \ot \monfctmul{F}{C}{C}) \asscv{F(C)}{F(C)}{F(C)} (\monfctmul{F}{C}{C}^{-1} \ot F(C)) \\
        & \qquad \phantom{{} = {}} (F(\br) \ot F(C)) (\monfctmul{F}{C}{C} \ot F(C)) \asscvinv{F(C)}{F(C)}{F(C)} (F(C) \ot \monfctmul{F}{C}{C}^{-1}) (F(C) \ot F(\Delta)).
    \end{align*}
    Since $F(C) \ot F(\Delta) = \monfctmul{F}{C}{C \ot C}^{-1} F(C \ot \Delta) \monfctmul{F}{C}{C}$ by (\ref{eq:monfct-0}), it suffices to prove that
    \begin{align*}
        & \asscv{F(C)}{F(C)}{F(C)} (\monfctmul{F}{C}{C}^{-1} \ot F(C)) (F(\Delta) \ot F(C)) \monfctmul{F}{C}{C}^{-1} F(\br) \\
        & \qquad = (F(C) \ot \monfctmul{F}{C}{C}^{-1}) (F(C) \ot F(\br)) (F(C) \ot \monfctmul{F}{C}{C}) \asscv{F(C)}{F(C)}{F(C)} (\monfctmul{F}{C}{C}^{-1} \ot F(C)) \\
        & \qquad \phantom{{} = {}} (F(\br) \ot F(C)) (\monfctmul{F}{C}{C} \ot F(C)) \asscvinv{F(C)}{F(C)}{F(C)} (F(C) \ot \monfctmul{F}{C}{C}^{-1}) \monfctmul{F}{C}{C \ot C}^{-1} F(C \ot \Delta),
    \end{align*}
    which using $(F(\Delta) \ot F(C)) \monfctmul{F}{C}{C}^{-1} = \monfctmul{F}{C \ot C}{C}^{-1} F(\Delta \ot C)$, $F(\br) \ot F(C) = \monfctmul{F}{C \ot C}{C}^{-1} F(\br \ot C) \monfctmul{F}{C \ot C}{C}$, and $F(C) \ot F(\br) = \monfctmul{F}{C}{C \ot C}^{-1} F(C \ot \br) \monfctmul{F}{C}{C \ot C}$ (which again follow from (\ref{eq:monfct-0})), is rewritten as
    \begin{align*}
        & \asscv{F(C)}{F(C)}{F(C)} (\monfctmul{F}{C}{C}^{-1} \ot F(C)) \monfctmul{F}{C \ot C}{C}^{-1} F(\Delta \ot C) F(\br) \\
        & \qquad = (F(C) \ot \monfctmul{F}{C}{C}^{-1}) \monfctmul{F}{C}{C \ot C}^{-1} F(C \ot \br) \monfctmul{F}{C}{C \ot C} (F(C) \ot \monfctmul{F}{C}{C}) \\
        & \qquad \phantom{{} = {}} \asscv{F(C)}{F(C)}{F(C)} (\monfctmul{F}{C}{C}^{-1} \ot F(C)) \monfctmul{F}{C \ot C}{C}^{-1} F(\br \ot C) \monfctmul{F}{C \ot C}{C} \\
        & \qquad \phantom{{} = {}} (\monfctmul{F}{C}{C} \ot F(C)) \asscvinv{F(C)}{F(C)}{F(C)} (F(C) \ot \monfctmul{F}{C}{C}^{-1}) \monfctmul{F}{C}{C \ot C}^{-1} F(C \ot \Delta).
    \end{align*}
    On the other hand, since (\ref{eq:monfct-1}) implies that
    \[\asscv{F(C)}{F(C)}{F(C)} (\monfctmul{F}{C}{C}^{-1} \ot F(C)) \monfctmul{F}{C \ot C}{C}^{-1} = (F(C) \ot \monfctmul{F}{C}{C}^{-1}) \monfctmul{F}{C}{C \ot C}^{-1} F(\assc{C}{C}{C}),\]
    it suffices to prove that
    \begin{equation}
        \label{eq:prop}
        \begin{aligned}
        & F(\assc{C}{C}{C}) F(\Delta \ot C) F(\br) \\
        & \qquad = F(C \ot \br) \monfctmul{F}{C}{C \ot C} (F(C) \ot \monfctmul{F}{C}{C}) \asscv{F(C)}{F(C)}{F(C)} \\
        & \qquad \phantom{{} = {}} (\monfctmul{F}{C}{C}^{-1} \ot F(C)) \monfctmul{F}{C \ot C}{C}^{-1} F(\br \ot C) \monfctmul{F}{C \ot C}{C} \\
        & \qquad \phantom{{} = {}} (\monfctmul{F}{C}{C} \ot F(C)) \asscvinv{F(C)}{F(C)}{F(C)} (F(C) \ot \monfctmul{F}{C}{C}^{-1}) \monfctmul{F}{C}{C \ot C}^{-1} F(C \ot \Delta).
        \end{aligned}
    \end{equation}
    Using (\ref{eq:monfct-1}) again, we can deduce the following equalities:
    \begin{align*}
        \monfctmul{F}{C}{C \ot C} (F(C) \ot \monfctmul{F}{C}{C}) \asscv{F(C)}{F(C)}{F(C)} 
        &= F(\assc{C}{C}{C}) \monfctmul{F}{C \ot C}{C} (\monfctmul{F}{C}{C} \ot F(C)); \\
        \monfctmul{F}{C \ot C}{C} (\monfctmul{F}{C}{C} \ot F(C)) \asscvinv{F(C)}{F(C)}{F(C)} 
        &= F(\assc{C}{C}{C})^{-1} \monfctmul{F}{C}{C \ot C} (F(C) \ot \monfctmul{F}{C}{C}).
    \end{align*}
    By substituting them, the right hand side of (\ref{eq:prop}) is simplified as
    \begin{align*}
        & F(C \ot \br) F(\assc{C}{C}{C}) \monfctmul{F}{C \ot C}{C} (\monfctmul{F}{C}{C} \ot F(C)) (\monfctmul{F}{C}{C}^{-1} \ot F(C)) \monfctmul{F}{C \ot C}{C}^{-1} \\
        & \quad F(\br \ot C) F(\assc{C}{C}{C})^{-1} \monfctmul{F}{C}{C \ot C} (F(C) \ot \monfctmul{F}{C}{C}) (F(C) \ot \monfctmul{F}{C}{C}^{-1}) \monfctmul{F}{C}{C \ot C}^{-1} F(C \ot \Delta).
    \end{align*}
    Therefore, (\ref{eq:prop}) is rewritten as 
    $F(\assc{C}{C}{C}) F(\Delta \ot C) F(\br) = F(C \ot \br) F(\assc{C}{C}{C}) F(\br \ot C) F(\assc{C}{C}{C}^{-1}) F(C \ot \Delta)$, which follows from (\ref{eq:brcoalg-2}) for the braided coalgebra $\brcoalg{C}$.
    This completes the proof.
\end{proof}

Let $\cat{C}$ be a monoidal category.
The (large) category $\Fun{C}{C}$ of endofunctors of $\cat{C}$ admits a strict monoidal structure $(\ordot, \id{\cat{C}})$ such that $F \ot G \coloneqq F \circ G$ and $\eta \ot \theta \coloneqq \eta \star \theta$ for all functors $\fct[F, F', G, G']{C}{C}$ and all natural transformations $\func[\eta]{F}{G}$ and $\func[\theta]{F'}{G'}$.
The assignment $X \mapsto X \ot (-)$ induces a monoidal functor $\func[\monfct{\imath}]{\cat{C}}{\Fun{C}{C}}$, where $\monfctmul{\imath}{X}{Y} = \assc{X}{Y}{-}^{-1}$ for all $X, Y \in \cat{C}$ and $\monfctuni{\imath} = \lucvar^{-1}$ (see \cite[Proposition 1.28]{bulacu-caenepeel-panaite-van-oystaeyen}).
Combining Propositions~\ref{prop:brcoalg-to-ABSO} and \ref{prop:brcoalg-monfct} in this situation, we obtain a (strict) monoidal {\ABSO} from a braided coalgebra in a possibly non-strict monoidal category $\cat{C}$.

\begin{prop}
    \label{prop:nonstr-augbrobj}
	Let $\cat{C}$ be a monoidal category, and $\brobj{C} = \brcoalg{C}$ a braided coalgebra in $\cat{C}$.
	Then, there exists a unique strict monoidal {\ABSO} $\wTcalC$ in $\Fun{C}{C}$ which satisfies the following equalities\emph{:}
    \begin{equation}
        \label{eq:wTcalC}
        \begin{alignedat}{2}
            \wTcalC[1] &= C \ot (-); 
            & \qquad \wTcalC[](\cofacevar) &= \lucvar \circ (\eps \ot (-)); \\
            \wTcalC[](\codegevar) &= \assc{C}{C}{-} \circ (\Delta \ot (-)); 
            & \qquad \wTcalC[](\copermvar) &= \assc{C}{C}{-} \circ (\br \ot (-)) \circ \assc{C}{C}{-}^{-1}.
        \end{alignedat}
    \end{equation}
\end{prop}

\begin{proof}
    Applying Proposition~\ref{prop:brcoalg-monfct} to the monoidal functor $\func[\imath]{\cat{C}}{\Fun{C}{C}}$ gives rise to a braided coalgebra $\brcoalgv{\imath(C)}$ in $\Fun{C}{C}$, where $\imath(C) = C \ot (-)$, and the morphisms $\Delta'$, $\eps'$, and $\br'$ are defined as follows: 
    \begin{align*}
		\Delta' &= \monfctmul{\imath}{C}{C}^{-1} \circ \imath(\Delta) 
        = \assc{C}{C}{-} \circ (\Delta \ot (-)); \\
        \eps' &= \monfctuni{\imath}^{-1} \circ \imath(\eps) 
        = \lucvar \circ (\eps \ot (-)); \\
		\br' &= \monfctmul{\imath}{C}{C}^{-1} \circ \imath(\br) \circ \monfctmul{\imath}{C}{C} 
        = \assc{C}{C}{-} \circ (\br \ot (-)) \circ \assc{C}{C}{-}^{-1}.
	\end{align*}
    Using Proposition~\ref{prop:brcoalg-to-ABSO}, we can construct $\wTcalC$ from this braided coalgebra, as desired.
\end{proof}

\begin{rem}
    \label{rem:strict}
    Let $\cat{C}$ be a monoidal category, and $\brobj{C}$ a braided coalgebra in $\cat{C}$.
    \begin{enumerate}
        \item When $\cat{C}$ is strict, Proposition~\ref{prop:nonstr-augbrobj} yields nothing new.
        Indeed, the monoidal functor $\func[\imath]{\cat{C}}{\Fun{C}{C}}$ is strict in this case, and $\wTcalC$ is just the composition $\op{\augbr} \xrightarrow{\sobj{X}} \cat{C} \longxrightarrow{\imath} \Fun{C}{C}$, where $\sobj{X}$ is constructed from $\brobj{C}$ as in Proposition~\ref{prop:brcoalg-to-ABSO}.
        \item \label{enum:factor} The functor $\func[\wTcalC]{\op{\augbr}}{\Fun{C}{C}}$ factors through the evident functor $\func{\op{\augbr}}{\op{\augsymm}}$ of Remark~\ref{rem:augsymm}~\ref{enum:symmcat} if and only if $\wTcalC[](\copermvar)^2 = \id{}$ if and only if $\br^2 = \id{}$.
    \end{enumerate}
\end{rem}

\subsection{Braided cochain complex and braided cohomology}
\label{subsec:3.3}

In this subsection, we define the braided cochain complex and the braided cohomology of a braided coalgebra in a $\Bbbk$-linear monoidal category, using the associated strict monoidal {\ABSO} constructed in Proposition~\ref{prop:nonstr-augbrobj}.

\begin{nota}
    \label{nota:V}
    We let $\cat{V}$ denote the category of $\Bbbk$-modules.
\end{nota}

\begin{defn}
    An \term{{\ABcosSO} in $\cat{V}$} is a functor $\func{\augbrsemi}{\cat{V}}$.
\end{defn}

\begin{nota}
    An {\ABcosSO} in $\cat{V}$ is typically denoted by $\cosobj{K}$ or $K$.
    In these cases, we write $K^n$ for its value on $\gene{n}$ for all $n \ge 0$.
\end{nota}

Let $\cosobj{K}$ be an {\ABcosSO} in $\cat{V}$.
The collection $\famnat{\Kcpx[n]}$ of objects in $\cat{V}$ gives rise to a cochain complex in $\cat{V}$ by letting the differential $\func[\codiff[n]]{\Kcpx[n]}{K^{n + 2}}$ be the alternating sum $\displaystyle\sum_{j = 0}^{n + 1} (-1)^j K(\coface{j}{n + 1})$ for all $n \ge 0$.
We will write $\Kcpx$ for this cochain complex in $\cat{V}$.
We refer to Lurie \cite[\ktag{Construction}{00QA} and \ktag{Remark}{04RW}]{lurie} for a dual construction.

\begin{defn}
    \label{defn:KB}
    Let $\cosobj{K}$ be an {\ABcosSO} in $\cat{V}$.
    For all $n \ge 0$, we define a $\Bbbk$-submodule $\KBvar[n]{K}$ of $\Kcpx[n]$ by $\KBvar[0]{K} \coloneqq K^1$ and
    \begin{equation*}
        \KBvar[n]{K} \coloneqq \{\varphi \in \Kcpx[n] \mid K(\coperm{i}{n})(\varphi) = - \varphi \text{ for all } 1 \le i \le n\} \quad \text{for all $n > 0$.}
    \end{equation*}
\end{defn}

\begin{rem}
    \label{rem:sign}
    Let $\cosobj{K}$ be an {\ABcosSO} in $\cat{V}$, and $n \ge 0$ a fixed integer.
    We can define a representation of the braid group $\brgrp$ on the $\Bbbk$-module $\Kcpx[n]$ as the composition
    \[\brgrp \xrightarrow{\rho_{n + 1}} \Aut{\augbrsemi}{\gene{n + 1}} \overset{\cosobj{K}}{\longrightarrow} \Aut{\cat{V}}{\Kcpx[n]},\]
    where $\rho_{n + 1}$ is defined as in Remark~\ref{rem:augbr}~\ref{enum:rem-augbr-2}.
    Let $\func[\sgn]{\brgrp}{\Bbbk^\times}$ denote the unique group representation such that $\sgn[\brperm{i}{n}] = -1$ for all $1 \le i \le n$; said differently, $\sgn$ is just the composition of the ordinary sign representation of the symmetric group $\symmgrp$ and the quotient map $\func{\brgrp}{\symmgrp}$.
    By considering their tensor product, we see that $\brgrp$ acts on the $\Bbbk$-module $\Kcpx[n]$ as $(g, \varphi) \mapsto g {\myspace} \varphi \coloneqq \sgn[g] (K \circ \rho_{n + 1})(g)(\varphi)$.
    We claim that $\KBvar[n]{K}$ coincides with the $\Bbbk$-submodule $(\Kcpx[n])^{\brgrp}$ of $\brgrp$-invariants:
    \begin{equation}
        \label{eq:invar}
        \KBvar[n]{K} = (\Kcpx[n])^{\brgrp} \coloneqq \{\varphi \in \Kcpx[n] \mid g {\myspace} \varphi = \varphi \text{ for all } g \in \brgrp\}.
    \end{equation}
    Since the group $\brgrp$ is generated by $\{\brperm{i}{n} \mid 1 \le i \le n\}$, it follows that
    \[(\Kcpx[n])^{\brgrp} = \{\varphi \in \Kcpx[n] \mid \brperm{i}{n} \varphi = \varphi \text{ for all } 1 \le i \le n\}.\]
    The desired equality (\ref{eq:invar}) follows since $\brperm{i}{n} \varphi = \sgn[\brperm{i}{n}] (K \circ \rho_{n + 1})(\brperm{i}{n})(\varphi) = - K(\coperm{i}{n})(\varphi)$ for all $1 \le i \le n$.
\end{rem}

\begin{prop}
    \label{prop:KB}
    Let $\cosobj{K}$ be an {\ABcosSO} in $\cat{V}$.
    Then, the collection $\famnat{\KBvar[n]{K}}$ of $\Bbbk$-modules determines a subcomplex $\KBvar{K}$ of $\Kcpx$.
\end{prop}

\begin{proof}
    Take $n > 0$ and $\varphi \in \KBvar[n - 1]{K}$.
    For all $1 \le i \le n$, we can compute
    \begin{align*}
        & K(\coperm{i}{n})(\codiff[n - 1](\varphi)) 
        = K(\coperm{i}{n})\biggl( \, \sum_{j = 0}^n (-1)^j K(\coface{j}{n})(\varphi) \biggr) 
        = \sum_{j = 0}^n (-1)^j K(\coperm{i}{n} \circ \coface{j}{n})(\varphi) \\
        & \quad = \sum_{j < i - 1} (-1)^j K(\coperm{i}{n} \circ \coface{j}{n})(\varphi) 
        + (-1)^{i - 1} K(\coperm{i}{n} \circ \coface{i - 1}{n})(\varphi) 
        + (-1)^i K(\coperm{i}{n} \circ \coface{i}{n})(\varphi) 
        + \sum_{j > i} (-1)^j K(\coperm{i}{n} \circ \coface{j}{n})(\varphi) \\
        & \quad = \sum_{j < i - 1} (-1)^j K(\coface{j}{n} \circ \coperm{i - 1}{n - 1})(\varphi) 
        + (-1)^{i - 1} K(\coface{i}{n})(\varphi) 
        + (-1)^i K(\coface{i - 1}{n})(\varphi) 
        + \sum_{j > i} (-1)^j K(\coface{j}{n} \circ \coperm{i}{n - 1})(\varphi) \\
        & \quad = \sum_{j < i - 1} (-1)^j K(\coface{j}{n})(- \varphi) 
        + (-1)^{i - 1} K(\coface{i}{n})(\varphi) 
        + (-1)^i K(\coface{i - 1}{n})(\varphi) 
        + \sum_{j > i} (-1)^j K(\coface{j}{n})(- \varphi) \\
        & \quad = - \sum_{j = 0}^n (-1)^j K(\coface{j}{n})(\varphi) 
        = - \codiff[n - 1](\varphi),
    \end{align*}
    where we used condition~\ref{enum:coaugbr-3} of Definition~\ref{defn:augbr} in the fourth equality, and the fifth equality follows from the assumption.
    This implies $\codiff[n - 1](\varphi) \in \KBvar[n]{K}$.
    Consequently, we have $\codiff[n - 1](\KBvar[n - 1]{K}) \subset \KBvar[n]{K}$.
\end{proof}

This construction $\cosobj{K} \mapsto \KBvar{K}$ is functorial in $K$; the next result is sufficient for our purposes.
Given a pair of {\ABcosSO}s $\cosobj{K}$ and $\cosobj{L}$ in $\cat{V}$, a natural transformation $\func{\cosobj{K}}{\cosobj{L}}$ is denoted by $\cosf$, and we write $\fn$ for its component $\func{K^n}{L^n}$ at $\gene{n}$ for all $n \ge 0$.

\begin{lem}
    \label{lem:KB}
    Let $\cosobj{K}$ and $\cosobj{L}$ be {\ABcosSO}s in $\cat{V}$, and $\func[\cosf]{\cosobj{K}}{\cosobj{L}}$ a natural transformation.
    Then, the collection $\famnat{\func[\fcpx[n]]{\Kcpx[n]}{\Lcpx[n]}}$ of $\Bbbk$-linear maps induces a chain map $\func{\KBvar{K}}{\KBvar{L}}$, which is an isomorphism if $\cosf$ is a natural isomorphism.
\end{lem}

\begin{proof}
    We first claim that the collection $\fcpx = \famnat{\fcpx[n]}$ 
    determines a chain map $\func{\Kcpx}{\Lcpx}$.
    Write $\codiffvar{K}$, and $\codiffvar{L}$ for the differentials of $\Kcpx$, and $\Lcpx$, respectively.
    It follows from the naturality of $\cosf$ that
    \[\fcpx[n] \circ \codiffvar{K}[n - 1] 
    = \sum_{j = 0}^n (-1)^j (\fcpx[n] \circ K(\coface{j}{n})) 
    = \sum_{j = 0}^n (-1)^j (L(\coface{j}{n}) \circ \fn) 
    = \codiffvar{L}[n - 1] \circ \fn\]
    for all $n > 0$, as desired.
    Next, we claim that $\fcpx[n](\KBvar[n]{K}) \subset \KBvar[n]{L}$ for all $n \ge 0$.
    Since $\KBvar[0]{K} = K^1$ and $\KBvar[0]{L} = L^1$ by definition, we may assume that $n > 0$.
    For all $\varphi \in \KBvar[n]{K}$ and $1 \le i \le n$, it follows that $L(\coperm{i}{n}) (\fcpx[n](\varphi)) = \fcpx[n](K(\coperm{i}{n}) (\varphi)) 
    = - \fcpx[n](\varphi)$, where the first equality follows from the naturality of $\cosf$, 
    and the second equality from the assumption, and thus $\fcpx[n](\varphi) \in \KBvar[n]{L}$.
    We therefore obtain a chain map $\func{\KBvar{K}}{\KBvar{L}}$.
    Take $n \ge 0$.
    Since $\fcpx[n]$ is an isomorphism of $\Bbbk$-modules, its restriction $\func{\KBvar[n]{K}}{\KBvar[n]{L}}$ is injective.
    On the other hand, for all $\psi \in \KBvar[n]{L}$, the surjectivity of $\fcpx[n]$ implies that $\psi = \fcpx[n](\varphi)$ for some $\varphi \in \Kcpx[n]$, and a similar computation gives $\fcpx[n](K(\coperm{i}{n}) (\varphi)) = L(\coperm{i}{n}) (\fcpx[n](\varphi)) = L(\coperm{i}{n}) (\psi) = - \psi = \fcpx[n](- \varphi)$ for all $1 \le i \le n$, and thus $\varphi \in \KBvar[n]{K}$ by the injectivity of $\fcpx[n]$.
    It follows that the $\Bbbk$-linear map $\func{\KBvar[n]{K}}{\KBvar[n]{L}}$ is surjective as well.
    Hence, the chain map $\func{\KBvar{K}}{\KBvar{L}}$ is an isomorphism.
\end{proof}

\begin{defn}
    \label{defn:br-cpx}
    Let $\cat{C}$ be a $\Bbbk$-linear monoidal category, $\brobj{C}$ a braided coalgebra in $\cat{C}$, and $M$ an object in $\cat{C}$.
    \begin{enumerate}
        \item \label{br-cpx:ABSO} We define an {\ABSO} $\wTC$ in $\cat{C}$ to be 
        $\evunit \circ \wTcalC$.
        \item We define an {\ABcosSO} $\KC$ in $\cat{V}$ to be the composite
        \[\augbrsemi \longxrightarrow{\incl} \augbr \xrightarrow{\wTC} \opC \xrightarrow{\linhom{C}{-}{M}} \cat{V}.\]
        \item We define the \textit{braided cochain complex $\KBC$ of $\brobj{C}$ with coefficients in $M$} to be the cochain complex $\KBvar{\KC[]}$ in $\cat{V}$, constructed from 
        $\KC$ as in Definition~\ref{defn:KB}.
        \item \label{br-cpx:br-coho} We write $\HBC$ for the cohomology of the cochain complex $\KBC$, 
        and refer to it as the \term{braided cohomology of $\brobj{C}$ with coefficients in $M$}.
    \end{enumerate}
\end{defn}

We now describe the {\ABSO} $\wTC$ explicitly.

\begin{rem}
    \label{rem:wTC}
    Let $\cat{C}$ be a $\Bbbk$-linear monoidal category, and $\brobj{C}$ a braided coalgebra in $\cat{C}$.
    For simplicity, we write $\sobjcal{X} \coloneqq \wTcalC$.
    We note that $\wTC[0] = \Xcal_0(\unit) = \id{\cat{C}}(\unit) = \unit$ and for all $n > 0$,
    \[\wTC[n] = \Xcal_n(\unit) 
    = \Xcal(\otfold[\gene{1}]{n})(\unit) 
    = ((\Xcal_1)^{\circ {\myspace} n})(\unit) 
    = \otfold{n} \ot \unit,\]
    which is isomorphic to $\otfold{n}$ via $\otfoldv{n - 1} \ot \ruc{C}$.
    We identify $\wTC$ with the {\ABSO} in $\cat{C}$ determined from the collection $\famnat{\otfold{n}}$ of objects in $\cat{C}$, together with the collections
    \[\family{\func[\face{j}{n}]{\otfoldv{n + 1}}{\otfold{n}}}{0 \le j \le n}; 
    \quad \family{\func[\dege{j}{n}]{\otfoldv{n + 1}}{\otfoldv{n + 2}}}{0 \le j \le n}; 
    \quad \family{\func[\perm{i}{n}]{\otfoldv{n + 1}}{\otfoldv{n + 1}}}{1 \le i \le n}\]
    of morphisms in $\cat{C}$, which are defined as follows:
    \begin{itemize}
        \item The morphism $\face{0}{0}$ is defined to be the composite $C \xrightarrow{\ruc{C}^{-1}} C \ot \unit \xrightarrow{\Xcal(\cofacevar)_{\unit}} \unit$.
        \item For all $n > 0$ and $0 \le j \le n$, the morphism $\face{j}{n}$ is defined to be the composite
        \[\otfoldv{n + 1} \xrightarrow{(\otfold{n} \ot \ruc{C})^{-1}} \otfoldv{n + 1} \ot \unit \xrightarrow{\Xcal(\coface{j}{n})_{\unit}} \otfold{n} \ot \unit \xrightarrow{\otfoldv{n - 1} \ot \ruc{C}} \otfold{n}.\]
        \item For all $0 \le j \le n$, the morphism $\dege{j}{n}$ is defined to be the composite
        \[\otfoldv{n + 1} \xrightarrow{(\otfold{n} \ot \ruc{C})^{-1}} \otfoldv{n + 1} \ot \unit \xrightarrow{\Xcal(\codege{j}{n})_{\unit}} \otfoldv{n + 2} \ot \unit \xrightarrow{\otfoldv{n + 1} \ot \ruc{C}} \otfoldv{n + 2}.\]
        \item For all $1 \le i \le n$, the morphism $\perm{i}{n}$ is defined to be the composite
        \[\otfoldv{n + 1} \xrightarrow{(\otfold{n} \ot \ruc{C})^{-1}} \otfoldv{n + 1} \ot \unit \xrightarrow{\Xcal(\coperm{i}{n})_{\unit}} \otfoldv{n + 1} \ot \unit \xrightarrow{\otfold{n} \ot \ruc{C}} \otfoldv{n + 1}.\]
    \end{itemize}
    We now simplify the above description of the morphism $\face{j}{n}$.
    First, the equality $\face{0}{0} = \eps$ holds since
    \begin{equation}
        \label{eq:face-0}
        \face{0}{0} 
        = \Xcal(\cofacevar)_{\unit} \circ \ruc{C}^{-1} 
        = \luc{\unit} \circ (\eps \ot \unit) \circ \ruc{C}^{-1} 
        = \luc{\unit} \circ \ruc{\unit}^{-1} \circ \eps = \eps,
    \end{equation}
    where the third equality follows from the naturality of $\rucvar$, and the fourth equality from \cite[Lemma XI.2.3]{kassel}.
    On the other hand, $\face{j}{n}$ is equal to the composite
    \[\otfoldv{n + 1} \xrightarrow{\otfold{j} \ot \eps \ot \otfoldv{n - j}} \otfold{j} \ot \unit \ot \otfoldv{n - j} \xrightarrow{\otfold{j} \ot \luc{\otfoldv{n - j}}} \otfold{n}\]
    for all $0 \le j < n$, and $\face{n}{n}$ is equal to the composite
    \[\otfoldv{n + 1} \xrightarrow{\otfold{n} \ot \eps} \otfold{n} \ot \unit \xrightarrow{\otfoldv{n - 1} \ot \ruc{C}} \otfold{n}\]
    for all $n > 0$.
    In order to prove these assertions, since
    \begin{align*}
        \Xcal(\coface{j}{n})_{\unit} 
        &= ((\Xcal_1)^{\circ {\myspace} j} \star \Xcal(\cofacevar) \star (\Xcal_1)^{\circ {\myspace} (n - j)})_{\unit} \\
        &= ((\Xcal_1)^{\circ {\myspace} j} \star (\lucvar \circ (\eps \ot (-))) \star (\Xcal_1)^{\circ {\myspace} (n - j)})_{\unit} 
        = (\otfold{j} \ot \luc{\otfoldv{n - j} \ot \unit}) \circ (\otfold{j} \ot \eps \ot \otfoldv{n - j} \ot \unit)
    \end{align*}
    for all $0 \le j \le n$, it suffices to prove the equalities
    \begin{align}
        (\otfoldv{n - 1} \ot \ruc{C}) \circ \luc{\otfold{n} \ot \unit} \circ (\eps \ot \otfold{n} \ot \unit) 
        &= \luc{\otfold{n}} \circ (\eps \ot \otfold{n}) \circ (\otfold{n} \ot \ruc{C}); \label{eq:face-1} \\
        (\otfoldv{n - 1} \ot \ruc{C}) \circ (\otfold{n} \ot \luc{\unit}) \circ (\otfold{n} \ot \eps \ot \unit) 
        &= (\otfoldv{n - 1} \ot \ruc{C}) \circ (\otfold{n} \ot \eps) \circ (\otfold{n} \ot \ruc{C}) \label{eq:face-2}
    \end{align}
    for all $n > 0$.
    By the naturality of $\lucvar$, (\ref{eq:face-1}) is verified as
    \begin{align*}
        (\otfoldv{n - 1} \ot \ruc{C}) \circ \luc{\otfold{n} \ot \unit} \circ (\eps \ot \otfold{n} \ot \unit) 
        &= \luc{\otfold{n}} \circ (\unit \ot \otfoldv{n - 1} \ot \ruc{C}) \circ (\eps \ot \otfold{n} \ot \unit) \\
        &= \luc{\otfold{n}} \circ (\eps \ot \otfoldv{n - 1} \ot \ruc{C}) 
        = \luc{\otfold{n}} \circ (\eps \ot \otfold{n}) \circ (\otfold{n} \ot \ruc{C}).
    \end{align*}
    On the other hand, (\ref{eq:face-2}) is reduced to the equality $\luc{\unit} \circ (\eps \ot \unit) = \eps \circ \ruc{C}$, which can be verified similarly to (\ref{eq:face-0}).
    The description of the morphism $\face{j}{n}$ can be summarized as follows:
    \begin{equation}
        \label{eq:face}
        \face{j}{n} = \begin{cases*}
            \eps & if $0 = j = n$, \\
            \otfold{j} \ot (\luc{\otfoldv{n - j}} \circ (\eps \ot \otfoldv{n - j})) & if $0 \le j < n$, \\
            \otfoldv{n - 1} \ot (\ruc{C} \circ (C \ot \eps)) & if $0 < j = n$.
        \end{cases*}
    \end{equation}
    The following formulas for the morphisms $\dege{j}{n}$ and $\perm{i}{n}$ can be deduced by parallel discussions:
    \begin{align}
        \label{eq:dege}
        \dege{j}{n} &= \begin{cases*}
            \otfold{j} \ot (\assc{C}{C}{\otfoldv{n - j}} \circ (\Delta \ot \otfoldv{n - j})) & if $0 \le j < n$, \\
            \otfold{n} \ot \Delta & if $0 \le j = n$;
        \end{cases*} \\
        \label{eq:perm}
        \perm{i}{n} &= \begin{cases*}
            \otfoldv{i - 1} \ot (\assc{C}{C}{\otfoldv{n - i}} \circ (\br \ot \otfoldv{n - i}) \circ \assc{C}{C}{\otfoldv{n - i}}^{-1}) & if $1 \le i < n$, \\
            \otfoldv{n - 1} \ot \br & if $1 \le i = n$.
        \end{cases*}
    \end{align}
    This implies that $\wTC$ coincides with 
    $\sobj{X}$, constructed as in Proposition~\ref{prop:brcoalg-to-ABSO}, when $\cat{C}$ is strict.
\end{rem}

We close this subsection by relating the above construction to symmetric groups in a special case.

\begin{rem}
    \label{rem:HS}
    Let $\cat{C}$ be a $\Bbbk$-linear monoidal category, $\brobj{C}$ a braided coalgebra in $\cat{C}$, and $M$ an object in $\cat{C}$, and suppose that $\br^2 = \id{}$.
    Then, similarly to Remark~\ref{rem:strict}~\ref{enum:factor}, the functor $\func[\KC]{\augbrsemi}{\cat{V}}$ factors through the evident functor $\func{\augbrsemi}{\augsymmsemi}$, where $\augsymmsemi$ is defined as in Remark~\ref{rem:augsymm}~\ref{enum:symmcat}, and thereby induces a representation $\func{\symmgrp}{\Aut{\cat{V}}{\KC[n + 1]}}$ of 
    $\symmgrp$ which makes the diagram
    \[\xymatrix{\brgrp \ar[d] \ar[rr]^-{\rho_{n + 1}} && \Aut{\cat{V}}{\KC[n + 1]} \ar[d]^-{\id{}} \\
    \symmgrp \ar @{-->} [rr] & & \Aut{\cat{V}}{\KC[n + 1]}}\]
    commutative for all $n \ge 0$.
    Arguing as in Remark~\ref{rem:sign}, $\KBC[n]$ can be regarded as the $\Bbbk$-submodule $\KC[n + 1]^{\symmgrp}$ of $\symmgrp$-invariants.
    In this case, the cochain complex $\KBC$ and its cohomology $\HBC$ are denoted by $\KSC$ and $\HSC$, which we refer to as the \textit{symmetric cochain complex} and the \textit{symmetric cohomology of $\brobj{C}$ with coefficients in $M$}, respectively.
\end{rem}

\section{Braided Morita invariance of braided cochain complex}
\label{sec:4}

Let $A$ and $B$ be $\Bbbk$-algebras.
We write $\AMod$ and $\BMod$ for the $\Bbbk$-linear categories of (left) $A$-modules and of $B$-modules, respectively (see Notation~\ref{nota:alg}~\ref{enum:AMod}).
Two $\Bbbk$-algebras $A$ and $B$ are said to be \term{Morita equivalent} if there exists a $\Bbbk$-linear equivalence $\func{\AMod}{\BMod}$.
When $A$ and $B$ are endowed with the structure of a $\Bbbk$-bialgebra (see Definition~\ref{defn:bialg}), $\AMod$ and $\BMod$ admit monoidal structures, and the $\Bbbk$-bialgebras $A$ and $B$ are said to be \term{monoidally Morita equivalent} (this terminology is borrowed from \cite[Definition 2.3]{shimizu}) if there exists a $\Bbbk$-linear monoidal equivalence $\func{\AMod}{\BMod}$, which is defined as follows:

\begin{defn}[{cf.~\cite[Definition 1.30]{bulacu-caenepeel-panaite-van-oystaeyen}}]
    \label{defn:lin-mon-eqv}
    A $\Bbbk$-linear monoidal functor $\fct[F]{C}{D}$ is said to be a \term{$\Bbbk$-linear monoidal equivalence} if there exist a $\Bbbk$-linear monoidal functor $\fct[G]{D}{C}$ as well as $\Bbbk$-linear monoidal natural isomorphisms $\func{\id{\cat{C}}}{G \circ F}$ and $\func{F \circ G}{\id{\cat{D}}}$.
\end{defn}

A \term{$\Bbbk$-linear braided monoidal equivalence} between $\Bbbk$-linear braided monoidal categories is defined in a similar way.
Two quasi-triangular $\Bbbk$-bialgebras $\qtb{A}$ and $\qtbv{B}$ are said to be \term{braided Morita equivalent} (see \cite[Definition 1.1]{muller-walton}) if there exists a $\Bbbk$-linear braided monoidal equivalence $\func{(\AMod, \braidtrivar)}{(\BMod, \braidtrivarv)}$ (for notation, see Subsection~\ref{subsec:4.3}).
In this section, we prove that the braided cochain complex of quasi-triangular $\Bbbk$-bialgebras (in the sense of Definition~\ref{defn:KBA}) is stable under certain $\Bbbk$-linear braided monoidal equivalences (see Theorem~\ref{thm:Morita} and Corollary~\ref{cor:Morita} for details).
We view this result as the (conditional) braided Morita invariance of the braided cochain complex of quasi-triangular $\Bbbk$-bialgebras.

We start with an observation in the general setting: Let $\fct[F]{C}{D}$ be a $\Bbbk$-linear monoidal equivalence, $\brobj{C}$ and $\brobjv{D}$ braided coalgebras in $\cat{C}$ and $\cat{D}$, respectively, and $M$ an object in $\cat{C}$.
We examine here the following condition:
\begin{itemize}
    \item[($*$)] There exists an isomorphism $\KC \iso \KD$ 
    for all $M \in \cat{C}$.
\end{itemize}
By Lemma~\ref{lem:KB}, this implies that $\KBC \iso \KBD$ as cochain complexes of $\Bbbk$-modules, and taking cohomology yields that $\HBC \iso \HBD$ as graded $\Bbbk$-modules, for all $M \in \cat{C}$.
Let us put $\sobjcalinj{X} \coloneqq \wTcalC \circ \incl$ and $\sobjcalinj{Y} \coloneqq \wTcalD \circ \incl$.
Take $M \in \cat{C}$ arbitrarily.
Since $F$ is fully faithful by the remark preceding 
\cite[Definition 3.5.13]{riehl}, $F$ induces an isomorphism $\linhom{C}{-}{M} \iso \linhom{D}{F(-)}{F(M)}$ by Lemma~\ref{lem:k-lin}, and pre-whiskering (see \cite[Definition 2.1.16]{johnson-yau}) with $\evunit \circ \sobjcalinj{X}$ gives an isomorphism
\[\KC = \linhom{C}{\evunit \circ \mathcal{X}'}{M} \iso \linhom{D}{F \circ \evunit \circ \mathcal{X}'}{F(M)}.\]
Since $\KD = \linhom{D}{\evunit \circ \mathcal{Y}'}{F(M)}$ by definition, condition ($*$) holds if there is an isomorphism $\evunit \circ \sobjcalinj{Y} \iso F \circ \evunit \circ \sobjcalinj{X}$.
By Lemma~\ref{lem:ev}~\ref{enum:ev-1}, we have $F \circ \evunit \circ \sobjcalinj{X} = \evunit \circ [\cat{C}, F] \circ \sobjcalinj{X}$.
On the other hand, using $\unit \iso F(\unit)$ and Lemma~\ref{lem:ev}~\ref{enum:ev-2}, we have $\evunit \circ \sobjcalinj{Y} \iso \ev{F(\unit)} \circ \sobjcalinj{Y} = \evunit \circ [F, \cat{D}] \circ \sobjcalinj{Y}$.
It thus follows that condition ($*$) holds if there is an isomorphism $[F, \cat{D}] \circ \sobjcalinj{Y} \iso [\cat{C}, F] \circ \sobjcalinj{X}$, or equivalently, if there is a collection $\famnat{\func[\vareta[n]]{\Ycal'_n \circ F}{F \circ \Xcal'_n}}$ of natural isomorphisms which satisfies (\ref{eq:mor-1}) below for all morphisms $\phi$ in $\augbrsemi$.
We will study natural transformations $\func{[F, \cat{D}] \circ \wTcalD}{[\cat{C}, F] \circ \wTcalC}$ by imposing some compatibility with respect to the (fixed) monoidal structures on $\cat{C}$, $\cat{D}$, and $F$.

\subsection{Relative morphisms between strict monoidal functors}
\label{subsec:4.1}

In this subsection, we let $\cat{I}$ denote either $\augbr$ or $\augbrsemi$, which we regard as a strict monoidal category, equipped with the monoidal structure described as in Definitions~\ref{defn:augbr} and \ref{defn:augbrsemi}.
We introduce here the notion of relative morphisms between strict monoidal {\ABbsbSO}s in strict monoidal categories of the form $\Fun{C}{C}$ for some category $\cat{C}$.
Then, we describe relative morphisms in terms of the evaluations of these functors at the morphisms $\cofacevar$, $\codegevar$, and $\copermvar$ in the case $\cat{I} = \augbr$ (see Proposition~\ref{prop:mor-3}).
We also obtain an intermediate result (see Proposition~\ref{prop:mor-2}), which deals with morphisms satisfying only (\ref{eq:mor-2}) and (\ref{eq:mor-3}).

\begin{defn}
	\label{defn:rel-mor}
    Let $\cat{C}$ and $\cat{D}$ be categories, and $\func[\sobjcal{X}]{\opI}{\Fun{C}{C}}$ and $\func[\sobjcal{Y}]{\opI}{\Fun{D}{D}}$ strict monoidal functors; that is, $\sobjcal{X}$ and $\sobjcal{Y}$ are strict monoidal {\ABbsbSO}s in $\Fun{C}{C}$ and $\Fun{D}{D}$, respectively.
    A \term{relative morphism} from $\sobjcal{X}$ to $\sobjcal{Y}$ is a pair $(F, \vareta)$ consisting of a functor $\fct[F]{C}{D}$ and a collection $\vareta = \famnat{\func[\vareta[n]]{\Ycal_n \circ F}{F \circ \Xcal_n}}$ of natural transformations, which satisfies the equalities
	\begin{align}
		\label{eq:mor-1} \vareta[m] \circ (\Ycal(\phi) \star F) &= (F \star \Xcal(\phi)) \circ \vareta[n]; \\
		\label{eq:mor-2} \varetazero &= \id{F}; \\
		\label{eq:mor-3} \vareta[m + n] &= (\vareta[m] \star \Xcal_n) \circ (\Ycal_m \star \vareta[n]),
	\end{align}
	where $m, n \ge 0$ and $\funcgene[\phi]{m}{n}$ is an arbitrary morphism in $\cat{I}$.
\end{defn}

Definition~\ref{defn:rel-mor} can be described in the language of $2$-category theory.
We refer to \cite{johnson-yau} for the definitions of some terminology which appear in the next remark.

\begin{rem}
	The strict monoidal category $\opI$ gives rise to a $2$-category $\opIvar$ with a single object $*$ where the $\mathrm{Hom}$-category of $1$-cells from $*$ to $*$ is equal to $\opI$ (see \cite[Examples 2.1.19 and 2.3.11]{johnson-yau}).
	We let $\biCAT$ denote the (huge) $2$-category of categories.
	Given a category $\cat{E}$, it follows from \cite[\ktag{Remark}{00F8}]{lurie} that there exists a bijective correspondence between strict monoidal functors $\func{\opI}{\Fun{E}{E}}$ and $2$-functors $\func{\opIvar}{\biCAT}$ which sends $*$ to $\cat{E}$.
	Using the notation of Definition~\ref{defn:rel-mor}, if we write $\deloop{\Xcal}$ and $\deloop{\Ycal}$ for the $2$-functors $\func{\opIvar}{\biCAT}$ that correspond to $\sobjcal{X}$ and $\sobjcal{Y}$, respectively, then, giving a relative morphism $\func{\sobjcal{X}}{\sobjcal{Y}}$ is equivalent to giving a \term{lax transformation} $\func{\deloop{\Xcal}}{\deloop{\Ycal}}$ in the sense of \cite[Definition 4.2.1]{johnson-yau}.
    The invertibility condition we wish to impose on a relative morphism $\func{\sobjcal{X}}{\sobjcal{Y}}$ to deduce condition ($*$) can be rephrased by saying that it corresponds to a \term{strong transformation} $\func{\deloop{\Xcal}}{\deloop{\Ycal}}$.
\end{rem}

Our terminology is justified by the following description in the absolute case:

\begin{rem}
	\label{rem:abso-mor}
	Let $\cat{C}$ be a monoidal category, $\func[\sobjcal{X}, \sobjcal{Y}]{\opI}{\Fun{C}{C}}$ strict monoidal functors.
    Then, for any collection $\vareta = \famnat{\func[\vareta[n]]{\Ycal_n}{\Xcal_n}}$ of natural transformations, the pair $(\id{\cat{C}}, \vareta)$ is a relative morphism $\func{\sobjcal{X}}{\sobjcal{Y}}$ if and only if $\vareta$ is a monoidal natural transformation $\func{\sobjcal{Y}}{\sobjcal{X}}$.
    Indeed, (\ref{eq:mor-1}) corresponds to the naturality of $\vareta$, and (\ref{eq:mor-2}) and (\ref{eq:mor-3}) to the extra conditions in Definition~\ref{defn:mon-nat-trans}.
\end{rem}

In the rest of this subsection, we prove that the collection $\vareta$ that appears in a relative morphism $\func[(F, \vareta)]{\sobjcal{X}}{\sobjcal{Y}}$ between strict monoidal {\ABSO}s can be recovered from $\vareta[1]$ subject to some conditions, using the description of the category $\augbr$ in Definition~\ref{defn:augbr}.

In what follows, we fix
\begin{itemize}
	\item categories $\cat{C}$ and $\cat{D}$,
	\item strict monoidal functors $\func[\sobjcal{X}]{\opI}{\Fun{C}{C}}$ and $\func[\sobjcal{Y}]{\opI}{\Fun{D}{D}}$, and
	\item a functor $\fct[F]{C}{D}$.
\end{itemize}

We first describe a condition on a natural transformation $\func[\theta]{\Ycal_1 \circ F}{F \circ \Xcal_1}$ such that $F$ lifts to the unique relative morphism $\func[(F, \vareta)]{\sobjcal{X}}{\sobjcal{Y}}$ with the property that $\vareta[1] = \theta$.

\begin{lem}
	\label{lem:pre-mor-2}
	Let $\vareta= \famnat{\func[\vareta[n]]{\Ycal_n \circ F}{F \circ \Xcal_n}}$ be a collection of natural transformations.
	We define $\Seta$ to be the subset of $\mathbb{N}^2$ consisting of those pairs $(m, n)$ such that \emph{(\ref{eq:mor-3})} holds.
	Then, for all $\ell, m, n \ge 0$ such that $(\ell, m), (m, n) \in \Seta$, the conditions $(\ell + m, n) \in \Seta$ and $(\ell, m + n) \in \Seta$ are equivalent.
\end{lem}

\begin{proof}
    This follows from the equalities
	\begin{align*}
		(\vareta[\ell + m] \star \Xcal_n) \circ (\Ycal_{\ell + m} \star \vareta[n]) 
		&= (((\vareta[\ell] \star \Xcal_m) \circ (\Ycal_\ell \star \vareta[m])) \star \Xcal_n) \circ (\Ycal_\ell \star \Ycal_m \star \vareta[n]) \\
		&= (\vareta[\ell] \star \Xcal_m \star \Xcal_n) \circ (\Ycal_\ell \star \vareta[m] \star \Xcal_n) \circ (\Ycal_\ell \star \Ycal_m \star \vareta[n]) \\
		&= (\vareta[\ell] \star \Xcal_m \star \Xcal_n) \circ (\Ycal_\ell \star ((\vareta[m] \star \Xcal_n) \circ (\Ycal_m \star \vareta[n]))) \\
		&= (\vareta[\ell] \star \Xcal_{m + n}) \circ (\Ycal_\ell \star \vareta[m + n]),
	\end{align*}
	where the first and fourth equalities follow from the assumptions $(\ell, m) \in \Seta$ and $(m, n) \in \Seta$.
\end{proof}

\begin{lem}
	\label{lem:mor-2-and-3}
	Let $\func[\theta]{\Ycal_1 \circ F}{F \circ \Xcal_1}$ be a natural transformation.
	We inductively define natural transformations $\func[\vareta[n]]{\Ycal_n \circ F}{F \circ \Xcal_n}$ by $\varetazero \coloneqq \id{F}$ and $\vareta[n + 1] \coloneqq (\theta \star \Xcal_n) \circ (\Ycal_1 \star \vareta[n])$ for all $n \ge 0$.
	Then, the pair $(F, \vareta)$ satisfies the equalities $\vareta[1] = \theta$, \emph{(\ref{eq:mor-2})}, and \emph{(\ref{eq:mor-3})} for all $m, n \ge 0$.
\end{lem}

\begin{proof}
    Since $\Ycal_1 \star \varetazero = \id{\Ycal_1 \circ F}$ and $\theta \star \Xcal_0 = \theta \star \id{\cat{C}} = \theta$, it follows that $\vareta[1] = (\theta \star \Xcal_0) \circ (\Ycal_1 \star \varetazero) = \theta \star \Xcal_0 = \theta$, whereas (\ref{eq:mor-2}) is clearly satisfied.
    It now remains to prove that (\ref{eq:mor-3}) holds, or equivalently, that $(m, n) \in \Seta$, for all $m, n \ge 0$, where $\Seta$ is defined as in Lemma~\ref{lem:pre-mor-2}.
    We argue by double induction.
    For all $n \ge 0$, since $\Ycal_0 \star \vareta[n] = \id{\cat{D}} \star \vareta[n] = \vareta[n]$ and $\varetazero \star \Xcal_n = \id{F \circ \Xcal_n}$, it follows that $(\varetazero \star \Xcal_n) \circ (\Ycal_0 \star \vareta[n]) = \Ycal_0 \star \vareta[n] = \vareta[n]$, and the case $m = 0$ follows.
    We fix $m \ge 0$, and suppose that (\ref{eq:mor-3}) holds for the fixed $m$ and all $n \ge 0$.
    By the definition of $\vareta$, we have $(1, p) \in \Seta$ for all $p \ge 0$, and in particular, $(1, m), (1, m + n) \in \Seta$.
    Since $(m, n) \in \Seta$ by the inductive hypothesis, Lemma~\ref{lem:pre-mor-2} implies that $(m + 1, n) 
    \in \Seta$, as desired.
\end{proof}

\begin{prop}
	\label{prop:mor-2}
    Let $\cat{C}$ and $\cat{D}$ be categories, $\func[\sobjcal{X}]{\opI}{\Fun{C}{C}}$ and $\func[\sobjcal{Y}]{\opI}{\Fun{D}{D}}$ strict monoidal functors, and $\fct[F]{C}{D}$ a functor.
    Then, the assignment $\vareta \mapsto \vareta[1]$ induces a bijective correspondence between
	\begin{itemize}
		\item collections $\vareta = \famnat{\func[\vareta[n]]{\Ycal_n \circ F}{F \circ \Xcal_n}}$ of natural transformations such that the pair $(F, \vareta)$ satisfies \emph{(\ref{eq:mor-2})} and \emph{(\ref{eq:mor-3})} for all $m, n \ge 0$, and
		\item natural transformations $\func[\theta]{\Ycal_1 \circ F}{F \circ \Xcal_1}$.
	\end{itemize}
\end{prop}

\begin{proof}
    We first note that this assignment is surjective by Lemma~\ref{lem:mor-2-and-3}.
    In order to prove the injectivity, take collections $\vareta$ and $\varetaVV$ of natural transformations subject to the displayed conditions, and suppose that $\vareta[1] = \varetaVV[1]$.
	We prove that $\vareta[n] = \varetaVV[n]$ by induction on $n \ge 0$.
	The case $n = 0$ follows since $\varetazero = \id{F} = \varetaVVzero$ by (\ref{eq:mor-2}).
	If the equality $\vareta[n] = \varetaVV[n]$ holds, then we can compute
	\[\vareta[n + 1] 
	= (\vareta[1] \star \Xcal_n) \circ (\Ycal_1 \star \vareta[n]) 
	= (\varetaVV[1] \star \Xcal_n) \circ (\Ycal_1 \star \vareta[n]) 
	= (\varetaVV[1] \star \Xcal_n) \circ (\Ycal_1 \star \varetaVV[n]) 
	= \varetaVV[n + 1],\]
	where the first and fourth equalities follow from (\ref{eq:mor-3}), the second equality from $\vareta[1] = \varetaVV[1]$, and the third equality from the inductive hypothesis.
	This shows that the assignment is also injective, as desired.
\end{proof}

\begin{lem}
	\label{lem:mor-1}
	Let $\vareta= \famnat{\func[\vareta[n]]{\Ycal_n \circ F}{F \circ \Xcal_n}}$ be a collection of natural transformations.
	For all $m, n \ge 0$, we define $\Teta{m}{n}$ to be the set of those morphisms $\funcgene[\phi]{m}{n}$ in $\cat{I}$ such that \emph{(\ref{eq:mor-1})} holds.
	\begin{enumerate}
        \item \label{Ieta-cond-1} For all $\ell, m, n \ge 0$, $\phi \in \Teta{m}{n}$ and $\psi \in \Teta{\ell}{m}$ imply that $\phi \circ \psi \in \Teta{\ell}{n}$.
		\item \label{Ieta-cond-2} For all $n \ge 0$, we have $\id{\gene{n}} \in \Teta{n}{n}$.
		\item \label{Ieta-cond-3} For all isomorphisms $\funcgene[\phi]{m}{n}$ in $\cat{I}$ which belong to $\Teta{m}{n}$, we have $\phi^{-1} \in \Teta{n}{m}$.
	\end{enumerate}
	Suppose further that the pair $(F, \vareta)$ satisfies \emph{(\ref{eq:mor-3})}.
	Then\emph{:}
	\begin{enumerate}[resume]
		\item \label{Ieta-cond-4} For all $\phi \in \Teta{m}{n}$ and $\phi' \in \Teta{m'}{n'}$, we have $\phi \ot \phi' \in \Teta{m + m'}{n + n'}$.
	\end{enumerate}
\end{lem}

\begin{proof}
	\begin{enumerate}
        \item This follows from the equalities
		\begin{align*}
			\vareta[\ell] \circ (\Ycal(\phi \circ \psi) \star F) 
			&= \vareta[\ell] \circ ((\Ycal(\psi) \circ \Ycal(\phi)) \star F) 
			= \vareta[\ell] \circ (\Ycal(\psi) \star F) \circ (\Ycal(\phi) \star F) \\
			&= (F \star \Xcal(\psi)) \circ \vareta[m] \circ (\Ycal(\phi) \star F) 
			= (F \star \Xcal(\psi)) \circ (F \star \Xcal(\phi)) \circ \vareta[n] \\
			&= (F \star (\Xcal(\psi) \circ \Xcal(\phi))) \circ \vareta[n] 
			= (F \star \Xcal(\phi \circ \psi)) \circ \vareta[n],
		\end{align*}
		where the third and fourth equalities follow from the assumptions $\psi \in \Teta{\ell}{m}$ and $\phi \in \Teta{m}{n}$.
		\item This follows since $\vareta[n] \circ (\Ycal(\id{\gene{n}}) \star F) = \vareta[n] \circ \id{\Ycal_n \circ F} = \id{F \circ \Xcal_n} \circ \vareta[n] = (F \star \Xcal(\id{\gene{n}})) \circ \vareta[n]$.
		\item Since $F \star \Xcal(\phi)$ is a natural isomorphism, this follows from the equalities
		\begin{align*}
			\vareta[n] \circ (\Ycal(\phi^{-1}) \star F) 
			&= (F \star \Xcal(\phi))^{-1} \circ (F \star \Xcal(\phi)) \circ \vareta[n] \circ (\Ycal(\phi^{-1}) \star F) \\
			&= (F \star \Xcal(\phi)^{-1}) \circ (F \star \Xcal(\phi)) \circ \vareta[n] \circ (\Ycal(\phi)^{-1} \star F) \\
			&= (F \star \Xcal(\phi^{-1})) \circ (F \star \Xcal(\phi)) \circ \vareta[n] \circ (\Ycal(\phi) \star F)^{-1} \\
			&= (F \star \Xcal(\phi^{-1})) \circ \vareta[m] \circ (\Ycal(\phi) \star F) \circ (\Ycal(\phi) \star F)^{-1} 
			= (F \star \Xcal(\phi^{-1})) \circ \vareta[m],
		\end{align*}
		where the fourth equality follows from the assumption $\phi \in \Teta{m}{n}$.
		\item Using (\ref{eq:mor-3}), we can deduce the following equalities:
		\begin{align*}
			\vareta[m + m'] \circ (\Ycal(\phi \ot \phi') \star F) 
			&= \vareta[m + m'] \circ (\Ycal(\phi) \star \Ycal(\phi') \star F) \\
			&= (\vareta[m] \star \Xcal_{m'}) \circ (\Ycal_m \star \vareta[m']) \circ (\Ycal(\phi) \star \Ycal(\phi') \star F) \\
			&= (\vareta[m] \star \Xcal_{m'}) \circ (\Ycal(\phi) \star (\vareta[m'] \circ (\Ycal(\phi') \star F))); \\
			(F \star \Xcal(\phi \ot \phi')) \circ \vareta[n + n'] 
			&= (F \star \Xcal(\phi) \star \Xcal(\phi')) \circ \vareta[n + n'] \\
			&= (F \star \Xcal(\phi) \star \Xcal(\phi')) \circ (\vareta[n] \star \Xcal_{n'}) \circ (\Ycal_n \star \vareta[n']) \\
			&= (((F \star \Xcal(\phi)) \circ \vareta[n]) \star \Xcal(\phi')) \circ (\Ycal_n \star \vareta[n']).
		\end{align*}
		It therefore suffices to prove the equality
		\[(\vareta[m] \star \Xcal_{m'}) \circ (\Ycal(\phi) \star (\vareta[m'] \circ (\Ycal(\phi') \star F))) = (((F \star \Xcal(\phi)) \circ \vareta[n]) \star \Xcal(\phi')) \circ (\Ycal_n \star \vareta[n']).\]
		This can be verified as
		\begin{align*}
			(\vareta[m] \star \Xcal_{m'}) \circ (\Ycal(\phi) \star (\vareta[m'] \circ (\Ycal(\phi') \star F))) 
			&= (\vareta[m] \star \Xcal_{m'}) \circ (\Ycal(\phi) \star ((F \star \Xcal(\phi')) \circ \vareta[n'])) \\
			&= (\vareta[m] \star \Xcal_{m'}) \circ (\Ycal(\phi) \star F \star \Xcal(\phi')) \circ (\Ycal_n \star \vareta[n']) \\
			&= ((\vareta[m] \circ (\Ycal(\phi) \star F)) \star \Xcal(\phi')) \circ (\Ycal_n \star \vareta[n']) \\
			&= (((F \star \Xcal(\phi)) \circ \vareta[n]) \star \Xcal(\phi')) \circ (\Ycal_n \star \vareta[n']),
		\end{align*}
		where the first and fourth equalities follow from $\phi' \in \Teta{m'}{n'}$ and $\phi \in \Teta{m}{n}$.
    \end{enumerate}
	\vspace{-1\baselineskip}
\end{proof}

\begin{rem}
	\label{rem:Ieta}
	Using the notation of Lemma~\ref{lem:mor-1}, conditions~\ref{Ieta-cond-1} and \ref{Ieta-cond-2} are rephrased as the following single condition: Given a collection $\vareta$ in general, we can define a subcategory $\Ieta$ of $\cat{I}$, which has the same objects as $\cat{I}$, whereas $\hom{{\minusmyspace} \hspace{0.1em} \Ieta \hspace{-0.1em}}{\gene{m}}{\gene{n}} \coloneqq \Teta{m}{n}$ for all $m, n \ge 0$.
	The condition~\ref{Ieta-cond-3} says that the inclusion functor $\Ieta \hookrightarrow \cat{I}$ is \term{conservative}, or in other words, it \term{reflects isomorphisms}, in the sense that any morphism in $\Ieta$ which is an isomorphism in $\cat{I}$ is necessarily an isomorphism in the subcategory $\Ieta$ as well.
    On the other hand, condition~\ref{Ieta-cond-4} says that $\Ieta$ is a monoidal subcategory of $\cat{I}$.
\end{rem}

\begin{prop}
	\label{prop:mor-3}
	Let $\cat{C}$ and $\cat{D}$ be categories, $\sobjcal{X}$ and $\sobjcal{Y}$ strict monoidal {\ABSO}s in $\Fun{C}{C}$ and $\Fun{D}{D}$, respectively, and $\fct[F]{C}{D}$ a functor.
	Then, the correspondence $\vareta \mapsto \vareta[1]$ of \emph{Proposition~\ref{prop:mor-2}} restricts to a bijective correspondence between
    \begin{itemize}
		\item collections $\vareta = \famnat{\func[\vareta[n]]{\Ycal_n \circ F}{F \circ \Xcal_n}}$ of natural transformations such that the pair $(F, \vareta)$ is a relative morphism from $\sobjcal{X}$ to $\sobjcal{Y}$, and
		\item natural transformations $\func[\theta]{\Ycal_1 \circ F}{F \circ \Xcal_1}$ which satisfy the following equalities\emph{:}
		\begin{align}
			(\Ycal(\cofacevar) \star F) 
			&= (F \star \Xcal(\cofacevar)) \circ \theta; \label{eq:theta-1} \\
			(\theta \star \Xcal_1) \circ (\Ycal_1 \star \theta) \circ (\Ycal(\codegevar) \star F) 
			&= (F \star \Xcal(\codegevar)) \circ \theta; \label{eq:theta-2} \\
			(\theta \star \Xcal_1) \circ (\Ycal_1 \star \theta) \circ (\Ycal(\copermvar) \star F) 
			&= (F \star \Xcal(\copermvar)) \circ (\theta \star \Xcal_1) \circ (\Ycal_1 \star \theta). \label{eq:theta-3}
		\end{align}
	\end{itemize}
\end{prop}

\begin{proof}
    Take a collection $\vareta = \famnat{\func[\vareta[n]]{\Ycal_n \circ F}{F \circ \Xcal_n}}$ of natural transformations such that the pair $(F, \vareta)$ satisfies (\ref{eq:mor-2}) and (\ref{eq:mor-3}) for all $m, n \ge 0$, and put $\theta \coloneqq \vareta[1]$.
    Since the pair $(F, \vareta)$ satisfies (\ref{eq:mor-2}) and (\ref{eq:mor-3}), it follows that $\varetazero = \id{F}$, $\vareta[1] = \theta$, and $\vareta[2] = (\vareta[1] \star \Xcal_1) \circ (\Ycal_1 \star \vareta[1]) = (\theta \star \Xcal_1) \circ (\Ycal_1 \star \theta)$, so that (\ref{eq:theta-1}), (\ref{eq:theta-2}), and (\ref{eq:theta-3}) are equivalent to the conditions that $\cofacevar \in \Teta{0}{1}$, $\codegevar \in \Teta{2}{1}$, and $\copermvar \in \Teta{2}{2}$.
    In view of Proposition~\ref{prop:mor-2}, it remains to check that the pair $(F, \vareta)$ satisfies (\ref{eq:mor-1}) for all morphisms $\phi$ in $\augbr$ if and only if $\theta$ satisfies (\ref{eq:theta-1})--(\ref{eq:theta-3}).
    We note that the former assertion holds if and only if the equality $\Teta{m}{n} = \hom{\cat{I}}{\gene{m}}{\gene{n}}$ holds for all $m, n \ge 0$, and the necessity follows readily from this.
    We now prove the sufficiency.
    Suppose that $\theta$ satisfies (\ref{eq:theta-1})--(\ref{eq:theta-3}).
    Since the pair $(F, \vareta)$ satisfies (\ref{eq:mor-3}) (see Proposition~\ref{prop:mor-2}), conditions~\ref{Ieta-cond-1}--\ref{Ieta-cond-4} of Lemma~\ref{lem:mor-1} are all satisfied, and in particular, we have the subcategory $\Ieta$ of $\augbr$, defined as in Remark~\ref{rem:Ieta}.
    It follows from the assumptions that $\cofacevar$, $\codegevar$, and $\copermvar$ are morphisms in $\Ieta$.
    Using conditions~\ref{Ieta-cond-2} and \ref{Ieta-cond-4} twice, it follows that $\coface{j}{n} = \gene{j} \ot \cofacevar \ot \gene{n - j}$ belongs to $\Ieta$ for all $0 \le j \le n$.
    A similar argument shows that all generating morphisms of $\augbr$ belong to $\Ieta$.
    Note that condition \ref{Ieta-cond-3} is required to prove that $(\coperm{i}{n})^{-1}$ belongs to $\Ieta$ for all $1 \le i \le n$.
    We therefore obtain the equality $\Ieta = \augbr$ in view of Remark~\ref{rem:generate-augbr}, which implies the former assertion, as desired.
\end{proof}

\subsection{Relative morphisms over monoidal functors}
\label{subsec:4.2}

In this subsection, we focus on dealing with relative morphisms in the case of interest.
Let $\cat{C}$ and $\cat{D}$ be monoidal categories, and $\brobj{C} = \brcoalg{C}$ and $\brobjv{D} = \brcoalgv{D}$ braided coalgebras in $\cat{C}$ and $\cat{D}$, and we write $\sobjcal{X}$ and $\sobjcal{Y}$ for the monoidal {\ABSO}s in $\Fun{C}{C}$ and $\Fun{D}{D}$ that are constructed from $\brobj{C}$ and $\brobjv{D}$ as in Proposition~\ref{prop:nonstr-augbrobj}.
Using the equalities (\ref{eq:wTcalC}), the involved morphisms are described explicitly as:
\begin{alignat}{2}
    \Xcal_1 &= C \ot (-); 
    & \Ycal_1 &= D \ot (-); \label{eq:XY-0} \\
    \Xcal(\cofacevar) &= \lucvar \circ (\eps \ot (-)); 
    & \Ycal(\cofacevar) &= \lucvar' \circ (\eps' \ot (-)); \label{eq:XY-1} \\
    \Xcal(\codegevar) &= \assc{C}{C}{-} \circ (\Delta \ot (-)); 
    & \Ycal(\codegevar) &= \asscv{D}{D}{-} \circ (\Delta' \ot (-)); \label{eq:XY-2} \\
    \Xcal(\copermvar) &= \assc{C}{C}{-} \circ (\br \ot (-)) \circ \assc{C}{C}{-}^{-1}; \qquad 
    & \Ycal(\copermvar) &= \asscv{D}{D}{-} \circ (\br' \ot (-)) \circ \asscvinv{D}{D}{-}. \label{eq:XY-3}
\end{alignat}
Let $\monfctvar{F}{C}{D}$ be a monoidal functor.
We study in more detail the bijective correspondence of Proposition~\ref{prop:mor-3} in this setting, and describe a relative morphism $\func[(F, \vareta)]{\sobjcal{X}}{\sobjcal{Y}}$ using a morphism in $\cat{D}$, by deforming the equalities (\ref{eq:theta-1})--(\ref{eq:theta-3}) with respect to the monoidal structure on $F$.
Since it is still difficult to work in this generality, we make an assumption on the collection $\vareta$ of natural transformations as well.
Note that $\vareta$ corresponds to the natural transformation $\vareta[1]$, which in view of (\ref{eq:XY-0}) is written as
\[D \ot F(-) = \Ycal_1 \circ F \to F \circ \Xcal_1 = F(C \ot (-)).\]
Postcomposing $\vareta[1]$ with $\monfctmul{F}{C}{-}^{-1}$ gives rise to a natural transformation from $D \ot F(-)$ to $F(C) \ot F(-)$.
We assume that this natural transformation is obtained by tensoring $\id{F}$ with some morphism $\func{D}{F(C)}$ in $\cat{D}$.
More explicitly, we assume that $\vareta$ satisfies the following condition:
\begin{itemize}
	\item[($**$)] There exists a morphism $\func[\alpha]{D}{F(C)}$ in $\cat{D}$ such that $\vareta[1] = \monfctmul{F}{C}{-} \circ (\alpha \ot F(-))$.
\end{itemize}
Lemma~\ref{lem:monfct-easy} guarantees that such a morphism $\alpha$ is uniquely determined as follows (if it exists):
\begin{equation}
	\label{eq:alpha}
	\alpha = \rucv{F(C)} \circ (F(C) \ot \monfctuni{F}^{-1}) \circ \monfctmul{F}{C}{\unit}^{-1} \circ (\vareta[1])_{\unit} \circ (D \ot \monfctuni{F}) \circ \rucvinv{D}.
\end{equation}

\begin{lem}
    \label{lem:alpha}
    Let $\fct[F = \monfct{F}]{C}{D}$ be a monoidal functor, and $\func[\alpha]{D}{F(C)}$ a morphism in $\cat{D}$.
    Let $\theta$ denote the natural transformation
    \[\Ycal_1 \circ F = D \ot F(-) \xrightarrow{\alpha \ot F(-)} F(C) \ot F(-) \xrightarrow{\monfctmul{F}{C}{-}} F(C \ot (-)) = F \circ \Xcal_1.\]
    \begin{enumerate}
        \item $\theta$ satisfies the equality \emph{(\ref{eq:theta-1})} if and only if $\alpha$ satisfies the equality
        \begin{equation}
            \label{eq:alpha-1}
            \monfctuni{F} \circ \eps' = F(\eps) \circ \alpha.
        \end{equation}
        \item $\theta$ satisfies the equality \emph{(\ref{eq:theta-2})} if and only if $\alpha$ satisfies the equality
        \begin{equation}
            \label{eq:alpha-2}
            \monfctmul{F}{C}{C} \circ (\alpha \ot \alpha) \circ \Delta' = F(\Delta) \circ \alpha.
        \end{equation}
        \item $\theta$ satisfies the equality \emph{(\ref{eq:theta-3})} if and only if $\alpha$ satisfies the equality
        \begin{equation}
            \label{eq:alpha-3}
            \monfctmul{F}{C}{C} \circ (\alpha \ot \alpha) \circ \br' = F(\br) \circ \monfctmul{F}{C}{C} \circ (\alpha \ot \alpha).
        \end{equation}
    \end{enumerate}
\end{lem}

\begin{proof}
    \begin{enumerate}
        \item Using (\ref{eq:XY-1}), the equality (\ref{eq:theta-1}) is rewritten as
        \[\lucv{F(-)} \circ (\eps' \ot F(-)) = F(\lucvar) \circ F(\eps \ot (-)) \circ \monfctmul{F}{C}{-} \circ (\alpha \ot F(-)).\]
        Since $F(\eps \ot (-)) \circ \monfctmul{F}{C}{-} = \monfctmul{F}{\unit}{-} \circ (F(\eps) \ot F(-))$ by (\ref{eq:monfct-0}), it is rewritten as
        \[\lucv{F(-)} \circ (\eps' \ot F(-)) = F(\lucvar) \circ \monfctmul{F}{\unit}{-} \circ (F(\eps) \ot F(-)) \circ (\alpha \ot F(-)).\]
        Since $\lucv{F(-)} = F(\lucvar) \circ \monfctmul{F}{\unit}{-} \circ (\monfctuni{F} \ot F(-))$ by (\ref{eq:monfct-2-l}) and since $F(\lucvar) \circ \monfctmul{F}{\unit}{-}$ is a natural isomorphism, it is equivalent to the equality
        \[(\monfctuni{F} \ot F(-)) \circ (\eps' \ot F(-)) = (F(\eps) \ot F(-)) \circ (\alpha \ot F(-)),\]
        and the desired equivalence follows from Lemma~\ref{lem:monfct-easy}.
        \item We omit the symbol $\circ$.
        Using (\ref{eq:XY-2}), the equality (\ref{eq:theta-2}) is rewritten as
        \begin{align*}
            & \monfctmul{F}{C}{C \ot (-)} (\alpha \ot F(C \ot (-))) (D \ot \monfctmul{F}{C}{-}) (D \ot (\alpha \ot F(-))) \asscv{D}{D}{F(-)} (\Delta' \ot F(-)) \\
            & \qquad \qquad \qquad \qquad \qquad \qquad \qquad \qquad \qquad \qquad \quad = F(\assc{C}{C}{-}) F(\Delta \ot (-)) \monfctmul{F}{C}{-} (\alpha \ot F(-)),
        \end{align*}
        which using the equalities
        \begin{align*}
            (\alpha \ot F(C \ot (-))) (D \ot \monfctmul{F}{C}{-}) (D \ot (\alpha \ot F(-))) 
            &= \alpha \ot (\monfctmul{F}{C}{-} (\alpha \ot F(-))) \\
            &= (F(C) \ot \monfctmul{F}{C}{-}) (\alpha \ot (\alpha \ot F(-)))
        \end{align*}
        is rewritten as
        \begin{align*}
            & \monfctmul{F}{C}{C \ot (-)} (F(C) \ot \monfctmul{F}{C}{-}) (\alpha \ot (\alpha \ot F(-))) \asscv{D}{D}{F(-)} (\Delta' \ot F(-)) \\
            & \qquad \qquad \qquad \qquad \qquad \qquad \qquad = F(\assc{C}{C}{-}) F(\Delta \ot (-)) \monfctmul{F}{C}{-} (\alpha \ot F(-)).
        \end{align*}
		Since $(\alpha \ot (\alpha \ot F(-))) \asscv{D}{D}{F(-)} = \asscv{F(C)}{F(C)}{F(-)} ((\alpha \ot \alpha) \ot F(-))$ by the naturality of $\asscvar'$, it is rewritten as
        \begin{align*}
            & \monfctmul{F}{C}{C \ot (-)} (F(C) \ot \monfctmul{F}{C}{-}) \asscv{F(C)}{F(C)}{F(-)} ((\alpha \ot \alpha) \ot F(-)) (\Delta' \ot F(-)) \\
            & \qquad \qquad \qquad \qquad \qquad \qquad \qquad \qquad = F(\assc{C}{C}{-}) F(\Delta \ot (-)) \monfctmul{F}{C}{-} (\alpha \ot F(-)).
        \end{align*}
		Since $\monfctmul{F}{C}{C \ot (-)} (F(C) \ot \monfctmul{F}{C}{-}) \asscv{F(C)}{F(C)}{F(-)} = F(\assc{C}{C}{-}) \monfctmul{F}{C \ot C}{-} (\monfctmul{F}{C}{C} \ot F(-))$ by (\ref{eq:monfct-1}) and since $F(\assc{C}{C}{-})$ is a natural isomorphism, it is equivalent to the equality
		\[\monfctmul{F}{C \ot C}{-} (\monfctmul{F}{C}{C} \ot F(-)) ((\alpha \ot \alpha) \ot F(-)) (\Delta' \ot F(-)) = F(\Delta \ot (-)) \monfctmul{F}{C}{-} (\alpha \ot F(-)).\]
		Since $F(\Delta \ot (-)) \monfctmul{F}{C}{-} = \monfctmul{F}{C \ot C}{-} (F(\Delta) \ot F(-))$ by (\ref{eq:monfct-0}) and since $\monfctmul{F}{C \ot C}{-}$ is a natural isomorphism, it is equivalent to the equality
		\[(\monfctmul{F}{C}{C} \ot F(-)) ((\alpha \ot \alpha) \ot F(-)) (\Delta' \ot F(-)) = (F(\Delta) \ot F(-)) (\alpha \ot F(-)),\]
        and the desired equivalence follows from Lemma~\ref{lem:monfct-easy}.
        \item We omit the symbol $\circ$.
        Using (\ref{eq:XY-3}), the equality (\ref{eq:theta-3}) is rewritten as
        \begin{align*}
            & \monfctmul{F}{C}{C \ot (-)} (\alpha \ot F(C \ot (-))) (D \ot \monfctmul{F}{C}{-}) (D \ot (\alpha \ot F(-))) \asscv{D}{D}{F(-)} (\br' \ot F(-)) \asscvinv{D}{D}{F(-)} \\
            & \quad = F(\assc{C}{C}{-}) F(\br \ot (-)) F(\assc{C}{C}{-}^{-1}) \monfctmul{F}{C}{C \ot (-)} (\alpha \ot F(C \ot (-))) (D \ot \monfctmul{F}{C}{-}) (D \ot (\alpha \ot F(-))),
        \end{align*}
        which using the equalities
        \[(\alpha \ot F(C \ot (-))) (D \ot \monfctmul{F}{C}{-}) 
        = \alpha \ot \monfctmul{F}{C}{-} 
        = (F(C) \ot \monfctmul{F}{C}{-}) (\alpha \ot (F(C) \ot F(-)))\]
        is rewritten as
        \begin{align*}
            & \monfctmul{F}{C}{C \ot (-)} (F(C) \ot \monfctmul{F}{C}{-}) (\alpha \ot (F(C) \ot F(-))) (D \ot (\alpha \ot F(-))) \asscv{D}{D}{F(-)} (\br' \ot F(-)) \asscvinv{D}{D}{F(-)} \\
            & \quad = F(\assc{C}{C}{-}) F(\br \ot (-)) F(\assc{C}{C}{-}^{-1}) \monfctmul{F}{C}{C \ot (-)} (\alpha \ot F(C \ot (-))) (D \ot \monfctmul{F}{C}{-}) (D \ot (\alpha \ot F(-))).
        \end{align*}
        Since
        \begin{align*}
            & (\alpha \ot (F(C) \ot F(-))) (D \ot (\alpha \ot F(-))) \asscv{D}{D}{F(-)} \\
            & \qquad \qquad = (\alpha \ot (\alpha \ot F(-))) \asscv{D}{D}{F(-)} 
            = \asscv{F(C)}{F(C)}{F(-)} ((\alpha \ot \alpha) \ot F(-)),
        \end{align*}
        by the naturality of $\asscvar'$, it is rewritten as
        \begin{align*}
            & \monfctmul{F}{C}{C \ot (-)} (F(C) \ot \monfctmul{F}{C}{-}) \asscv{F(C)}{F(C)}{F(-)} ((\alpha \ot \alpha) \ot F(-)) (\br' \ot F(-)) \asscvinv{D}{D}{F(-)} \\
            & \quad = F(\assc{C}{C}{-}) F(\br \ot (-)) F(\assc{C}{C}{-}^{-1}) \monfctmul{F}{C}{C \ot (-)} (\alpha \ot F(C \ot (-))) (D \ot \monfctmul{F}{C}{-}) (D \ot (\alpha \ot F(-))).
        \end{align*}
        Since $\monfctmul{F}{C}{C \ot (-)} (F(C) \ot \monfctmul{F}{C}{-}) \asscv{F(C)}{F(C)}{F(-)} = F(\assc{C}{C}{-}) \monfctmul{F}{C \ot C}{-} (\monfctmul{F}{C}{C} \ot F(-))$ by (\ref{eq:monfct-1}) and since $F(\assc{C}{C}{-})$ is a natural isomorphism, it is equivalent to the equality
        \begin{align*}
            & \monfctmul{F}{C \ot C}{-} (\monfctmul{F}{C}{C} \ot F(-)) ((\alpha \ot \alpha) \ot F(-)) (\br' \ot F(-)) \asscvinv{D}{D}{F(-)} \\
            & \quad = F(\br \ot (-)) F(\assc{C}{C}{-}^{-1}) \monfctmul{F}{C}{C \ot (-)} (\alpha \ot F(C \ot (-))) (D \ot \monfctmul{F}{C}{-}) (D \ot (\alpha \ot F(-))).
        \end{align*}
        Since
        \[D \ot (\alpha \ot F(-)) = (D \ot (\alpha \ot F(-))) \asscv{D}{D}{F(-)} \asscvinv{D}{D}{F(-)} = \asscv{D}{F(C)}{F(-)} ((D \ot \alpha) \ot F(-)) \asscvinv{D}{D}{F(-)}\]
        by the naturality of $\asscvar'$ and since $\asscvinv{D}{D}{F(-)}$ is a natural isomorphism, it is equivalent to the equality
        \begin{align*}
            & \monfctmul{F}{C \ot C}{-} (\monfctmul{F}{C}{C} \ot F(-)) ((\alpha \ot \alpha) \ot F(-)) (\br' \ot F(-)) \\
            & \quad = F(\br \ot (-)) F(\assc{C}{C}{-}^{-1}) \monfctmul{F}{C}{C \ot (-)} (\alpha \ot F(C \ot (-))) (D \ot \monfctmul{F}{C}{-}) \asscv{D}{F(C)}{F(-)} ((D \ot \alpha) \ot F(-)).
        \end{align*}
        Since
        \begin{align*}
            & (\alpha \ot F(C \ot (-))) (D \ot \monfctmul{F}{C}{-}) \asscv{D}{F(C)}{F(-)} ((D \ot \alpha) \ot F(-)) \\
            & \qquad = (\alpha \ot \monfctmul{F}{C}{-}) \asscv{D}{F(C)}{F(-)} ((D \ot \alpha) \ot F(-)) \\
            & \qquad = (F(C) \ot \monfctmul{F}{C}{-}) (\alpha \ot (F(C) \ot F(-))) \asscv{D}{F(C)}{F(-)} ((D \ot \alpha) \ot F(-)) \\
            & \qquad = (F(C) \ot \monfctmul{F}{C}{-}) \asscv{F(C)}{F(C)}{F(-)} ((\alpha \ot F(C)) \ot F(-)) ((D \ot \alpha) \ot F(-)) \\
            & \qquad = (F(C) \ot \monfctmul{F}{C}{-}) \asscv{F(C)}{F(C)}{F(-)} ((\alpha \ot \alpha) \ot F(-))
        \end{align*}
        by the naturality of $\asscvar'$, it is rewritten as
        \begin{align*}
            & \monfctmul{F}{C \ot C}{-} (\monfctmul{F}{C}{C} \ot F(-)) ((\alpha \ot \alpha) \ot F(-)) (\br' \ot F(-)) \\
            & \qquad = F(\br \ot (-)) F(\assc{C}{C}{-}^{-1}) \monfctmul{F}{C}{C \ot (-)} (F(C) \ot \monfctmul{F}{C}{-}) \asscv{F(C)}{F(C)}{F(-)} ((\alpha \ot \alpha) \ot F(-)).
        \end{align*}
        Using $\monfctmul{F}{C}{C \ot (-)} (F(C) \ot \monfctmul{F}{C}{-}) \asscv{F(C)}{F(C)}{F(-)} = F(\assc{C}{C}{-}) \monfctmul{F}{C \ot C}{-} (\monfctmul{F}{C}{C} \ot F(-))$ again, it is rewritten as
        \begin{align*}
            & \monfctmul{F}{C \ot C}{-} (\monfctmul{F}{C}{C} \ot F(-)) ((\alpha \ot \alpha) \ot F(-)) (\br' \ot F(-)) \\
            & \qquad \qquad = F(\br \ot (-)) \monfctmul{F}{C \ot C}{-} (\monfctmul{F}{C}{C} \ot F(-)) ((\alpha \ot \alpha) \ot F(-)).
        \end{align*}
        Since $F(\br \ot (-)) \monfctmul{F}{C \ot C}{-} = \monfctmul{F}{C \ot C}{-} (F(\br) \ot F(-))$ by (\ref{eq:monfct-0}) and since $\monfctmul{F}{C \ot C}{-}$ is a natural isomorphism, it is equivalent to the equality
        \begin{align*}
            & (\monfctmul{F}{C}{C} \ot F(-)) ((\alpha \ot \alpha) \ot F(-)) (\br' \ot F(-)) \\
            & \qquad \qquad = (F(\br) \ot F(-)) (\monfctmul{F}{C}{C} \ot F(-)) ((\alpha \ot \alpha) \ot F(-)),
        \end{align*}
        and the desired equivalence follows from Lemma~\ref{lem:monfct-easy}.
    \end{enumerate}
	\vspace{-1\baselineskip}
\end{proof}

Our main result of Subsection~\ref{subsec:4.2} can be stated as follows:

\begin{thm}[Theorem~\ref{mainthm:A}]
    \label{thm:mor-4}
    Let $\cat{C}$ and $\cat{D}$ be monoidal categories, $\brobj{C}$ and $\brobjv{D}$ braided coalgebras in $\cat{C}$ and $\cat{D}$, respectively, and $\monfctvar{F}{C}{D}$ a monoidal functor.
    Then, the assignment
    \begin{equation}
		\label{eq:thm}
		\vareta \mapsto \rucv{F(C)} \circ (F(C) \ot \monfctuni{F}^{-1}) \circ \monfctmul{F}{C}{\unit}^{-1} \circ (\vareta[1])_{\unit} \circ (D \ot \monfctuni{F}) \circ \rucvinv{D}
	\end{equation}
	induces a bijective correspondence between
    \begin{enumerate}[label=\upshape{(\roman*)}, widest=iii]
        \item \label{enum:class-1} \addtocounter{enumi}{1} collections $\vareta$ of natural transformations such that the pair $(F, \vareta)$ is a relative morphism from $\wTcalC$ to $\wTcalD$, and $\vareta$ satisfies condition \emph{($**$)}, and
        \item \label{enum:class-3} morphisms $\func[\alpha]{D}{F(C)}$ in $\cat{D}$ which satisfy \emph{(\ref{eq:alpha-1})}--\emph{(\ref{eq:alpha-3})}.
    \end{enumerate}
\end{thm}

\begin{proof}
    By Proposition~\ref{prop:mor-3}, the assignment $\vareta \mapsto \vareta[1]$ induces a bijective correspondence between \ref{enum:class-1} and
    \begin{enumerate}[label=(\roman*), start=2]
		\item \label{enum:class-2} natural transformations $\func[\theta]{D \ot F(-)}{F(C \ot (-))}$ which satisfy (\ref{eq:theta-1})--(\ref{eq:theta-3}) as well as the condition that there exist a morphism $\func[\alpha]{D}{F(C)}$ in $\cat{D}$ such that $\theta = \monfctmul{F}{C}{-} \circ (\alpha \ot F(-))$,
    \end{enumerate}
	whereas Lemma~\ref{lem:alpha} and the preceding arguments imply that the assignment $\alpha \mapsto \monfctmul{F}{C}{-} \circ (\alpha \ot F(-))$ induces a bijective correspondence between \ref{enum:class-3} and \ref{enum:class-2}, whose inverse is written explicitly as
	\[\theta \mapsto \rucv{F(C)} \circ (F(C) \ot \monfctuni{F}^{-1}) \circ \monfctmul{F}{C}{\unit}^{-1} \circ \theta_{\unit} \circ (D \ot \monfctuni{F}) \circ \rucvinv{D};\]
	see also (\ref{eq:alpha}).
	The desired 
	correspondence between \ref{enum:class-1} and \ref{enum:class-3} is the zigzag of these 
	correspondences.
\end{proof}

\begin{rem}
    \label{rem:mor-of-coalg}
    Let $\cat{C}$ be a monoidal category, and $\brcoalg{C}$ and $\brcoalgv{D}$ braided coalgebras in $\cat{C}$, and consider the strict monoidal functor $\fct[\id{\cat{C}}]{C}{C}$.
    Theorem~\ref{thm:mor-4}, in view of Remark~\ref{rem:abso-mor}, asserts that the correspondence $\vareta \mapsto \ruc{C} \circ (\vareta[1])_{\unit} \circ \ruc{D}^{-1}$ induces a bijective correspondence between
    \begin{itemize}
        \item monoidal natural transformations $\func[\vareta]{\wTcalD}{\wTcalC}$ such that 
        $\vareta[1]$ is written in the form $\alpha \ot (-)$ for some morphism $\func[\alpha]{D}{C}$ in $\cat{C}$, and
        \item morphisms $\func[\alpha]{D}{C}$ in $\cat{C}$ which satisfy the equalities $\eps' = \eps \circ \alpha$, $(\alpha \ot \alpha) \circ \Delta' = \Delta \circ \alpha$, and $(\alpha \ot \alpha) \circ \br' = \br \circ (\alpha \ot \alpha)$; that is, \textit{morphisms $\func[\alpha]{\brobjv{D}}{\brobj{C}}$ of braided coalgebras in $\cat{C}$} in the sense of \cite[Definition 2.1 (3)]{ardizzoni-menini}.
    \end{itemize}
\end{rem}

Remark~\ref{rem:mor-of-coalg} suggests that we refer to morphisms $\func[\alpha]{D}{F(C)}$ in the correspondence of Theorem~\ref{thm:mor-4} as \textit{relative morphisms of braided coalgebras} from $\brobjv{D}$ to $\brobj{C}$ (with respect to $F$).

Using the notation of Theorem~\ref{thm:mor-4}, when the morphisms $\br$ and $\br'$ of the braided coalgebras $\brobj{C}$ and $\brobjv{D}$ are \term{globally defined} in the sense that they come from the braidings on $\cat{C}$ and $\cat{D}$ as in Example~\ref{ex:coalg-br}~\ref{enum:ex-2}, respectively, the hypotheses on $\alpha$ in Theorem~\ref{thm:mor-4} can be weakened as follows:

\begin{cor}
	\label{cor:mor-5}
    Let $\brmoncat{C}$ and $\brmoncatv{D}$ be braided monoidal categories, $C$ and $D$ coalgebras in $\cat{C}$ and $\cat{D}$, and $\brmonfct{F}{C}{D}$ a braided monoidal functor.
    Then, the correspondence \emph{(\ref{eq:thm})} of \emph{Theorem~\ref{thm:mor-4}} induces a bijective correspondence between
    \begin{itemize}
        \item collections $\vareta$ of natural transformations such that the pair $(F, \vareta)$ is a relative morphism from $\wTcal(C, \braid{C}{C})$ to $\wTcal(D, \braidv{D}{D})$ and $\vareta$ satisfies condition \emph{($**$)}, and
        \item morphisms $\func[\alpha]{D}{F(C)}$ in $\cat{D}$ which satisfy \emph{(\ref{eq:alpha-1})} and \emph{(\ref{eq:alpha-2})}.
    \end{itemize}
\end{cor}

\begin{proof}
    It suffices to prove that (\ref{eq:alpha-3}) is automatically satisfied in this case for any morphism $\func[\alpha]{D}{F(C)}$ in $\cat{D}$, which can be verified as follows:
    \begin{align*}
        \monfctmul{F}{C}{C} \circ (\alpha \ot \alpha) \circ \braidv{D}{D} 
        &= \monfctmul{F}{C}{C} \circ \braidv{F(C)}{F(C)} \circ (\alpha \ot \alpha) 
        = F(\braid{C}{C}) \circ \monfctmul{F}{C}{C} \circ (\alpha \ot \alpha),
    \end{align*}
    where the first and second equalities follow from (\ref{eq:brmoncat-1}) and (\ref{eq:brmonfct}), respectively.
\end{proof}

By replacing $\augbr$ with $\augbrsemi$, we obtain a result which is similar to Corollary~\ref{cor:mor-5}:

\begin{cor}
    \label{cor:mor-6}
    Let $\brmoncat{C}$ and $\brmoncatv{D}$ be braided monoidal categories, $C$ and $D$ coalgebras in $\cat{C}$ and $\cat{D}$, and $\brmonfct{F}{C}{D}$ a braided monoidal functor.
    Then, the correspondence \emph{(\ref{eq:thm})} of \emph{Theorem~\ref{thm:mor-4}} induces a bijective correspondence between
    \begin{itemize}
        \item collections $\vareta$ of natural transformations such that the pair $(F, \vareta)$ is a relative morphism from $\wTcal(C, \braid{C}{C}) \circ \incl$ to $\wTcal(D, \braidv{D}{D}) \circ \incl$ and $\vareta$ satisfies condition \emph{($**$)}, and
        \item morphisms $\func[\alpha]{D}{F(C)}$ in $\cat{D}$ which satisfy \emph{(\ref{eq:alpha-1})}.
    \end{itemize}
\end{cor}

\begin{rem}
    \label{rem:rel-isom}
    We use the notation of Theorem~\ref{thm:mor-4}.
    Let $\vareta = \famnat{\func[\vareta[n]]{\Ycal_n \circ F}{F \circ \Xcal_n}}$ be a collection of natural transformations such that the pair $(F, \vareta)$ is a relative morphism $\func{\wTcalC}{\wTcalD}$, and we write $\alpha$ for the morphism $\func{D}{F(C)}$ in $\cat{D}$ that corresponds to $\vareta$ via the bijective correspondence (\ref{eq:thm}).
    Then, it is not difficult to show that $\vareta[n]$ is a natural isomorphism for all $n \ge 0$ if and only if $\alpha$ is an isomorphism in $\cat{D}$.
    Indeed, the necessity follows directly from (\ref{eq:thm}).
    Suppose that $\alpha$ is an isomorphism in $\cat{D}$.
    Then, $\eta_1 = \monfctmul{F}{C}{-} \circ (\alpha \ot F(-))$ is a natural isomorphism.
    Using (\ref{eq:mor-3}), one can prove by induction 
    that $\vareta[n]$ is a natural isomorphism for all $n \ge 0$.
\end{rem}

\section{Quasi-triangular bialgebras and examples}
\label{sec:5}

In this section, we restrict our attention to the study of $\Bbbk$-linear braided monoidal categories which arise from quasi-triangular $\Bbbk$-bialgebras.
We briefly review the basic definitions here, and refer to Etingof--Gelaki--Nikshych--Ostrik \cite[Section 8.3]{etingof-gelaki-nikshych-ostrik} or Kassel \cite[Section VIII.2]{kassel} for a detailed account of this theory.
We write $\ordot \coloneqq \ordot_{\Bbbk}$ and ${\dim} \coloneqq {\dim_{{\myspace} \Bbbk}}$.

\subsection{Quasi-triangular bialgebras}
\label{subsec:4.3}

Recall from Notation~\ref{nota:V} that $\cat{V}$ denotes the category of $\Bbbk$-modules, which we regard as a $\Bbbk$-linear monoidal category with respect to the usual monoidal structure $(\ordot, \Bbbk)$.

\begin{nota}
    \label{nota:alg}
    \begin{enumerate}
        \item \label{enum:AMod} Let $A$ be a $\Bbbk$-algebra.
        We write $\AMod$ for the $\Bbbk$-linear category of (left) $A$-modules, and we let $\ff[A]$ or $\ff$ denote the forgetful functor $\func{\AMod}{\cat{V}}$.
        \item \label{enum:Amod} Let $A$ be a finite-dimensional algebra over a field $\Bbbk$.
        We write $\Amod$ for the full subcategory of $\AMod$ consisting of finitely generated $A$-modules.
    \end{enumerate}
\end{nota}

\begin{defn}[{cf.~\cite[Definition III.2.2]{kassel}}]
    \label{defn:bialg}
    A triple $\bialg{A}$ is called a \term{$\Bbbk$-bialgebra} if $A$ is a $\Bbbk$-algebra and $\func[\Delta]{A}{A \ot A}$ and $\func[\eps]{A}{\Bbbk}$ are morphisms of $\Bbbk$-algebras such that $\coalg{\ff(A)}$ is a coalgebra in $\cat{V}$.
\end{defn}

Given a $\Bbbk$-bialgebra $\bialg{A}$, the category $\AMod$ of Notation~\ref{nota:alg}~\ref{enum:AMod} admits a monoidal structure $(\ordot, \Bbbk)$ such that the forgetful functor $\func[\ff]{\vMod{A}}{\cat{V}}$ admits a monoidal structure in the obvious way, where for all $A$-modules $M$ and $N$, $M \ot N$ is regarded as an $A$-module via  $a (m \ot n) \coloneqq \Delta(a) (m \ot n)$ for all $a \in A$, $m \in M$, and $n \in N$, and $\Bbbk$ is regarded as an $A$-module via $a \lambda \coloneqq \eps(a) \lambda$ for all $a \in A$ and $\lambda \in \Bbbk$ (see \cite[Proposition XI.3.1]{kassel} or Witherspoon \cite[page 186]{witherspoon}, for details).

\begin{defn}[{cf.~\cite[Definitions VIII.2.1 and VIII.2.2]{kassel}}]
    \label{defn:qtb}
    A quadruple $\qtribialg{A}$ is called a \term{quasi-triangular $\Bbbk$-bialgebra} if $\bialg{A}$ is a $\Bbbk$-bialgebra, $R = \swvar{1} \ot \swvar{2}$ is a unit element in $A \ot A$ (where the summation symbol is suppressed), such that the equality $(\tau \circ \Delta) (a) = R \Delta (a) R^{-1}$ holds for all $a \in A$, and the equalities $(\Delta \ot A) (R) = \swvar{1} \ot \swvar{1} \ot (\swvar{2} \swvar{2})$ and $(A \ot \Delta) (R) = (\swvar{1} \swvar{1}) \ot \swvar{2} \ot \swvar{2}$ hold, where $\func[\tau]{A \ot A}{A \ot A}[a \ot a'][a' \ot a]$ denotes the flip morphism.
    In this case, the element $R \in A \ot A$ is referred to as a (\textit{universal}) \textit{R-matrix} of $A$.
\end{defn}

We will often abbreviate a quasi-triangular $\Bbbk$-bialgebra $\qtribialg{A}$ to $\qtb{A}$.

\begin{rem}
    \label{rem:qtb-to-brcoalg}
    Let $\qtribialg{A}$ be a quasi-triangular $\Bbbk$-bialgebra.
    Then, the element $R = \swvar{1} \ot \swvar{2}$ gives a braiding $\braidtrivar$ on $\AMod$, where for all $A$-modules $M$ and $N$, the isomorphism $\func[\braidtri{M}{N}]{M \ot N}{N \ot M}$ 
    is defined by $\braidtri{M}{N}(m \ot n) \coloneqq (\swvar{2} n) \ot (\swvar{1} m)$ for all $m \in M$ and $n \in N$ (see \cite[Proposition XV.2.2 (a)]{kassel} for a proof).
    Since $\coalg{A}$ is a coalgebra in $\AMod$, it follows that $(A, \braidtri{A}{A}) = (A, \Delta, \eps, \braidtri{A}{A})$ is a braided coalgebra in $\AMod$ (see Example~\ref{ex:coalg-br}~\ref{enum:ex-2}).
\end{rem}

\begin{defn}
    \label{defn:KBA}
    Let $\qtb{A}$ be a quasi-triangular $\Bbbk$-bialgebra, and $M$ an $A$-module.
    We write
    \begin{gather*}
        \wTcal\qtb{A} \coloneqq \wTcal(A, \braidtri{A}{A}); \quad 
        \wT\qtb{A} \coloneqq \wT(A, \braidtri{A}{A}); \quad 
        \K{A, R} \coloneqq \K{A, \braidtri{A}{A}}; \\
        \KBA \coloneqq \KB{A, \braidtri{A}{A}}; \quad 
        \HB{A, R} \coloneqq \HB{A, \braidtri{A}{A}},
    \end{gather*}
    where the right hand sides are constructed from the braided coalgebra $(A, \braidtri{A}{A})$ (and an object $M$) in $\AMod$ as in Proposition~\ref{prop:nonstr-augbrobj} and Definition~\ref{defn:br-cpx}~\ref{br-cpx:ABSO}--\ref{br-cpx:br-coho}.
    We refer to $\KBA$ and $\HB{A, R}$ as the \textit{braided cochain complex} and the \textit{braided cohomology of $\qtb{A}$ with coefficients in $M$}, respectively.
\end{defn}

\begin{rem}
    We use the notation of Definition~\ref{defn:qtb}.
    Recall that a quasi-triangular $\Bbbk$-bialgebra $\qtb{A}$ is said to be \term{triangular} if $R^{-1} = R_2 \ot R_1$ (see \cite[Definition 8.3.3]{etingof-gelaki-nikshych-ostrik}).
    In this case, the equality $(\braidtrivar)^2 = \id{}$ holds, and given an $A$-module $M$, the cochain complex $\KB{A, R}$ and its cohomology $\HB{A, R}$ are constructed using the representations $\func{\symmgrp}{\Aut{\cat{V}}{\K[n + 1]{A, R}}}$ of the symmetric groups (see Remark~\ref{rem:HS}), and are denoted by $\KS{A, R}$ and $\HS{A, R}$, which we call the \textit{symmetric cochain complex} and the \textit{symmetric cohomology of $\qtb{A}$ with coefficients in $M$}, respectively.
    In particular, when a $\Bbbk$-bialgebra $\bialg{A}$ is \term{cocommutative} in the sense that $\tau \circ \Delta = \Delta$ holds (see \cite[Definition III.1.1 (a)]{kassel}), where $\tau$ denotes the flip morphism of Definition~\ref{defn:qtb}, the pair $(A, 1 \ot 1)$ is a triangular $\Bbbk$-bialgebra, and the cochain complex $\KS{A, 1 \ot 1}$ is canonically identified (by evident $\symmgrp$-equivariant isomorphisms) with the homogeneous cochain complex $\KS{A}$, defined in \cite[page 12]{shiba-sanada-itaba}, which is isomorphic to the cochain complex $\KS{G}$ of \cite[Definition 3.1]{pirashvili-A} for a group $G$, when $\Bbbk = \mathbb{Z}$ 
    and $A$ is the group algebra of $G$ over $\Bbbk$.
    In particular, Proposition~\ref{prop:KB} contains \cite[Proposition 3.2]{shiba-sanada-itaba} as a special case.
\end{rem}

We now give an explicit description of the {\ABSO} $\wT\qtb{A}$, in order to highlight that the braided cochain complex $\KBA$ and the braided cohomology $\HB{A, R}$ in the sense of Definition~\ref{defn:KBA} are, up to isomorphism, generalizations of the homogeneous cochain complex and the symmetric cohomology of cocommutative $\Bbbk$-Hopf algebras, defined as in \cite{shiba-sanada-itaba}.

\begin{rem}
    \label{rem:explicit}
    Let $\qtribialg{A}$ be a quasi-triangular $\Bbbk$-bialgebra.
    Using Remark~\ref{rem:wTC}, $\wT\qtb{A}$ can be identified with the {\ABSO} which is determined from the collection $\famnat{A^{\ot {\myspace} n}}$ of $A$-modules, together with the collections $\family{\face{j}{n}}{0 \le j \le n}$, $\family{\dege{j}{n}}{0 \le j \le n}$, and $\family{\perm{i}{n}}{1 \le i \le n}$ of $A$-linear maps, which are defined by the equalities (\ref{eq:face}), (\ref{eq:dege}), and (\ref{eq:perm}), which in this case admit more concrete forms as
    \begin{align}
        \face{j}{n}(a_0 \ot \cdots \ot a_n) &= \eps(a_j) (a_0 \ot \cdots \ot a_{j - 1} \ot a_{j + 1} \ot \cdots \ot a_n); \nonumber \\
        \dege{j}{n}(a_0 \ot \cdots \ot a_n) &= a_0 \ot \cdots \ot a_{j - 1} \ot \sw{a_j}{1} \ot \sw{a_j}{2} \ot a_{j + 1} \ot \cdots \ot a_n; \nonumber \\
        \perm{i}{n}(a_0 \ot \cdots \ot a_n) &= a_0 \ot \cdots \ot a_{i - 2} \ot (\swvar{2} a_i) \ot (\swvar{1} a_{i - 1}) \ot a_{i + 1} \ot \cdots \ot a_n, \label{eq:perm-qtb}
    \end{align}
    where we used the Sweedler notation $\Delta(a_j) \coloneqq \sw{a_j}{1} \ot \sw{a_j}{2}$.
    In particular, when $\bialg{A}$ is cocommutative and $R = 1 \ot 1$, the right hand side of (\ref{eq:perm-qtb}) is simplified as $a_0 \ot \cdots \ot a_{i - 2} \ot a_i \ot a_{i - 1} \ot a_{i + 1} \ot \cdots \ot a_n$, and this already appears in the right hand side of \cite[(3.1.1)]{pirashvili-A} and that of \cite[(3.2)]{shiba-sanada-itaba}.
\end{rem}

Our main result asserts that the braided cochain complex of quasi-triangular bialgebras is a braided Morita invariant under some condition:

\begin{thm}[Theorem~\ref{mainthm:B}]
    \label{thm:Morita}
    Let $\qtb{A} = \qtribialg{A}$ and $\qtbv{B} = \qtribialgv{B}$ be quasi-triangular $\Bbbk$-bialgebras, and suppose that $\func[\monfct{F}]{(\AMod, \braidtrivar)}{(\BMod, \braidtrivarv)}$ is a $\Bbbk$-linear braided monoidal equivalence such that $\ff[B] \circ F \iso \ff[A]$ as $\Bbbk$-linear functors.
    Then, there exists an isomorphism $\KBA \iso \KB{B, R'}[F(M)]$ 
    for all $A$-modules $M$.
    In particular, the braided cohomology of quasi-triangular $\Bbbk$-bialgebras is invariant under braided Morita equivalences which commute with the forgetful functors up to $\Bbbk$-linear natural isomorphism.
\end{thm}

\begin{proof}
    Take a $\Bbbk$-linear natural isomorphism $\func[\xi']{\ff[B] \circ F}{\ff[A]}$.
    It follows from a discussion similar to that after Etingof--Gelaki \cite[Theorem 4.3.1]{etingof-gelaki-A} that there exist an isomorphism $\func[\beta]{B}{A}$ of $\Bbbk$-algebras and a $\Bbbk$-linear natural isomorphism $\func[\xi]{F}{{\beta}^*}$, where $\func[\beta^*]{\vMod{A}}{\BMod}$ denotes the pullback functor induced from $\beta$ (see \cite[Section XI.4, Example 2]{kassel} for the definition).
    Indeed, $\beta$ is given explicitly by $\beta(b) \coloneqq \xi'_A(b \, \xi^{\prime {\myspace} -1}_A(1))$ for all $b \in B$, and $\xi$ is the unique $\Bbbk$-linear natural isomorphism such that $\ff[B] \star \xi = \xi'$.
    Note that $\beta$ can be regarded as a $B$-linear map $\func{B}{\beta^*(A)}$.
    If we set $\lambda \coloneqq (\xi_{{\myspace} \Bbbk} \circ \monfctuni{F}) (1)$, which is a unit in $\Bbbk$ since $\xi_{{\myspace} \Bbbk} \circ \monfctuni{F}$ is an automorphism of $\Bbbk$, 
    then, $(F, \lambda \monfctmulvar{F}, \lambda^{-1} \monfctuni{F})$ is also a $\Bbbk$-linear braided monoidal equivalence; this rescaling of $(\monfctmulvar{F}, \monfctuni{F})$ is taken from the proof of \cite[Theorem 2.2]{ng-schauenburg}.
    We now claim that the $B$-linear map $\func[\alpha \coloneqq \xi_A^{-1} \circ \beta]{B}{F(A)}$ satisfies (\ref{eq:alpha-1}) with respect to this monoidal functor.
    Since we have $F(\eps) \circ \alpha = F(\eps) \circ \xi_{A}^{-1} \circ \beta 
    = \xi_{{\myspace} \Bbbk}^{-1} \circ \eps \circ \beta$ by the naturality of $\xi$, it suffices to prove that the equality $(\lambda^{-1} \monfctuni{F} \circ \eps')(b) = (\xi_{{\myspace} \Bbbk}^{-1} \circ \eps \circ \beta)(b)$ holds for all $b \in B$, which can be verified as follows:
    \begin{align*}
        (\lambda^{-1} \monfctuni{F} \circ \eps') (b) 
        &= \lambda^{-1} \monfctuni{F} (\eps'(b) {\myspace} 1) 
        = \lambda^{-1} \monfctuni{F} (b {\myspace} 1) 
        = \lambda^{-1} (\xi_{{\myspace} \Bbbk}^{-1} \circ \xi_{{\myspace} \Bbbk} \circ \monfctuni{F}) (b {\myspace} 1) 
        = \lambda^{-1} \xi_{{\myspace} \Bbbk}^{-1} (b (\xi_{{\myspace} \Bbbk} \circ \monfctuni{F}) (1)) \\
        &= \lambda^{-1} \xi_{{\myspace} \Bbbk}^{-1} (\lambda b {\myspace} 1) 
        = \xi_{{\myspace} \Bbbk}^{-1} (b {\myspace} 1) 
        = \xi_{{\myspace} \Bbbk}^{-1} (\beta(b) {\myspace} 1) 
        =\xi_{{\myspace} \Bbbk}^{-1} ((\eps \circ \beta) (b) {\myspace} 1) 
        = (\xi_{{\myspace} \Bbbk}^{-1} \circ \eps \circ \beta)(b).
    \end{align*}
    Since $\alpha$ is an isomorphism, it follows from Corollary~\ref{cor:mor-6} (see also Remark~\ref{rem:rel-isom}) that there is a collection $\vareta$ of natural isomorphisms such that $(F, \vareta)$ is a relative morphism $\func{\wTcal\qtb{A} \circ \incl}{\wTcal\qtbv{B} \circ \incl}$.
    By an argument after Definition~\ref{defn:lin-mon-eqv}, this gives the desired isomorphism $\KBA \iso \KB{B, R'}[F(M)]$.
\end{proof}

We now focus on the finite-dimensional case.
Let $\Bbbk$ be a field, and $\qtb{A}$ a finite-dimensional quasi-triangular $\Bbbk$-bialgebra.
In this case, $A$ is contained in the monoidal full subcategory $\Amod$ of $\AMod$, defined as in Notation~\ref{nota:alg}~\ref{enum:Amod}, and $\braidtrivar$ restricts to a braiding on $\Amod$ (also denoted $\braidtrivar$), and we thus obtain a braided coalgebra $(A, \braidtri{A}{A})$ in $\Amod$.

\begin{cor}
    \label{cor:Morita}
    Let $\qtb{A}$ and $\qtbv{B}$ be finite-dimensional quasi-triangular bialgebras over a field $\Bbbk$, and suppose that $\func[F]{(\Amod, \braidtrivar)}{(\Bmod, \braidtrivarv)}$ is a $\Bbbk$-linear braided monoidal equivalence.
    Then, there exists an isomorphism $\KBA \iso \KB{B, R'}[F(M)]$ 
    for all finitely generated $A$-modules $M$.
    In particular, the braided cohomology of finite-dimensional quasi-triangular bialgebras over a field is a braided Morita invariant.
\end{cor}

In order to deduce this from Theorem~\ref{thm:Morita}, we first need to prove the following result:

\begin{lem}
    \label{lem:extend-braided}
    Let $\qtb{A}$ and $\qtbv{B}$ be finite-dimensional quasi-triangular bialgebras over a field $\Bbbk$.
    Then, any $\Bbbk$-linear braided monoidal functor $\func[F = \monfct{F}]{(\Amod, \braidtrivar)}{(\Bmod, \braidtrivarv)}$ extends uniquely, up to isomorphism, to a $\Bbbk$-linear braided monoidal functor $\func[\wt{F} = \monfct{\wt{F}}]{(\AMod, \braidtrivar)}{(\BMod, \braidtrivarv)}$ which commutes with filtered colimits.
    If $F$ is a $\Bbbk$-linear braided monoidal equivalence, then so is $\wt{F}$.
\end{lem}

The proof of Lemma~\ref{lem:extend-braided} requires us to write modules as filtered colimits of finitely generated ones, in a way that is compatible with tensor products.
To achieve this, we write them as filtered colimits of all finitely generated submodules.
We also need to keep track of relevant isomorphisms explicitly, as treating them as abstract ones is insufficient to show that the induced functor commutes with the braidings.

\begin{rem}
    \label{rem:loc-coh}
    Let $A$ be a finite-dimensional bialgebra over a field $\Bbbk$.
    Then, the $\Bbbk$-linear category $\AMod$ is \term{locally coherent Grothendieck} in the sense of Herzog \cite[page 511]{herzog}, and $\Amod$ coincides with the full subcategory $\mathrm{fp}(\AMod)$ of \term{finitely presented} objects in $\AMod$, which satisfies the finiteness assumptions of Xu--Zheng \cite[page 69]{xu-zheng}.
    Since the tensor product functor on $\AMod$ commutes with filtered colimits, the aforementioned monoidal structure on $\AMod$ coincides (up to isomorphism) with the monoidal structure inherited from that on $\Amod$ as in \cite[Remark 3.4]{xu-zheng}.
\end{rem}

\begin{nota}
    Let $A$ be a finite-dimensional bialgebra over a field $\Bbbk$, and $M$ an $A$-module.
    \begin{enumerate}
        \item We write $\cat{J}_M$ for the partially ordered set consisting of finitely generated $A$-submodules of $M$. 
    \end{enumerate}
    We regard $\cat{J}_M$ as a (small) category in the usual way (see \cite[page 11]{maclane}).
    Since $\cat{J}_M$ is directed, it is filtered when regarded as a category, and the collection $\family{\iota_{M'} \colon M' \hookrightarrow M}{M' \in {\myspace \myspace} \cat{J}_M}$ of inclusion maps exhibits $M$ as a filtered colimit of the evident functor $\func[J_M]{\cat{J}_M}{\AMod}$.
    \begin{enumerate}[resume]
        \item Let $\cat{D}$ be a category, and $\fct[G]{\AMod}{D}$ a functor.
        We write $\colim[M' \subset M]{G(M')}$ or $\colim[\cat{J}_M]{G(M')}$ for the filtered colimit $\colim[\cat{J}_M]{\hspace{-0.1em} (G \circ J_M)}$ of the composition $\cat{J}_M \xrightarrow{J_M} \AMod \longxrightarrow{G} \cat{D}$, if it exists.
    \end{enumerate}
\end{nota}

\begin{proof}[Proof of \emph{Lemma~\ref{lem:extend-braided}}]
    Applying \cite[Lemma 3.6]{xu-zheng}, we can extend $F$ uniquely, up to isomorphism, to a $\Bbbk$-linear monoidal functor $\func[\wt{F}]{\AMod}{\BMod}$ which commutes with filtered colimits.
    Note that the proof of \cite[Lemma 3.6]{xu-zheng} remains valid when finitely presented objects are not necessarily rigid, and is thus applicable in this case.
    Without loss of generality, we may assume that the induced functor $\func{\Amod}{\Bmod}$ is equal to $F$ in the strict sense.
    We claim that $\wt{F}$ is a braided monoidal functor $\func{(\AMod, \braidtrivar)}{(\BMod, \braidtrivarv)}$.
    For an $A$-module $L$, we let $\family{\func[s_{L'}]{F(L')}{\colim[L' \subset L]{F(L')}}}{L' \in \cat{J}_L}$ denote the collection of $B$-linear maps forming the (filtered) colimit $\colim[L' \subset L]{F(L')}$, and we define $\func[\phi_L]{\colim[L' \subset L]{F(L')}}{\wt{F}(L)}$ to be the unique $B$-linear map such that $\phi_L \circ s_{L'} = \wt{F}(\iota_{L'})$ for all $L' \in \cat{J}_L$, which is an isomorphism since $\wt{F}$ commutes with filtered colimits.
    Take $A$-modules $M$ and $N$.
    Observe that the isomorphism $\func[\monfctmul{\wt{F}}{M}{N}]{\wt{F}(M) \ot \wt{F}(N)}{\wt{F}(M \ot N)}$ factors as a composition of isomorphisms of $B$-modules as follows:
    \begin{align*}
        \wt{F}(M) \ot \wt{F}(N) 
        & \xrightarrow{\phi_M^{-1} \ot \phi_N^{-1}} \Bigl(\dcolim[\cat{J}_M]{\hspace{0.1em} F(M')}\Bigr) \ot \Bigl(\dcolim[\cat{J}_N]{\hspace{0.1em} F(N')}\Bigr) 
        \xrightarrow{\gamma_{M, N}^{-1}} \dcolim[\cat{J}_M \times \cat{J}_N]{\hspace{-0.4em} (F(M') \ot F(N'))} \\
        & \xrightarrow{\colim[]{\monfctmul{F}{M'}{N'}}} 
        \dcolim[\cat{J}_M \times \cat{J}_N]{\hspace{-0.4em} F(M' \ot N')} 
        \xrightarrow{\chi_{M, N}} \dcolim[\cat{J}_{M \ot N}]{\hspace{-0.2em} F(X')} 
        \xrightarrow{\phi_{M \ot N}} \wt{F}(M \ot N).
    \end{align*}
    Here, $\gamma_{M, N}$ is defined to be the composition of the dotted isomorphisms below
    \[\begin{tikzcd}
        {F(M') \ot F(N')} &&&& {\colim[M', N']{(F(M') \ot F(N'))}} \\
        {F(M') \ot F(N')} && {\colim[N']{(F(M') \ot F(N'))}} && {\colim[M']{(\colim[N']{(F(M') \ot F(N'))})}} \\
        {F(M') \ot F(N')} && {F(M') \ot (\colim[N']{F(N')})} && {\colim[M']{(F(M') \ot (\colim[N']{F(N')}))}} \\
        {F(M') \ot F(N')} && {F(M') \ot (\colim[N']{F(N')})} && {(\colim[M']{F(M')}) \ot (\colim[N']{F(N')})} \\
        {F(M') \ot F(N')} && {F(M') \ot (\colim[N']{F(N')})} && {\wt{F}(M) \ot (\colim[N']{F(N')})} \\
        {F(M') \ot F(N')} && {F(M') \ot \wt{F}(N)} && {\wt{F}(M) \ot \wt{F}(N),}
        \arrow[from=1-1, to=1-5]
        \arrow["{\id{}}"', from=1-1, to=2-1]
        \arrow[dashed, from=1-5, to=2-5]
        \arrow[from=2-1, to=2-3]
        \arrow["{\id{}}"', from=2-1, to=3-1]
        \arrow[from=2-3, to=2-5]
        \arrow[from=2-3, to=3-3]
        \arrow[dashed, from=2-5, to=3-5]
        \arrow["{F(M') \ot s_{N'}}", from=3-1, to=3-3]
        \arrow["{\id{}}"', from=3-1, to=4-1]
        \arrow[from=3-3, to=3-5]
        \arrow["{\id{}}"', from=3-3, to=4-3]
        \arrow[dashed, from=3-5, to=4-5]
        \arrow["{F(M') \ot s_{N'}}", from=4-1, to=4-3]
        \arrow["{\id{}}"', from=4-1, to=5-1]
        \arrow["{s_{M'} \ot \varinjlim{F(N')}}", from=4-3, to=4-5]
        \arrow["{\id{}}"', from=4-3, to=5-3]
        \arrow["{\phi_M \ot (\varinjlim{F(N')})}", from=4-5, to=5-5]
        \arrow["{F(M') \ot s_{N'}}", from=5-1, to=5-3]
        \arrow["{\id{}}"', from=5-1, to=6-1]
        \arrow["{\wt{F}(\iota_{M'}) \ot (\varinjlim{F(N')})}", from=5-3, to=5-5]
        \arrow["{F(M') \ot \phi_N}"', from=5-3, to=6-3]
        \arrow["{\wt{F}(M) \ot \phi_N}", from=5-5, to=6-5]
        \arrow["{F(M') \ot \wt{F}(\iota_{N'})}"', from=6-1, to=6-3]
        \arrow["{\wt{F}(\iota_{M'}) \ot \wt{F}(N)}"', from=6-3, to=6-5]
    \end{tikzcd}\]
    and $\chi_{M, N}$ is the canonical isomorphism induced from the functor $\func{\cat{J}_M \times \cat{J}_N}{\cat{J}_{M \ot N}}[(M', N')][M' \ot N']$, whose inverse is constructed using the property that for all $X' \in \cat{J}_{M \ot N}$, we can choose $M' \in \cat{J}_M$ and $N' \in \cat{J}_N$ so that $X' \subset M' \ot N'$.
    Given a functor $\func[H]{\cat{J}_M \times \cat{J}_N}{\BMod}$, we will identify $\colim[\cat{J}_M \times \cat{J}_N]{H}$ with the filtered colimit of the composition $\cat{J}_N \times \cat{J}_M \to \cat{J}_M \times \cat{J}_N \xrightarrow{H} \BMod$.
    We now wish to prove the equality $\monfctmul{\wt{F}}{N}{M} \circ \braidtriv{\wt{F}(M)}{\wt{F}(N)} = \wt{F}(\braidtri{M}{N}) \circ \monfctmul{\wt{F}}{M}{N}$ (see also (\ref{eq:brmonfct})), which is equivalent to the equality
    \begin{align*}
        & \phi_{N \ot M} \circ \chi_{N, M} \circ \varinjlim{\monfctmul{F}{N'}{M'}} \circ \gamma_{N, M}^{-1} \circ (\phi_N^{-1} \ot \phi_M^{-1}) \circ \braidtriv{\wt{F}(M)}{\wt{F}(N)} \circ (\phi_M \ot \phi_N) \circ \gamma_{M, N} \\
        & \qquad \qquad \qquad \qquad \qquad \qquad \qquad \qquad \qquad = \wt{F}(\braidtri{M}{N}) \circ \phi_{M \ot N} \circ \chi_{M, N} \circ \varinjlim{\monfctmul{F}{M'}{N'}}.
    \end{align*}
    We write $t_{M', N'}$ for the evident $B$-linear map $\func{F(M') \ot F(N')}{\colim[\cat{J}_M \times \cat{J}_N]{(F(M') \ot F(N'))}}$.
    Since
    \begin{align*}
        \wt{F}(\braidtri{M}{N}) \circ \phi_{M \ot N} \circ \chi_{M, N} \circ t_{M', N'} 
        &= \wt{F}(\braidtri{M}{N}) \circ \phi_{M \ot N} \circ s_{M' \ot N'} \\
        &= \wt{F}(\braidtri{M}{N}) \circ \wt{F}(\iota_{M' \ot N'}) 
        = \wt{F}(\iota_{N' \ot M'}) \circ F(\braidtri{M'}{N'}) \\
        &= \phi_{N \ot M} \circ s_{N' \ot M'} \circ F(\braidtri{M'}{N'}) \\
        &= \phi_{N \ot M} \circ \chi_{N, M} \circ t_{N', M'} \circ F(\braidtri{M'}{N'}) \\
        &= \phi_{N \ot M} \circ \chi_{N, M} \circ \varinjlim{F(\braidtri{M'}{N'})} \circ t_{M', N'},
    \end{align*}
    where the third equality follows from (\ref{eq:brmoncat-1}) for $\braidtrivar$, $\iota_{N' \ot M'} = \iota_{N'} \ot \iota_{M'}$, and $\iota_{M' \ot N'} = \iota_{M'} \ot \iota_{N'}$, it follows that $\wt{F}(\braidtri{M}{N}) \circ \phi_{M \ot N} \circ \chi_{M, N} = \phi_{N \ot M} \circ \chi_{N, M} \circ \varinjlim{F(\braidtri{M'}{N'})}$, and it suffices to prove the equality
    \begin{equation}
        \label{eq:braided-lem}
        \begin{split}
            & \varinjlim{\monfctmul{F}{N'}{M'}} \circ \gamma_{N, M}^{-1} \circ (\phi_N^{-1} \ot \phi_M^{-1}) \circ \braidtriv{\wt{F}(M)}{\wt{F}(N)} \circ (\phi_M \ot \phi_N) \circ \gamma_{M, N} \\
            & \qquad \qquad \qquad \qquad \qquad \qquad \qquad \qquad \qquad \qquad = \varinjlim{F(\braidtri{M'}{N'})} \circ \varinjlim{\monfctmul{F}{M'}{N'}}.
        \end{split}
    \end{equation}
    On the other hand, we have
    \begin{align*}
        & \gamma_{N, M}^{-1} \circ (\phi_N^{-1} \ot \phi_M^{-1}) \circ \braidtriv{\wt{F}(M)}{\wt{F}(N)} \circ (\phi_M \ot \phi_N) \circ \gamma_{M, N} \circ t_{M', N'} \\
        & \quad = \gamma_{N, M}^{-1} \circ (\phi_N^{-1} \ot \phi_M^{-1}) \circ \braidtriv{\wt{F}(M)}{\wt{F}(N)} \circ (\wt{F}(\iota_{M'}) \ot \wt{F}(\iota_{N'})) \\
        & \quad = \gamma_{N, M}^{-1} \circ (\phi_N^{-1} \ot \phi_M^{-1}) \circ (\wt{F}(\iota_{N'}) \ot \wt{F}(\iota_{M'})) \circ \braidtriv{F(M')}{F(N')} \\
        & \quad = \gamma_{N, M}^{-1} \circ (s_{N'} \ot s_{M'}) \circ \braidtriv{F(M')}{F(N')} 
        = t_{N', M'} \circ \braidtriv{F(M')}{F(N')} 
        = \varinjlim{\braidtriv{F(M')}{F(N')}} \circ t_{M', N'},
    \end{align*}
    where the first and the fourth equalities follow from the commutativity of the above diagram, and the second equality from (\ref{eq:brmoncat-1}) for $\braidtrivarv$.
    Thus, the left hand side of (\ref{eq:braided-lem}) is $\varinjlim{\monfctmul{F}{N'}{M'}} \circ \varinjlim{\braidtriv{F(M')}{F(N')}}$.
    We are reduced to prove the equality $\monfctmul{F}{N'}{M'} \circ \braidtriv{F(M')}{F(N')} = F(\braidtri{M'}{N'}) \circ \monfctmul{F}{M'}{N'}$ for all $M' \in \cat{J}_M$ and $N' \in \cat{J}_N$, which follows since $F$ is a braided monoidal functor by virtue of our assumption.
\end{proof}

\begin{proof}[Proof of \emph{Corollary~\ref{cor:Morita}}]
    By Theorem~\ref{thm:Morita} and Lemma~\ref{lem:extend-braided}, and also the proof of \cite[Proposition 4.2.2]{etingof-gelaki-A}, it suffices to prove that $F^{-1}(B) \iso A$ as $A$-modules.
    Let $\{[S_1], \dots, [S_m]\}$ be the set of isomorphism classes of simple $A$-modules, and we write $P_i$ for the projective cover of $S_i$. 
    We note that $\{[P_1], \dots, [P_m]\}$ forms the set of isomorphism classes of finitely generated indecomposable projective $A$-modules, and given a finitely generated projective $A$-module $P$, the multiplicity of $P_i$ in $P$ is $\dim \hom{A}{P}{S_i} / \dim \End{A}{S_i}$; 
    see \cite[Theorem 4.3.1 (2)]{etingof-gelaki-A} for the special case when $P = A$ and $\Bbbk$ is algebraically closed.
    Arguing as in the proof of \cite[Proposition 4.2.2]{etingof-gelaki-A}, we are reduced to prove $\dim F(S_i) = \dim S_i$ for all $1 \le i \le m$.
    This follows from the proof of Etingof--Gelaki \cite[Theorem 6.1]{etingof-gelaki-B}, using the Frobenius-Perron theorem (cf.~\cite[Theorem 3.2.1 (3)]{etingof-gelaki-nikshych-ostrik}) and the fact that for all $1 \le i, j, j' \le m$, the multiplicity of $F(S_j)$ in the composition series of $F(S_i) \ot F(S_{j'}) \iso F(S_i \ot S_{j'})$ is equal to that of $S_j$ in the composition series of $S_i \ot S_{j'}$.
    The last assertion follows since $F$ is a $\Bbbk$-linear monoidal equivalence.
\end{proof}

\begin{rem}
    Replacing $\dim$ by the length as a $\Bbbk$-module, Corollary~\ref{cor:Morita} and the preceding remark hold in a more general setting where $\Bbbk$ is a nonzero commutative artinian ring and $A$ and $B$ are finitely generated as $\Bbbk$-modules; that is, when 
    $A$ and $B$ are \textit{artinian $\Bbbk$-algebras}.
\end{rem}

\subsection{Examples}
\label{subsec:5.2}

Let $\Bbbk$ be a field of characteristic different from two.
We consider the $8$-dimensional Hopf algebra $\bialg{E(2)}$ introduced by Radford \cite{radford-B}, where $E(2)$ is generated by $g$, $x$, and $y$, subject to the relations $g^2 = 1$, $x^2 = 0$, $y^2 = 0$, $xg = - gx$, $yg = - gy$, and $yx = - xy$, and $\Delta$, $\eps$, and the antipode $S$ are determined by $\Delta(g) \coloneqq g \ot g$, $\Delta(x) \coloneqq x \ot g + 1 \ot x$, $\Delta(y) \coloneqq y \ot g + 1 \ot y$, $\eps(g) \coloneqq 1$, $\eps(x) \coloneqq 0$, $\eps(y) \coloneqq 0$, $S(g) \coloneqq g$, $S(x) \coloneqq gx$, and $S(y) \coloneqq gy$.
It follows from Gelaki \cite[Proposition 2.2.3.1]{gelaki} that every R-matrix of $E(2)$ is written in the form $R_X$ for an arbitrary $2 \times 2$ matrix $X = \bigl[\sm{a}{b}{c}{d}\bigr]$ with values in $\Bbbk$, where $R_X$ is defined as follows:
\begin{align*}
    & \frac{1}{2} (1 \ot 1 + 1 \ot g + g \ot 1 - g \ot g 
    + a (x \ot x + x \ot gx - gx \ot x + gx \ot gx) \\
    & \quad + b (x \ot y + x \ot gy - gx \ot y + gx \ot gy) 
    + c (y \ot x + y \ot gx - gy \ot x + gy \ot gx) \\
    & \quad + d (y \ot y + y \ot gy - gy \ot y + gy \ot gy) 
    + (bc - ad) (xy \ot xy + xy \ot gxy - gxy \ot xy + gxy \ot gxy)).
\end{align*}

\begin{prop}
    \label{prop:E2-1}
    Let $X$ be a $2 \times 2$ matrix with values in $\Bbbk$, and write $R_X$ for the corresponding R-matrix of $E(2)$.
    Then, the dimension of the $\Bbbk$-vector space $\KB[1]{E(2), R_X}[E(2)]$ is determined as follows\emph{:}
    \[\dim \KB[1]{E(2), R_X}[E(2)] = \begin{cases*}
        32 & \emph{if $X$ is symmetric,} \\
        20 & \emph{otherwise.}
    \end{cases*}\]
\end{prop}

\begin{prop}
    \label{prop:E2-2}
    The braided Morita equivalence classes of quasi-triangular $\Bbbk$-bialgebras of the form $(E(2), R)$ are $(E(2), R_O)$ and $(E(2), R_J)$, where $O = \bigl[\sm{0}{0}{0}{0}\bigr]$ and $J = \bigl[\sm{0}{1}{- 1}{0}\bigr]$.
\end{prop}

\begin{proof}[Proof of \emph{Proposition~\ref{prop:E2-2}} assuming \emph{Proposition~\ref{prop:E2-1}}]
    We put $J(\lambda) \coloneqq \bigl[\sm{0}{\lambda}{- \lambda}{0}\bigr]$ for $\lambda \in \Bbbk$, so that $O = J(0)$ and $J = J(1)$.
    Let $X= \bigl[\sm{a}{b}{c}{d}\bigr]$ be a $2 \times 2$ matrix with values in $\Bbbk$.
    It follows from Carnovale--Cuadra \cite[Proposition 2.4]{carnovale-cuadra} that $(E(2), R_X)$ is braided Morita equivalent to $(E(2), R_{(X - X^t) / 2}) = (E(2), R_{J((b - c) / 2)})$.
    On the other hand, for any $\lambda \in \Bbbk^\times$, it follows from \cite[Theorem 2.2.2]{gelaki} that $(E(2), R_{J(\lambda)})$ is isomorphic to $(E(2), R_{J(1)})$ since we have $J(\lambda) = \bigl[\sm{1}{0}{0}{\lambda}\bigr] J(1) \bigl[\sm{1}{0}{0}{\lambda}\bigr]$.
    Hence, $(E(2), R_X)$ is braided Morita equivalent to $(E(2), R_O)$ when $X$ is symmetric, and to $(E(2), R_J)$ otherwise.
    Now, it remains to show that $(E(2), R_O)$ and $(E(2), R_J)$ are not braided Morita equivalent.
    Suppose not, and let $\func[F]{(\Mod{E(2)}, R_O)}{(\Mod{E(2)}, R_J)}$ be a $\Bbbk$-linear braided monoidal equivalence.
    Since $E(2)$ is finite-dimensional, it follows from the proofs of Theorem~\ref{thm:Morita} and Corollary~\ref{cor:Morita} that $F \iso \beta^*$ for some isomorphism $\func[\beta]{E(2)}{E(2)}$ of $\Bbbk$-algebras, so that $F(E(2)) \iso \beta^*(E(2)) \iso E(2)$ as $E(2)$-modules.
    Combining this with Theorem~\ref{thm:Morita}, we obtain that
    \[\KB[1]{E(2), R_O}[E(2)] \iso \KB[1]{E(2), R_J}[F(E(2))] \iso \KB[1]{E(2), R_J}[E(2)].\]
    This contradicts the following inequality, which is an immediate consequence of Proposition~\ref{prop:E2-1}:
    \[\dim \KB[1]{E(2), R_O}[E(2)] = 32 > 20 = \dim \KB[1]{E(2), R_J}[E(2)]. \qedhere\]
\end{proof}

The proof of Proposition~\ref{prop:E2-1} will occupy the remainder of this subsection.
Arguing as in the proof of Proposition~\ref{prop:E2-2}, we are reduced to prove the following equality:
\begin{equation}
    \label{eq:lambda}
    \dim \KB[1]{E(2), R_{J(\lambda)}}[E(2)] = \begin{cases*}
        32 & if $\lambda = 0$, \\
        20 & if $\lambda \ne 0$.
    \end{cases*}
\end{equation}
For simplicity, we write $E \coloneqq E(2)$ and $R_\lambda \coloneqq R_{J(\lambda)}$.
We identify $\K[2]{E, R_\lambda}[E] = \hom{E}{\wT[2](E, R_\lambda)}{E}$ with $\hom{E}{E \ot E}{E}$ as in Remark~\ref{rem:wTC}.
We note that the braid group $\brgrp[2]$ 
acts on $\hom{E}{E \ot E}{E}$ as
\begin{align*}
    (\brperm{1}{1} \varphi) (a_0 \ot a_1) 
    = \sgn[\brperm{1}{1}] (\varphi \circ \perm{1}{1}) (a_0 \ot a_1) 
    = - \varphi((\swvarL{2} a_1) \ot (\swvarL{1} a_0))
\end{align*}
for all $a_0, a_1 \in E$ and all $E$-linear maps $\func[\varphi]{E \ot E}{E}$ (see Remark~\ref{rem:sign} and (\ref{eq:perm-qtb})).
We will further make use of an isomorphism $\func{\hom{E}{E \ot E}{E}}{\hom{\Bbbk}{E}{E}}[\varphi][\varphi(1 \ot (-))]$, whose inverse is given by $\phi \mapsto [a_0 \ot a_1 \mapsto \sw{a_0}{1} \phi(S(\sw{a_0}{2}) a_1)]$, (see \cite[page 186]{witherspoon}, for example).
Take an $E$-linear map $\func[\varphi]{E \ot E}{E}$.
We put $\phi \coloneqq \varphi(1 \ot (-))$, $\gbf{+} \coloneqq 1 + g$, and $\gbf{-} \coloneqq 1 - g$.
Since
\begin{align*}
    (\brperm{1}{1} \varphi)(1 \ot a) 
    &= - \varphi((\swvarL{2} a) \ot \swvarL{1}) 
    = - \sw{(\swvarL{2} a)}{1} \phi(S(\sw{(\swvarL{2} a)}{2}) \swvarL{1}) \\
    &= - \sw{\swvarL{2}}{1} \sw{a}{1} \phi(S(\sw{a}{2}) S(\sw{\swvarL{2}}{2}) \swvarL{1})
\end{align*}
for all $a \in E$, $\varphi$ belongs to $\KB[1]{E, R_\lambda}[E]$ if and only if 
$\sw{\swvarL{2}}{1} \sw{a}{1} \phi(S(\sw{a}{2}) S(\sw{\swvarL{2}}{2}) \swvarL{1}) = - \phi(a)$ holds for all $a \in E$; the left hand side is written as
\begin{align*}
    & \frac{1}{2} (\gbf{-} \sw{a}{1} \phi(S(\sw{a}{2})) + \gbf{+} \sw{a}{1} \phi(S(\sw{a}{2}) g) \\
    & \quad \qquad - \lambda \gbf{-} y \sw{a}{1} \phi(S(\sw{a}{2}) x) + \lambda \gbf{+} y \sw{a}{1} \phi(S(\sw{a}{2}) gx) \\
    & \quad\qquad + \lambda \gbf{-} x \sw{a}{1} \phi(S(\sw{a}{2}) y) - \lambda \gbf{+} x \sw{a}{1} \phi(S(\sw{a}{2}) gy) \\
    & \quad \qquad - \lambda \gbf{-} (2 + \lambda xy) \sw{a}{1} \phi(S(\sw{a}{2}) xy) 
    - \lambda \gbf{+} (2 + \lambda xy) \sw{a}{1} \phi(S(\sw{a}{2}) gxy)).
\end{align*}
Hence, $\varphi \in \KB[1]{E, R_\lambda}[E]$ holds if and only if $(\phi(1), \phi(g), \phi(x), \phi(gx), \phi(y), \phi(gy), \phi(xy), \phi(gxy)) \in E^{\oplus 8}$ belongs to the null space of the following $8 \times 8$ matrix with values in $E$:
{\scriptsize\[\begin{bmatrix}
    2 + \gbf{-} & \gbf{+} & - \lambda \gbf{-} y & \lambda \gbf{+} y & 
    \lambda \gbf{-} x & - \lambda \gbf{+} x & - \lambda \gbf{-} (2 + \lambda xy) & - \lambda \gbf{+} (2 + \lambda xy) \\
    \gbf{+} & \gbf{+} & - \lambda \gbf{+} y & - \lambda \gbf{-} y & 
    \lambda \gbf{+} x & \lambda \gbf{-} x & - \lambda \gbf{+} (2 + \lambda xy) & \lambda \gbf{-} (2 + \lambda xy) \\
    \gbf{+} x & \gbf{-} x & \gbf{-} - \lambda \gbf{+} xy & \gbf{-} (1 + \lambda xy) & 
    0 & 0 & - \lambda \gbf{+} x & - \lambda \gbf{-} x \\
    - \gbf{-} x & \gbf{+} x & \gbf{-} (1 + \lambda xy) & 2 + \gbf{+} + \lambda \gbf{+} xy & 
    0 & 0 & \lambda \gbf{-} x & - \lambda \gbf{+} x \\
    \gbf{+} y & \gbf{-} y & 0 & 0 & 
    \gbf{-} - \lambda \gbf{+} xy & \gbf{-} (1 + \lambda xy) & - \lambda \gbf{+} y & - \lambda \gbf{-} y \\
    - \gbf{-} y & \gbf{+} y & 0 & 0 & 
    \gbf{-} (1 + \lambda xy) & 2 + \gbf{+} + \lambda \gbf{+} xy & \lambda \gbf{-} y & - \lambda \gbf{+} y \\
    \gbf{-} xy & \gbf{+} xy & \gbf{-} y & - \gbf{+} y & 
    - \gbf{-} x & \gbf{+} x & 2 + \gbf{-} & \gbf{+} \\
    \gbf{+} xy & - \gbf{-} xy & \gbf{+} y & \gbf{-} y & 
    - \gbf{+} x & - \gbf{-} x & \gbf{+} & \gbf{+}
\end{bmatrix}\raisebox{-3.75em}{.}\]
}By elementary row operations, this matrix is converted to
{\small\begin{equation}
    \label{eq:matrix}
    \begin{bmatrix}
        2 & \gbf{+} (1 + 2 \lambda xy) & - \lambda \gbf{+} y & 0 & 
        \lambda \gbf{+} x & 0 & 0 & - \lambda \gbf{+} (2 + 7 \lambda xy) \\
        0 & \lambda (5 + 3 g) xy & - 2 \lambda \gbf{+} y & 0 & 
        2 \lambda \gbf{+} x & 0 & 0 & - \lambda ((6 + 2 g) + 2 \lambda (11 + 9 g) xy) \\
        0 & 0 & \lambda \gbf{+} xy & 0 & 
        0 & 0 & 0 & - 4 \lambda \gbf{-} x \\
        0 & (3 - g) x & 2 \gbf{-} (1 - \lambda xy) & 4 & 
        0 & 0 & 0 & - \lambda (7 - 5 g) x \\
        0 & 0 & 0 & 0 & 
        \lambda \gbf{+} xy & 0 & 0 & - 4 \lambda \gbf{-} y \\
        0 & (3 - g) y & 0 & 0 & 
        2 \gbf{-} (1 - \lambda xy) & 4 & 0 & - \lambda (7 - 5 g) y \\
        0 & - 2 \gbf{+} xy & (3 + g) y & 0 & 
        - (3 + g) x & 0 & 4 & 2 \gbf{+} (1 + 6 \lambda xy) \\
        0 & 0 & 0 & 0 & 
        0 & 0 & 0 & \lambda (5 + 3 g) xy
    \end{bmatrix}\raisebox{-3.75em}[0pt][0pt]{\makebox[0pt]{.}}
\end{equation}
}We thus have a monomorphism $\func{\KB[1]{E, R_\lambda}[E]}{E^{\oplus 4}}[\varphi][(\phi(g), \phi(x), \phi(y), \phi(gxy))]$, which identifies $\KB[1]{E, R_\lambda}[E]$ with the null space of the following $4 \times 4$ matrix with values in $E$:
\begin{equation*}
    \mathsf{M}_\lambda \coloneqq {\small\begin{bmatrix}
        \lambda (5 + 3 g) xy & - 2 \lambda \gbf{+} y & 2 \lambda \gbf{+} x & - \lambda ((6 + 2 g) + 2 \lambda (11 + 9 g) xy) \\
        0 & \lambda \gbf{+} xy & 0 & - 4 \lambda \gbf{-} x \\
        0 & 0 & \lambda \gbf{+} xy & - 4 \lambda \gbf{-} y \\
        0 & 0 & 0 & \lambda (5 + 3 g) xy
    \end{bmatrix}}\raisebox{-3.5em}[0pt][0pt]{\makebox[0pt]{.}}
\end{equation*}
Let $M_8(\Bbbk)$ denote the $\Bbbk$-algebra of $8 \times 8$ matrices with values in $\Bbbk$.
We now regard $\mathsf{M}_\lambda$ as a $32 \times 32$ matrix with values in $\Bbbk$ using the injective map $\func[h]{E}{M_8(\Bbbk)}$ of $\Bbbk$-algebras, which sends $a \in E$ to the matrix of the $\Bbbk$-linear map $\func[a(-)]{E}{E}$ with respect to the $\Bbbk$-basis $(m_1, \dots, m_8) \coloneqq (1, g, x, gx, y, gy, xy, gxy)$ of $E$.
In other words, the $8 \times 8$ matrix $h(a) = [\mu_{ij}]$ is characterized by $a m_j = \sum_i \mu_{ij} m_i$ for all $1 \le j \le 8$.
More concretely, the map $\func[h]{E}{M_8(\Bbbk)}$ of $\Bbbk$-algebras is defined on generators as follows:
\begin{gather*}
    h(g) = {\small\begin{bmatrix}
        0 & 1 & 0 & 0 & 0 & 0 & 0 & 0 \\
        1 & 0 & 0 & 0 & 0 & 0 & 0 & 0 \\
        0 & 0 & 0 & 1 & 0 & 0 & 0 & 0 \\
        0 & 0 & 1 & 0 & 0 & 0 & 0 & 0 \\
        0 & 0 & 0 & 0 & 0 & 1 & 0 & 0 \\
        0 & 0 & 0 & 0 & 1 & 0 & 0 & 0 \\
        0 & 0 & 0 & 0 & 0 & 0 & 0 & 1 \\
        0 & 0 & 0 & 0 & 0 & 0 & 1 & 0
    \end{bmatrix}}; \\
    h(x) = {\small\begin{bmatrix}
        0 & 0 & 0 & 0 & 0 & 0 & 0 & 0 \\
        0 & 0 & 0 & 0 & 0 & 0 & 0 & 0 \\
        1 & 0 & 0 & 0 & 0 & 0 & 0 & 0 \\
        0 & - 1 & 0 & 0 & 0 & 0 & 0 & 0 \\
        0 & 0 & 0 & 0 & 0 & 0 & 0 & 0 \\
        0 & 0 & 0 & 0 & 0 & 0 & 0 & 0 \\
        0 & 0 & 0 & 0 & 1 & 0 & 0 & 0 \\
        0 & 0 & 0 & 0 & 0 & - 1 & 0 & 0
    \end{bmatrix}}; \quad 
    h(y) = {\small\begin{bmatrix}
        0 & 0 & 0 & 0 & 0 & 0 & 0 & 0 \\
        0 & 0 & 0 & 0 & 0 & 0 & 0 & 0 \\
        0 & 0 & 0 & 0 & 0 & 0 & 0 & 0 \\
        0 & 0 & 0 & 0 & 0 & 0 & 0 & 0 \\
        1 & 0 & 0 & 0 & 0 & 0 & 0 & 0 \\
        0 & - 1 & 0 & 0 & 0 & 0 & 0 & 0 \\
        0 & 0 & 1 & 0 & 0 & 0 & 0 & 0 \\
        0 & 0 & 0 & - 1 & 0 & 0 & 0 & 0
    \end{bmatrix}}\raisebox{-3.5em}[0pt][0pt]{\makebox[0pt]{.}}
\end{gather*}
We can reduce $\mathsf{M}_\lambda$ to its row echelon form $\begin{bmatrix}
    \mathsf{N}_{\lambda, 1} & \mathsf{N}_{\lambda, 2} \\
    O_{20, 16} & O_{20, 16}
\end{bmatrix}$, where $\mathsf{N}_{\lambda, 1}$ and $\mathsf{N}_{\lambda, 2}$ are matrices of size $12 \times 16$, which are displayed as follows:

\vspace{-0.5em}

\begin{table}[H]
\begin{center}
\caption{Description of the matrices $\mathsf{N}_{\lambda, 1}$ and $\mathsf{N}_{\lambda, 2}$}
\begin{tabular}{|c|c|}
    \hline 
    $\mathsf{N}_{\lambda, 1}$ & {\footnotesize$\begin{bmatrix}
        \lambda & 7 \lambda & 0 & 0 & 0 & 0 & 0 & 0 & 
        0 & 0 & - 2 \lambda & 2 \lambda & 0 & 0 & 0 & 0 \\
        0 & 4 \lambda & 0 & 0 & 0 & 0 & 0 & 0 & 
        0 & 0 & - \lambda & \lambda & 0 & 0 & 0 & 0 \\
        0 & 0 & 0 & 0 & 0 & 0 & 0 & 0 & 
        \lambda & 0 & 0 & 0 & 0 & 0 & 0 & 0 \\
        0 & 0 & 0 & 0 & 0 & 0 & 0 & 0 & 
        0 & \lambda & 0 & 0 & 0 & 0 & 0 & 0 \\
        0 & 0 & 0 & 0 & 0 & 0 & 0 & 0 & 
        0 & 0 & 0 & 0 & 0 & 0 & 0 & 0 \\
        0 & 0 & 0 & 0 & 0 & 0 & 0 & 0 & 
        0 & 0 & 0 & 0 & 0 & 0 & 0 & 0 \\
        0 & 0 & 0 & 0 & 0 & 0 & 0 & 0 & 
        0 & 0 & 0 & 0 & 0 & 0 & 0 & 0 \\
        0 & 0 & 0 & 0 & 0 & 0 & 0 & 0 & 
        0 & 0 & 0 & 0 & 0 & 0 & 0 & 0 \\
        0 & 0 & 0 & 0 & 0 & 0 & 0 & 0 & 
        0 & 0 & 0 & 0 & 0 & 0 & 0 & 0 \\
        0 & 0 & 0 & 0 & 0 & 0 & 0 & 0 & 
        0 & 0 & 0 & 0 & 0 & 0 & 0 & 0 \\
        0 & 0 & 0 & 0 & 0 & 0 & 0 & 0 & 
        0 & 0 & 0 & 0 & 0 & 0 & 0 & 0 \\
        0 & 0 & 0 & 0 & 0 & 0 & 0 & 0 & 
        0 & 0 & 0 & 0 & 0 & 0 & 0 & 0
    \end{bmatrix}$} \\
    \hline 
    $\mathsf{N}_{\lambda, 2}$ & {\footnotesize$\begin{bmatrix}
        0 & 0 & 0 & 0 & 2 \lambda & - 2 \lambda & 0 & 0 & 
        0 & 0 & 0 & 0 & 0 & 0 & 2 \lambda & - 10 \lambda \\
        0 & 0 & 0 & 0 & \lambda & - \lambda & 0 & 0 & 
        0 & 0 & 0 & 0 & 0 & 0 & 2 \lambda & - 6 \lambda \\
        0 & 0 & 0 & 0 & 0 & 0 & 0 & 0 & 
        0 & 0 & 0 & 0 & 0 & 0 & 0 & 0 \\
        0 & 0 & 0 & 0 & 0 & 0 & 0 & 0 & 
        0 & 0 & 0 & 0 & 0 & 0 & 0 & 0 \\
        \lambda & 0 & 0 & 0 & 0 & 0 & 0 & 0 & 
        0 & 0 & 0 & 0 & 0 & 0 & 0 & 0 \\
        0 & \lambda & 0 & 0 & 0 & 0 & 0 & 0 & 
        0 & 0 & 0 & 0 & 0 & 0 & 0 & 0 \\
        0 & 0 & 0 & 0 & 0 & 0 & 0 & 0 & 
        \lambda & 0 & 0 & 0 & 0 & 0 & 0 & 0 \\
        0 & 0 & 0 & 0 & 0 & 0 & 0 & 0 & 
        0 & \lambda & 0 & 0 & 0 & 0 & 0 & 0 \\
        0 & 0 & 0 & 0 & 0 & 0 & 0 & 0 & 
        0 & 0 & \lambda & 0 & 0 & 0 & 0 & 0 \\
        0 & 0 & 0 & 0 & 0 & 0 & 0 & 0 & 
        0 & 0 & 0 & \lambda & 0 & 0 & 0 & 0 \\
        0 & 0 & 0 & 0 & 0 & 0 & 0 & 0 & 
        0 & 0 & 0 & 0 & \lambda & 0 & 0 & 0 \\
        0 & 0 & 0 & 0 & 0 & 0 & 0 & 0 & 
        0 & 0 & 0 & 0 & 0 & \lambda & 0 & 0
    \end{bmatrix}$} \\
    \hline 
\end{tabular}
\end{center}
\end{table}

\vspace{-0.5em}

\noindent We therefore compute
\[\dim \KB[1]{E, R_\lambda}[E] 
= \dim \Ker{\mathsf{M}_\lambda} 
= 32 - \rank{\mathsf{M}_\lambda} 
= \begin{cases*}
    32 & if $\lambda = 0$; \\
    32 - 12 & if $\lambda \ne 0$,
\end{cases*} 
= \begin{cases*}
    32 & if $\lambda = 0$; \\
    20 & if $\lambda \ne 0$,
\end{cases*}\]
which yields the desired equality (\ref{eq:lambda}).

\begin{rem}
    Let $\Bbbk$ be a field of characteristic different from two.
    We write $E(1)$ for the $4$-dimensional Sweedler Hopf algebra $\bialg{E(1)}$, where $E(1)$ is generated by $g$ and $x$, subject to the relations $g^2 = 1$, $x^2 = 0$, and $xg = - gx$, and $\Delta$, $\eps$, and the antipode $S$ are defined in such a way that the evident surjective map $\func{E(2)}{E(1)}$ of $\Bbbk$-algebras can be promoted to 
    an isomorphism $\func{E(2) / (y)}{E(1)}$ of $\Bbbk$-Hopf algebras.
    It follows from Radford \cite[pages 296 and 297]{radford-A} that every R-matrix of $E(1)$ is written in the form
    \begin{align*}
        R'_\mu \coloneqq 
        & \frac{1}{2} (1 \ot 1 + 1 \ot g + g \ot 1 - g \ot g 
        + \mu (x \ot x + x \ot gx - gx \ot x + gx \ot gx))
    \end{align*}
    for an arbitrary $\mu \in \Bbbk$, which can be identified with the R-matrix of $E(2)$ that corresponds to the $2 \times 2$ matrix $\bigl[\sm{\mu}{0}{0}{0}\bigr]$.
    Let $M$ be an arbitrary $E(1)$-module.
    By an entirely parallel discussion, one can show that there exists an isomorphism $\func{\KB[1]{E(1), R'_\mu}}{M^{\oplus 2}}[\varphi][(\varphi(1 \ot g), \varphi(1 \ot x))]$.
    More precisely, the matrix (\ref{eq:matrix}) with respect to the R-matrix $R'_\mu$ of $E(2)$ is written (up to scalar multiplication) as
    {\small\[\begin{bmatrix}
        4 & 2 \gbf{+} & - (2 + \gbf{+}) \mu x & 0 & 
        \cdots \\
        0 & 0 & 0 & 0 & 
        \cdots \\
        0 & 0 & 0 & 0 & 
        \cdots \\
        0 & (3 - g) x & 2 \gbf{-} & 4 & 
        \cdots \\
        \vdots & \vdots & \vdots & \vdots & \ddots
    \end{bmatrix}\raisebox{-2.5em}[0pt][0pt]{\makebox[0pt]{,}}\]
    }which gives the above isomorphism.
    In particular, this implies that $\KB[1]{E(1), R'_0} \iso \KB[1]{E(1), R'_\mu}$ for all $\mu \in \Bbbk$.
    This is a consequence of Theorem~\ref{thm:Morita} and the fact that $(E(1), R'_0)$ and $(E(1), R'_\mu)$ are braided Morita equivalent (see Andruskiewitsch--Etingof--Gelaki \cite[Example 3.4.2]{andruskiewitsch-etingof-gelaki}, for example).
\end{rem}

\section*{Acknowledgments}
The authors are grateful to Shu Minaki for his helpful comments and discussions.
The second author was supported by JSPS Grant-in-Aid for Scientific Research (C) 24K06653.




\begin{thebibliography}{99}
    \paper{andruskiewitsch-etingof-gelaki}{N. Andruskiewitsch, P. Etingof, and S. Gelaki}{Triangular Hopf algebras with the Chevalley property}{Michigan Math.\ J.}{49}{2001}{2}{277}{298}
    \bibitem{angelini-knoll-chan-gerhardt-merling-peroux}G. Angelini-Knoll, D. Chan, T. Gerhardt, M. Merling, and M. P\'eroux, Topological symmetric and braid homologies, preprint (2026). \texttt{arXiv:2605.22946v1 [math.AT]}.
    \papervv{angelini-knoll-merling-peroux}{G. Angelini-Knoll, M. Merling, and M. P\'eroux}{Topological $\Delta \mathbf{G}$-homology of rings with twisted $G$-action}{Adv.\ Math.}{493}{2026}{Paper No.~110907}
    \papervar{ardizzoni-menini}{A. Ardizzoni and C. Menini}{Adjunctions and braided objects}{J. Algebra Appl.}{13}{2014}{6}{Paper No.~1450019, 47}
    \paper{artin}{E. Artin}{Theorie der Z\"opfe}{Abh.\ Math.\ Sem.\ Univ.\ Hamburg}{4}{1925}{1}{47}{72}
    \paperv{bardakov-gongopadhyay-singh-vesnin-wu}{V. G. Bardakov, K. Gongopadhyay, M. Singh, A. Vesnin, and J. Wu}{Some problems on knots, braids, and automorphism groups}{Sib.\ \`Elektron.\ Mat.\ Izv.}{12}{2015}{394}{405}
    \paper{bardakov-neshchadim-singh}{V. G. Bardakov, M. V. Neshchadim, and M. Singh}{Exterior and symmetric (co)homology of groups}{Internat.\ J. Algebra Comput.}{30}{2020}{8}{1577}{1607}
    \book{bohm}{G. B\"ohm}{Hopf algebras and their generalizations from a category theoretical point of view}{Springer}{Cham}{2018} 
    \book{bulacu-caenepeel-panaite-van-oystaeyen}{D. Bulacu, S. Caenepeel, F. Panaite, and F. Van Oystaeyen}{Quasi-Hopf algebras}{Cambridge Univ.\ Press}{Cambridge}{2019} 
    \paper{carnovale-cuadra}{G. Carnovale and J. Cuadra}{Cocycle twisting of $E(n)$-module algebras and applications to the Brauer group}{$K$-Theory}{33}{2004}{3}{251}{276}
    \paper{etingof-gelaki-A}{P. Etingof and S. Gelaki}{On cotriangular Hopf algebras}{Amer.\ J. Math.}{123}{2001}{4}{699}{713}
    \paperv{etingof-gelaki-B}{P. Etingof and S. Gelaki}{On families of triangular Hopf algebras}{Int.\ Math.\ Res.\ Not.}{14}{2002}{757}{768}
    \book{etingof-gelaki-nikshych-ostrik}{P. Etingof, S. Gelaki, D. Nikshych, and V. Ostrik}{Tensor categories}{Amer.\ Math.\ Soc.}{Providence, RI}{2015} 
    \paper{gelaki}{S. Gelaki}{On pointed ribbon Hopf algebras}{J. Algebra}{181}{1996}{3}{760}{786}
    \paper{fiedorowicz-loday}{Z. Fiedorowicz and J.-L. Loday}{Crossed simplicial groups and their associated homology}{Trans.\ Amer.\ Math.\ Soc.}{326}{1991}{1}{57}{87}
    \paper{herzog}{I. Herzog}{The Ziegler spectrum of a locally coherent Grothendieck category}{Proc.\ London Math.\ Soc.\ (3)}{74}{1997}{3}{503}{558}
    \bibitem{johnson-yau}N. Johnson and D. Y. Yau, 2-dimensional categories, preprint (2020). \texttt{arXiv:2002.06055v3 [math.CT]}.
    \book{kassel}{C. Kassel}{Quantum groups}{Springer-Verlag}{New York}{1995} 
    \paper{krasauskas}{R. Krasauskas}{Skew-simplicial groups}{Litovsk.\ Mat.\ Sb.}{27}{1987}{1}{89}{99}
    \bibitem{lurie}J. Lurie, Kerodon (2018). Available at \url{https://kerodon.net}.
    \bookv{maclane}{S. {\maclane}}{Categories for the working mathematician}{second}{Springer-Verlag}{New York}{1998} 
    \papervar{muller-walton}{M. M\"uller and C. Walton}{Morita invariants of quasitriangular coideal subalgebras}{J. Pure Appl.\ Algebra}{230}{2026}{3}{Paper No.~108210, 17}
    \paper{ng-schauenburg}{S.-H. Ng and P. Schauenburg}{Central invariants and higher indicators for semisimple quasi-Hopf algebras}{Trans.\ Amer.\ Math.\ Soc.}{360}{2008}{4}{1839}{1860}
    \paper{pevtsova-witherspoon}{J. Pevtsova and S. Witherspoon}{Varieties for modules of quantum elementary abelian groups}{Algebr.\ Represent.\ Theory}{12}{2009}{6}{567}{595}
    \paperv{pirashvili-A}{M. Pirashvili}{Symmetric cohomology of groups}{J. Algebra}{509}{2018}{397}{418}
    \paper{pirashvili-B}{M. Pirashvili}{Crossed modules and symmetric cohomology of groups}{Homology Homotopy Appl.}{22}{2020}{2}{123}{134}
    \paperv{pirashvili-pirashvili}{M. Pirashvili and T. Pirashvili}{Symmetric cohomology of groups and Poincar\'e duality}{J. Algebra}{614}{2023}{177}{198}
    \paper{radford-A}{D. E. Radford}{Minimal quasitriangular Hopf algebras}{J. Algebra}{157}{1993}{2}{285}{315}
    \bibitem{radford-B}D. E. Radford, \textit{On Kauffman's knot invariants arising from finite-dimensional Hopf algebras}, in: J. Bergen and S. Montgomery (eds.), Advances in Hopf algebras, 
    Lecture Notes in Pure and Appl.\ Math.\ \textbf{158} (1994), 205--266.
    \book{riehl}{E. Riehl}{Categorical homotopy theory}{Cambridge Univ.\ Press}{Cambridge}{2014}
    Available at the author's website \url{https://emilyriehl.github.io/files/cathtpy.pdf}. 
    \paper{schlichtkrull-solberg}{C. Schlichtkrull and M. Solberg}{Braided injections and double loop spaces}{Trans.\ Amer.\ Math.\ Soc.}{368}{2016}{10}{7305}{7338}
    \papervar{shiba-sanada-itaba}{Y. Shiba, K. Sanada, and A. Itaba}{Symmetric cohomology and symmetric Hochschild cohomology of cocommutative Hopf algebras}{J. Algebra Appl.}{23}{2024}{13}{Paper No.~2450223, 25}
    \paper{shimizu}{K. Shimizu}{Monoidal Morita invariants for finite group algebras}{J. Algebra}{323}{2010}{2}{397}{418}
    \paper{staic-A}{M. D. Staic}{From $3$-algebras to $\Delta$-groups and symmetric cohomology}{J. Algebra}{322}{2009}{4}{1360}{1378}
    \paper{staic-B}{M. D. Staic}{Symmetric cohomology of groups in low dimension}{Arch.\ Math.\ (Basel)}{93}{2009}{3}{205}{211}
    \book{witherspoon}{S. J. Witherspoon}{Hochschild cohomology for algebras}{Amer.\ Math.\ Soc.}{Providence, RI}{2019} 
    \papervar{xu-zheng}{Y. Xu and J. Zheng}{Hopf algebras are determined by their monoidal derived categories}{Math. Z.}{312}{2026}{3}{Paper No.~69, 23}
    \bibitem{zarelua}A. V. Zarelua, \textit{Exterior homology and cohomology of finite groups} (Russian), Tr.\ Mat.\ Inst.\ Steklova \textbf{225} (1999), 202--231.
\end{thebibliography}
\end{document}